\documentclass[a4paper,11pt]{amsart}

\usepackage[utf8]{inputenc} 
\usepackage[T1]{fontenc}

\usepackage{tgtermes}
\usepackage{newtxmath}

\usepackage[english]{babel}
\usepackage {mathtools}
\usepackage{mathrsfs}
\usepackage{geometry}
\allowdisplaybreaks 

\newcommand{\RR}{\mathbb{R}}
\newcommand{\CC}{\mathbb{C}}
\newcommand{\NN}{\mathbb{N}}

\newcommand{\Rr}{\mathrm{Re}}
\newcommand{\Ii}{\mathrm{Im}}

\theoremstyle{plain}
\newtheorem{Th}{Theorem}[section]
\newtheorem{Prop}[Th]{Proposition}
\newtheorem{Lem}[Th]{Lemma}

\newtheorem{Cor}[Th]{Corollary}
\newtheorem{Rq}[Th]{Remark}

\numberwithin{equation}{section}

\newenvironment{Thp}[1]
  {\renewcommand{\theTh}{\ref{#1}$'$}%
   \addtocounter{Th}{-1}%
   \begin{Th}}
  {\end{Th}}

\newcounter{Hyp}
\newcommand{\Hyp}[1]{\stepcounter{Hyp}\smallskip\textbf{(H\theHyp)}\quad \begin{minipage}[t]{0.9\textwidth} #1 \end{minipage}\smallskip}

\begin{document}

\title[On smoothness and uniqueness of NLS multi-solitons]{On smoothness and uniqueness of multi-solitons of the non-linear Schrödinger equations}

\author{Raphaël Côte}
\address[Raphaël Côte]{Institut de Recherche Mathématique Avancée UMR 7501, Université de Strasbourg, Strasbourg, France}
\email{cote@math.unistra.fr}

\author{Xavier Friederich}
\thanks{Corresponding author: Xavier Friederich}
\address[Xavier Friederich]{Institut de Recherche Mathématique Avancée UMR 7501, Université de Strasbourg, Strasbourg, France}
\email{friederich@math.unistra.fr}

\maketitle

\let\thefootnote\relax
\let\thefootnote\relax

\begin{abstract}
 In this paper, we study some properties of multi-solitons for the non-linear Schrödinger equations in $\RR^d$ with general non-linearities. Multi-solitons have already been constructed in $H^1(\RR^d)$ in \cite{cote,schMM,merle}. We show here that multi-solitons are smooth, depending on the regularity of the non-linearity. We obtain also a result of uniqueness in some class, either when the ground states are all stable, or in the mass-critical case. 
\end{abstract}

\section{Introduction}

\subsection{Generalities on the non-linear Schrödinger equations} 

We consider non-linear Schrö\-din\-ger equations in $\RR^d$ which admit traveling solitary waves (solitons). More precisely, we focus on 
\begin{gather}\label{NLS}\tag{NLS}
\partial_t u=i\big(\Delta u+f(|u|^2)u \big),
\end{gather} 
where $u: I \times \RR^d \to \CC$, $I \subset \RR$ is a time interval, and $f:[0,+\infty) \to \RR$ is an $H^1$-subcritical non-linearity. 

For $d\leq 3$ and for particular functions $f$, equation \eqref{NLS} arises in the mathematical description of many physical phenomena; it is used mainly to model non-linear wave dynamics. For instance, it is fundamental in the description of the dynamic of particles moving in electromagnetic fields \cite{malomed} and quantum systems like Bose-Einstein condensates \cite{pitaevskii}. With particular non-linearities obtained by linear combinations of quadratic, cubic, and quintic terms it appears also when one tries to describe the propagation of laser beams in some mediums \cite{brabec} or of more general ultrashort optical pulses (see for example \cite{kkde} for the study of some solutions of these equations), with applications in medical imaging, material processing and optical communications (we refer to \cite{felice} for further details concerning the applications of \eqref{NLS} to fiber optics for example).
\begin{align} \label{def:nonlinearity}
f(r) = r^{\frac{p-1}{2}}, \quad 1< p < 1+\frac{4}{(d-2)_+},\quad r\ge 0.
\end{align}
(If $d =1$ or $2$, the condition is $p >1$ and if $d\ge 3$, the condition is $ 1< p < 1+\frac{4}{d-2}$). We will give results on general non-linearities in paragraph \ref{paragraph:gen}. \\

Ginibre and Velo \cite{ginibre} proved that \eqref{NLS} is locally well-posed in $H^1(\RR^d)$: for all $u_0\in H^1(\RR^d)$, there exist $T>0$ and a unique maximal solution $u\in\mathscr{C}([0,T),H^1(\RR^d))$ of \eqref{NLS} such that $u(0)=u_0$. For any such $H^1$ solution, the following quantities are conserved for all $t \in [0,T)$:
\begin{itemize}
\item the $L^2$ mass $\displaystyle \int_{\RR^d}|u(t,x)|^2\;dx$.
\item the energy $\displaystyle \int_{\RR^d} \left( \frac{1}{2} |\nabla u(t,x)|^2 - \frac{1}{p+1} |u(t,x)|^{p+1} \right) \;dx$.
\item the momentum $\displaystyle \Ii\int_{\RR^d}\nabla u(t,x)\overline{u}(t,x)\;dx$. 
\end{itemize}
Furthermore, for all $s \in \NN\setminus\{ 0, 1 \}$, if $$g:z\mapsto zf(|z|^2)=z|z|^{p-1}$$ is $\mathscr{C}^s$ on $\CC$ as an $\RR$-differentiable function (that is if $p> s$ or $p$ is an odd integer), and in case where $s<\frac{d}{2}$, if in addition $p<1+\frac{4}{d-2s}$ , then \eqref{NLS} is locally well-posed in $H^s(\RR^d)$ according to Kato \cite[Theorem 4.1]{kato}. \\

Also \eqref{NLS} is invariant under properties of space-time translation, phase, and galilean invariances: if $t^0 \in \RR$, $v\in\RR^d$, $x^0\in\RR^d$, $\gamma\in\RR$, and $u$ is a solution to \eqref{NLS}, then
\begin{equation} \label{def:inv}
(t,x) \mapsto u(t-t^0, x-x^0-v t) e^{i\big(\frac{1}{2}v\cdot x- \frac{|v|^2}{4} t+\gamma\big)}
\end{equation}
is also a solution to \eqref{NLS}. What is more, \eqref{NLS} with a pure power non-linearity \eqref{def:nonlinearity} is scaling invariant: if $\lambda >0$ and  $u$ is a solution to \eqref{NLS}, then
\begin{equation} \label{def:scaling}
(t,x) \mapsto \frac{1}{\lambda^{1/(p-1)}} u \left( \frac{t}{\lambda}, \frac{x}{\lambda^{1/2}} \right)
\end{equation}
is still a solution to \eqref{NLS}. 
\bigskip

Let us introduce now some particular solutions of \eqref{NLS} which are essential in the theory and on which our paper is based. Given $\omega >0$, Berestycki and Lions \cite{beres} proved the existence of a (non-vanishing) positive radial solution $Q_\omega\in H^1(\RR^d)$ to the following elliptic problem
\begin{equation}\label{elliptic}
\Delta Q_\omega+f(Q_\omega^2)Q_\omega=\omega Q_\omega, \quad Q_\omega >0
\end{equation}
(By scaling, it suffices to prove the existence for $\omega=1$). A solution to \eqref{elliptic} is called a \emph{ground state} (and if one relaxes the sign condition, we speak of \emph{bound state}). Using a Pohozaev identity \cite{pohozaev}, one can show that \eqref{elliptic} has no solution in $H^1(\RR^d)$ for $p \ge 1+ \frac{4}{(d-2)_+}$. Moreover, for $s \in \NN^*$ and if $g$ is $\mathscr C^s$ on $[0,+\infty)$, then $Q_\omega$ is $\mathscr C^{s+2}$ on $\RR^d$ and one has exponential decay (see \cite[proof of Lemma 1]{beres}): 
there exists $C_s>0$ such that for each multi-index $\delta\in\NN^d$ with $|\delta|\le s+2$, 
\begin{equation} \label{exp_decay_Q}
\forall\;x\in\RR^d,\qquad \left|\partial^\delta Q_\omega(x)\right|\le C_s e^{-\frac{\sqrt{\omega}}{2} |x|}.
\end{equation}
Then the function
\begin{equation}\label{soliton}
(t,x)\mapsto Q_\omega(x) e^{i \omega t}
\end{equation}
is a solution to \eqref{NLS}. Using the invariances \eqref{def:inv} of the equation, one obtains a whole family of solutions of \eqref{NLS} known as \emph{solitons}.

\bigskip

Dynamical properties of solitons have been extensively studied. One important result is related to their orbital stability: solitons are orbitally stable if $ p < 1 + \frac{4}{d}$ and unstable if $ p \ge 1 + \frac{4}{d}$. Recall that the case $p = 1 + \frac{4}{d}$ corresponds to the $L^2$-critical exponent: in this particular case, the $L^2$ norm of a solution is invariant by scaling \eqref{def:scaling}.

\bigskip

In this article, we are interested in qualitative properties of multi-solitons, that is solutions of \eqref{NLS} which behave as a sum of decoupled solitary waves for large times. \\

Let us begin with the definition of some further notations. Fix $K\in\NN\setminus\{0,1\}$ and for all $k=1,\dots,K$, let 
\[ \omega_k>0,  \quad \gamma_k\in\RR,  \quad x_k^0\in\RR^d,  \quad \text{and } v_k\in\RR^d \text{ such that for all } k\ne k',\quad v_k\neq v_{k'}. \]
For all $k=1,\dots,K$, we consider
\[ R_k(t,x) =  Q_{\omega_k}(x-x_k^0-v_kt)e^{i\big(\frac{1}{2}v_k\cdot x+\big(\omega_k-\frac{|v_k|^2}{4}\big)t+\gamma_k\big)}, \]
which is a soliton of \eqref{NLS} moving on the line $x=x_k^0+v_kt$. We denote also $$R:=\sum_{k=1}^KR_k.$$
In general, $R$ is obviously not a solution to \eqref{NLS} because of the non-linearity. A \emph{multi-soliton} is a solution $u$ of \eqref{NLS} defined on $[T_0,+\infty)$ for some $T_0\in\RR$ and such that 
\begin{equation}\label{def_multisol} 
\lim_{t\to+\infty}\|u(t)-R(t)\|_{H^1}=0. 
\end{equation}

Multi-solitons were explicitly constructed in the integrable case, that is with $f(x)=x$ and $d=1$, using the inverse scattering method (see Zakharov and Shabat \cite{zakharov}).

The first construction in a non-integrable context is due to Merle \cite{merle}, in the critical case $p=1+\frac{4}{d}$. 
Later, following closely the ideas of Martel in \cite{martel} for the construction of multi-solitons for the Korteweg-de Vries equations, Martel and Merle \cite{schMM} constructed multi-solitons of \eqref{NLS}, in the $L^2$-subcritical case $1<p<1+\frac{4}{d}$. This result was extended to $L^2$-supercritical exponent by Côte, Martel and Merle \cite{cote}. Let us recall the results.

\begin{Th}[Merle \cite{merle}, Martel and Merle \cite{schMM}, Côte, Martel and Merle \cite{cote}]\label{reg1}
There exist $\theta>0$ (depending on $v_k$, $\omega_k$ for $1\le k \le K$), $T_0\ge 0$, and a solution $u\in\mathscr{C}([T_0,+\infty),H^1(\RR^d))$ of \eqref{NLS} such that 
\begin{equation} \label{eq:exp_conv}
\forall\; t\ge T_0,\qquad \|u(t)-R(t)\|_{H^1}\le e^{-2\theta t}.
\end{equation}
\end{Th}

Let us also mention the works by Le Coz and Tsai \cite{LCT14} and Le Coz, Li and Tsai \cite{LCLT15} where \emph{infinite} trains of solitons are constructed, in the context of \eqref{NLS}. The construction of multi-solitons  in $H^1$ was done for many other non-linear dispersive models (besides the generalized Korteweg-de Vries equations) such as the non-linear Klein-Gordon equation \cite{CM}, the Hartree equation \cite{KMR}, the water-waves system \cite{ming}, and in both stable and unstable contexts, which means assuming that all $Q_{\omega_k}$ are stable or not. 

Even though solutions of (\ref{NLS}) behaving as a sum of  decoupled \emph{general bound states} (that is, solutions to \eqref{elliptic} which  change sign) have been  studied in the last years (see for example \cite{cc} on (\ref{NLS}) or \cite{CMa} on non-linear Klein-Gordon equation), in the present paper we concentrate only on multi-solitons based on ground states. Our goal here is to study uniqueness and smoothness issues.

\bigskip

To our knowledge, the only work where multi-solitons are shown to be more regular than $H^1$ is for the generalized Korteweg-de Vries equation (which is one-dimensional), with monomial non-linearity, by Martel \cite{martel}, where the exponential convergence \eqref{eq:exp_conv} is shown to hold in $H^s(\RR)$ for all $s \in \NN$ (with a constant $C_s$ depending on $s$ in front of the exponential term and a convergence rate $\theta$ independent of $s$): see Proposition 5 and its proof for the $L^2$-subcritical and critical cases; the $L^2$-supercritical case can be treated likewise, as it is mentioned in \cite[Remark 1]{cote}).

A natural question is thus to understand for \eqref{NLS} whether the multi-soliton $u$ in Theorem \ref{reg1} is smoother than $H^1$: for example, does it belong to $\mathscr{C}([T_0,+\infty),H^s(\RR^d))$ for $s>1$ and does it hold $\| u(t) - R(t) \|_{H^s} \to 0$ as $t \to +\infty$?

Another natural question is the uniqueness or the classification of multi-solitons. Again, to our knowledge, the only complete study of the question was done for the generalized Korteweg-de Vries equations: multi-solitons were proved to be unique in the $L^2$-subcritical and critical cases by Martel \cite{martel}, and were classified in the $L^2$-supercritical case by Combet \cite{combetgkdv} (there is a $K$-parameter family of $K$-solitons, each instability direction yielding a free parameter). Actually, smoothness of the multi-solitons constructed in Theorem \ref{reg1} is an important ingredient in the proof of uniqueness (or classification) in dimension $d\ge 2$.

\subsection{Main results} 

Our first result concerns the construction of a multi-soliton in $H^s(\RR^d)$, where the regularity index $s >1$ depends on the regularity of the function $g$. We prove in particular that the convergence occurs with an exponential rate in $H^s(\RR^d)$. The result is stated here for pure power non-linearities, and we will discuss general non-linearities in the next paragraph.

\begin{Th}[Smoothness of multi-solitons] \label{reg}
Assume that $p \ge 3$. Let $\theta>0$ be defined as in Theorem \ref{reg1} and $s_0 = \lfloor p-1 \rfloor \ge 2$, or $s_0 = +\infty$ if $p$ is an odd integer.
There exist $T_1>0$ and $u\in\mathscr{C}([T_1,+\infty),H^{s_0}(\RR^d))$ a solution of \eqref{NLS} with pure power nonlinearity \eqref{def:nonlinearity} such that for all non-negative integer $s\le s_0$, there exists $C_s \ge 1$ such that
\begin{equation}\label{resreg1}
\forall\; t\ge T_1,\quad \|u(t)-R(t)\|_{H^s} \le C_se^{-\frac{2\theta}{s+1}t}.
\end{equation}
Moreover, if $p$ is an odd integer, then for all integer $s\ge 0$, 
\begin{equation}\label{resreg2}
\forall\; t\ge T_1,\quad \|u(t)-R(t)\|_{H^s} \le \sqrt{C_s}e^{-\theta t}.
\end{equation}
\end{Th}

\begin{Rq}
Theorem \ref{reg} completes Theorem \ref{reg1} by showing the existence of smooth multi-solitons. Notice that its applications are limited to dimensions $d\leq 3$, since we consider the pure power case and due to the $H^1$-subcritical assumption $p<1+\frac{4}{(d-2)_+}$ which is required for the existence of solitons. 

In particular, in dimension $d=1$ and $d=2$, multi-solitons belong to $H^\infty(\RR^d)$ when $p$ is an odd integer, and in dimension $d=3$, multi-solitons are $H^{\infty}(\RR^3)$ when $p=3$ (which corresponds to the most physically relevant case).

The exponential decay rate $\frac{2\theta}{s+1}$ does depend on $s$ (vanishing when $s$ is large), which could be a problem for some applications. Observe however that if one is willing to consider only regularity indices $s \le \frac{s_0}{2}$ (say), then a straightforward interpolation argument between the $H^1$ and $H^{s_0}$ bounds gives the convergence with uniform exponential decay rate $\theta$:
\[ \forall\; s \le \frac{s_0}{2},\: \forall\; t\ge T_1,\quad \|u(t)-R(t)\|_{H^s} \le C_s' e^{-\theta t}. \]
\end{Rq}

Our second goal is to obtain some kind of uniqueness result for \eqref{NLS}. The simplest (and most satisfying) statement one could think of would be uniqueness in the class of solutions $u$ defined for large enough times and convergent to the profile $R$:
\[ \| u(t)-R(t)\|_{H^1} \to 0 \quad \text{as} \quad t \to +\infty. \] 
Such a uniqueness result, unconditional to any decay rate, was obtained for multi-solitons of  the generalized Korteweg-de Vries equation in \cite{martel}.
This seems currently out of reach for \eqref{NLS}, but one can expect uniqueness in some smaller class, where it is assumed that $\| u(t) - R(t) \|_{H^1}$ tends to $0$ with sufficiently fast decay rate: for example, in the class of exponential convergence, as it was done in \cite{LCT14,LCLT15} for fast spreading \eqref{NLS} multi-solitons (see also \cite{Fer21} for the case of gaussons of the logarithmic \eqref{NLS}).

In the following result, we prove uniqueness in the class of multi-solitons $u$ such that $\| u(t) - R(t) \|_{H^1}$ decreases faster than a high power of $\frac{1}{t}$ for large values of $t$. This in particular breaks the barrier of the exponential class, in which multi-solitons naturally lie. More precisely, we state the following:

\begin{Th}[Conditional uniqueness]\label{uni}
Let $d\le 2$ and $3 \le p \le 1 + \frac{4}{d}$. There exists $N \in \NN$ large such that there is a unique $u \in\mathscr{C}([T_1,+\infty),H^{1}(\RR^d))$ solution to \eqref{NLS} such that
\begin{equation}\label{classexpo}
\|u(t)-R(t)\|_{H^1}=\mathrm{O}\left(\frac{1}{t^{N}}\right), \qquad \text{as } t\rightarrow +\infty.
\end{equation}
In particular, the multi-solitons of Theorems \ref{reg1} and \ref{reg} coincide (and one can take $T_0=T_1$).
\end{Th} 

\begin{Rq}
The crucial point in Theorem \ref{uni} is obviously the uniqueness part. For pure power non-linearities, Theorem \ref{uni} provides conditional uniqueness in the sense of \eqref{classexpo}, in the $L^2$-subcritical and critical cases with $p \ge 3$ in dimension $1$, and in the $L^2$-critical case $p=3$ in dimension $2$. 

The requirement that the non-linearity be $L^2$-subcritical or $L^2$-critical is to be expected as no uniqueness holds in the $L^2$-supercritical case; see Côte and Le Coz \cite{cc} for example. 
\end{Rq}

\subsection{General non-linearities}\label{paragraph:gen}

In order to consider general non-linearities, one must make a number of assumptions which we discuss in this paragraph. 

Well-posedness in $H^1(\RR^d)$ is classically done under the hypothesis that $g:\CC\to \CC$ is $\mathscr{C}^1$ and satisfies

\Hyp{
$g(0)=0$ and there exists $p \in \left(1, 1+\frac{4}{(d-2)_+} \right)$ such that  $\frac{\partial g}{\partial x}(z), \frac{\partial g}{\partial y}(z)= \mathrm{O}(|z|^{p-1})$ as $|z| \to +\infty$.
}
In order that the Cauchy problem for \eqref{NLS} be well-posed in $H^s$ for $s\in\NN \setminus \{0,1\}$, Kato \cite[Theorem 4.1]{kato} requires furthermore that $g$ is $\mathscr{C}^s$, and if $s\leq \frac{d}{2}$, one also needs:

\Hyp{
if $g$ is a polynomial in $z$ and $\overline{z}$, its degree is $\deg g< 1+\frac{4}{d-2s}$ \\
if $g$ is not a polynomial, there exists $p \in \left[\lceil s\rceil, 1+\frac{4}{d-2s} \right[$ such that  $\frac{\partial^s g}{\partial_x^k\partial_y^{s-k}}(z) = \mathrm{O}(|z|^{p-s})$ as $|z| \to +\infty$, for all $k=0,\dots,s$, with $\lceil s\rceil$ the smallest integer greater or equal to $s$.
}
\\

The existence of solitons with frequency $\omega >0$ is not as immediate as in the pure power case. Under the assumption that
\begin{equation}\label{pointxi}
\exists\;\xi_\omega\in\RR^*_+,\qquad F(\xi_\omega)>\omega\xi_\omega \quad \text{where} \quad F(r):=\int_0^rf(\rho)\;d\rho
\end{equation} for $r \ge 0$, Berestycki and Lions \cite{beres} showed the existence of a positive radial ground state $Q_\omega\in H^1(\RR^d)$ to \eqref{elliptic}.
Note that if there exist $\tilde{\lambda}>0$, $p'>0$ and $r_0>0$ such that 
\[ \forall\; r\geq r_0,\quad f(r)\geq \tilde{\lambda} r^{p'}, \]
then (\ref{pointxi}) holds for all $\omega>0$.
If $d=1$, a necessary and sufficient condition for the existence of a positive solution \eqref{elliptic} is that $\omega$ is such that 
\begin{equation} \label{exists_GS}
r_0:=\inf\Big\{r>0\;\Big|\:\displaystyle F(r)=\omega r\Big\}
\end{equation}
 exists and $f(r_0)>\omega$ (see \cite{beres}). \\
 Let us denote by $\mathcal{O}$ a subset of $(0,+\infty)$ such that 
 \begin{equation} \text{for all } \omega \in \mathcal O, \text{ there exists a solution } Q_\omega \text{ to } \eqref{elliptic}. \end{equation}

Recall that when it exists, a positive solution of (\ref{elliptic}) is radial (see \cite[section 3]{beres} and also Gidas, Ni and Nirenberg \cite[Theorem 1']{GNN} for non-linearities $f$ such that $r\mapsto rf(r^2)$ is increasing). We underline that it is not unique in general. Indeed, Dàvila, Pino and Guerra \cite{davila} showed the existence of at least three positive $H^1$ solutions of 
\[ \Delta u+u^p+\lambda u^2=u \]
for some $\lambda >0$ and $p \in (1,5)$ in dimension $d=3$. See \cite{davila} for other counterexamples in dimension 3. \\
On the other side, Kwong \cite{kwong} showed uniqueness of a positive radial ground state in the pure power case, and one can extend this to more general non-linearities; we refer to Mc Leod and Serrin \cite{leod},  Serrin and Tang \cite{serrin} and  Jang \cite{jang} for full details. One of the most important statements may be found in Serrin and Tang \cite{serrin}: a sufficient condition for uniqueness when $d\geq 3$ is the existence of $\alpha>0$ such that
\[ \begin{dcases}
\forall\;x\in(0,\alpha],\quad f(x)\leq 1, \quad \text{ and }\quad \forall\;x\in(\alpha,+\infty),\quad f(x)>1\\
x\mapsto \frac{xf'(x)}{f(x)-1}\quad\text{is not increasing on } (\alpha,+\infty).
\end{dcases} \]
In \cite{jang}, a slightly more general condition (inspired by \cite{serrin}) yields uniqueness for \eqref{elliptic} in any dimension $d\ge 2$.

Let us point out that conditions for existence and uniquenes of a ground state have been discussed for specific non-linearities in the litterature. For example, Berestycki and Lions condition concerning existence and Serrin and Tang condition concerning uniqueness of a ground state apply to the (important) cubic-quintic non-linearity (corresponding to $g(z)=z|z|^2-z|z|^4$ or $f(r)=r-r^2$). Killip, Oh, Pocovnicu and Visan studied more precisely the properties of ground states associated with this nonlinearity and showed in particular that existence and uniqueness of a positive radially symmetric solution to
$$\Delta Q_{\omega}+Q_{\omega}^3-Q_{\omega}^5=\omega Q_{\omega} $$
hold if and only if $\omega\in\mathcal{O}:=\left(0,\frac{3}{16}\right)$; see \cite[Lemma 2.1 and Theorem 2.2]{killip}. 

\bigskip

Pursuing with general non-linearities, we will also need a number of assumptions on the linearized operators around solitons. Fix  $\omega \in \mathcal O$, and let
\[ \begin{array}{cccl}
\mathscr{L}_\omega : &H^1(\RR^d,\CC)& \to &H^1(\RR^d,\CC)\\
&v = v_1 + iv_2 & \mapsto &- \Delta v + \omega v - (f(Q_\omega^2) v+ 2 Q_\omega^2 f'(Q^2_\omega) v_1) 
\end{array} \]
so that the linearized equation of \eqref{NLS} around $e^{i \omega t}(Q+v)$ is $\partial_t v = i\mathscr{L}_\omega v$. We also define
the linearized energy around $Q_\omega$, for any $w = w_1+iw_2 \in H^1(\RR^d,\CC)$
\begin{align*}
H(w) :& =  \int_{\RR^d} \left( |\nabla w|^2 + \omega |w|^2 - \left( f(Q_\omega^2) |w|^2 + 2 Q_\omega^2 f'(Q^2_\omega) w_1^2 \right) \right)\;dx \\
& = \Rr\int_{\RR^d} \mathscr{L}_\omega w \overline{w}\; dx =  \int_{\RR^d} L_{+,\omega} w_1  w_1\;dx +  \int_{\RR^d} L_{-,\omega} w_2  w_2\; dx,
\end{align*}
where
\begin{align*}
L_{+,\omega}w_1 & :=-\Delta w_1+\omega w_1-\big(f(Q_\omega^2)+2Q_\omega^2f'(Q_\omega^2)\big)w_1\\
L_{-,\omega}w_2 & :=-\Delta w_2+\omega w_2-f(Q_\omega^2)w_2.
\end{align*}

We do two (mutually incompatible) coercivity assumptions, depending on whether $Q_\omega$ is stable or not. They write as follows:

\Hyp{ (Stable case) There exists $\mu_+ >0$ such that for all $w = w_1 + iw_2 \in  H^1(\RR^d,\CC)$
\begin{equation} \label{coer_stable} 
H(w) \ge \mu_+ \| w \|_{H^1}^2 - \frac{1}{\mu_+} \left( \left( \int_{\RR^d} w_1 Q_\omega\;dx \right)^2 + \sum_{i=1}^d \left( \int_{\RR^d} w_1 \partial_{x_i} Q_\omega\;dx \right)^2 + \left( \int_{\RR^d} w_2 Q_\omega\;dx \right)^2 \right).
\end{equation}
}

\Hyp{ (Unstable case) There exists an eigenfunction $Y_\omega  = Y_1 + i Y_2 \in  H^1(\RR^d,\CC)$ of $i\mathscr L_\omega$ (with eigenvalue $e_0 >0$) and $\mu_+ >0$ such that for all $w = w_1 + iw_2 \in  H^1(\RR^d,\CC)$,
\begin{multline} \label{coer_unstable} 
H(w) \ge \mu_+ \| w \|_{H^1}^2 - \frac{1}{\mu_+} \left( \int_{\RR^d} w_1 Y_2 dx \right)^2 \\
-  \frac{1}{\mu_+}\left( \sum_{i=1}^d \left( \int_{\RR^d} w_1 \partial_{x_i} Q_\omega\;dx \right)^2 + \left( \int_{\RR^d} w_2 Y_1\;dx \right)^2 + \left( \int_{\RR^d} w_2 Q_\omega\;dx \right)^2 \right).
\end{multline}
}
Assumptions (H3) and (H4) are intimately related to the stability of $Q_\omega$. Regarding the stable case, we have the following result by Grillakis, Shatah and Strauss \cite[p. 341-345]{gss2} (see also the work by Weinstein \cite{weinsteinmod,weinstein} and by Maris \cite[Lemma 2.4]{maris}).

\begin{Prop}
Assume that $\mathcal O$ is open and that the map $\omega \mapsto Q_\omega$ is of class $\mathscr C^1$. Let $\omega_0\in\mathcal{O}$. Under the non-degeneracy assumption that
\begin{equation}\label{nondeg}
\mathrm{Ker}(L_{+,\omega_0})= \mathrm{Span}\left\{\frac{\partial Q_{\omega_0}}{\partial x_i},\;i=1,\dots,d\right\}, 
\end{equation}
we have the following dichotomy:
\begin{itemize}
\item If $\displaystyle\frac{d}{d\omega}_{|\omega=\omega_0}\int_{\RR^d}Q_{\omega}(x)^2\;dx>0$, then \eqref{coer_stable} holds, and as a consequence, $Q_{\omega_0}$ is orbitally stable in $H^1(\RR^d)$.
\item If $\displaystyle\frac{d}{d\omega}_{|\omega=\omega_0}\int_{\RR^d}Q_{\omega}(x)^2\;dx<0$, then $Q_{\omega_0}$ is orbitally unstable in $H^1(\RR^d)$.
\end{itemize}
\end{Prop}

We also refer to Cazenave and Lions \cite[Theorem II.2 and Remark II.3]{caze} for another approach to $H^1$ orbital stability of the solitons based on $Q_{\omega_0}$.

For the pure power case, $\frac{d}{d\omega}\int_{\RR^d}Q_{\omega}(x)^2\;dx = \left(\frac{2}{p-1} - \frac{d}{2}\right) \omega^{\frac{-p+3}{p-1} - \frac{d}{2}}\int_{\RR^d}Q_1(x)^2\;dx$ so that it is positive when $1<p<1+\frac{4}{d}$, that is in the $L^2$-subcritical case (and in particular (H3) holds in that case) and it is negative when $1+\frac{4}{d}<p<\frac{d+2}{d-2}$, that is in the $L^2$-supercritical case.

\bigskip

Regarding the unstable case, following the ideas of Duyckaerts and Merle \cite{dm}, Duyckaerts and Roudenko \cite{dr}, and Côte, Martel and Merle \cite{cote}, the coercivity result below holds.

\begin{Prop}[(3.6) in \cite{cote}]
Let $\omega\in\mathcal{O}$ such that $i\mathscr L_\omega$ admits a non zero eigenfunction $Y_\omega \in H^1(\RR^d)$. Then \eqref{coer_unstable} holds.
\end{Prop}

An important step is therefore the construction of an eigenfunction $Y_\omega$: this can be done in the $L^2$-supercritical pure power case  ($p > 1+ \frac{4}{d}$)  for all $\omega>0$, and so (H4) holds in that case.

\bigskip

We are now in a position to state our results for general non-linearities. For smoothness, it reads as follows.

\begin{Thp}{reg}\label{th_reg_f}
Let $s_0>\frac{d}{2}$. Assume that $g$  satisfies (H1) and belongs to $W^{s_0+1,\infty}_{loc}(\CC)$. Assume moreover that for all $k=1, \dots, K$, $\omega_k$ belongs to $\mathcal{O}$ and $Q_{\omega_k}$ satisfies either (H3) or (H4). \\
Then the conclusions of Theorem \ref{reg} hold.
\end{Thp}

And below is about uniqueness.

\begin{Thp}{uni}\label{th_uni_f}
Let $d\le 3$ and $\tilde{f}:z\mapsto f(|z|^2)$ be of class $\mathscr C^2$ on $\CC$ (as an $\RR$-differentiable function), such that its second differential satisfies
\begin{equation}\label{prop_croissance}
\left\|D^2_{z}\tilde{f}\right\|=\mathrm{O}\left(|z|^{\frac{4}{d}-2}\right), \qquad\text{as } |z|\rightarrow +\infty.
\end{equation}
If $f$ is not the pure power non-linearity, assume that for all $k=1, \dots, K$, $\omega_k \in \mathcal O$ and $Q_{\omega_k}$ satisfies (H3), and in the case where $d\ge 2$, assume moreover that $g$ belongs to $W^{s_0+1,\infty}_{loc}(\CC)$, where $s_0:=\lfloor \frac{d}{2}\rfloor+1$. \\
Then the conclusion of Theorem \ref{uni} holds.
\end{Thp} 

\begin{Rq}
Theorems \ref{uni} and \ref{th_uni_f} are restricted to dimensions $d \le 6$. For $d\ge 7$, a similar uniqueness result can be proved (using the same method as that we develop in section \ref{uniq}), provided a smaller class of multi-solitons $u$ is considered, and for which a bound on $\| u(t) -R(t) \|_{L^\infty}$ is furthermore assumed. This is the purpose of the next proposition.
\end{Rq}

\begin{Prop}\label{uni:4d}
Let $d\ge 4$, $s_0:=\big\lfloor\frac{d}{2}\big\rfloor+1$, and $\tilde{f}:z\mapsto f(|z|^2)$ be of class $\mathscr C^2$ on $\CC$ (as an $\RR$-differentiable function), such that its second differential satisfies
\begin{equation}\label{prop_croissance}
\left\|D^2_{z}\tilde{f}\right\|=\mathrm{O}\left(|z|^{\frac{4}{d}-2}\right), \qquad\text{as } |z|\rightarrow +\infty.
\end{equation}
Assume that $g$ belongs to $W^{s_0+1,\infty}_{loc}(\CC)$. Assume moreover that for all $k=1, \dots, K$, $\omega_k \in \mathcal O$ and $Q_{\omega_k}$ satisfies (H3).\\
Then for any $\alpha >0$, there exists $N\in\NN^*$ such that there exists a unique $u\in\mathscr{C}([T_1,+\infty),H^{1}(\RR^d)\cap L^\infty(\RR^d))$ solution to \eqref{NLS} such that \[
\|u(t)-R(t)\|_{H^1}=\mathrm{O}\left(\frac{1}{t^N}\right)\quad\text{ and }\quad
\displaystyle\int_t^{+\infty}\|u(s)-R(s)\|_{L^\infty}\;ds=\mathrm{O}\left(\frac{1}{t^\alpha}\right),
\qquad\text{as } t\rightarrow+\infty.\]
\end{Prop}

\subsection{Outline of the paper and the proofs}\label{outline}

\subsubsection{The main content}

We will prove Theorems \ref{th_reg_f} and \ref{th_uni_f} which generalize Theorems \ref{reg} and \ref{uni} respectively when applied to pure power non-linearities.

Section 2 is devoted to the proof of our regularity result, that is Theorem \ref{th_reg_f}. We start from a well-chosen sequence  $(u_n)$ of solutions satisfying uniform $H^1$ estimates and which were constructed in  \cite{cote} and \cite{schMM} (we emphasize that we do not work with the already built multi-soliton in $H^1(\RR^d)$ given in Theorem \ref{reg1}). Taking some inspiration from  Martel \cite[section 3]{martel} in the context of the generalized Korteweg-de Vries equations, we prove $H^s$ uniform estimates for $(u_n)$ via an induction on the index of regularity. We can combine both stable and unstable cases since we start from the same uniform exponential $H^1$ estimates obtained in \cite{cote} and \cite{schMM}. From these $H^s$ estimates we deduce (by a usual compactness argument) the existence of a multi-soliton satisfying the conclusions of Theorem \ref{th_reg_f}. 

The induction argument relies on the study of a functional related to $\|u_n\|_{H^s}^2$, suitably modified so as to cancel ill-behaved terms; this functional takes the same form in all dimensions (see (\ref{funct}) in subsection \ref{parag_uniform_Hs0}). In \cite{martel}, for gKdV, from $s=3$, all quantities of the type $\|u_n(t)-R(t)\|_{H^s}$ introduced are shown to decrease exponentially in large time with the same rate. Our proof is more technical, insofar as the algebra is not as favorable. In the context of \eqref{NLS}, the terms involving real and imaginary parts can not be treated in the same way at once, and in dimension $d \ge 2$, derivative can fall on terms in many various ways. As nonlinearities are not necessarily smooth (as it is the case in \cite{martel}), we need to count carefully the number of times one can perform an integration by parts. This explains why in the case of  \eqref{NLS}, the rate of the exponential decay on $\|u_n(t)-R(t)\|_{H^s}$ is halved when passing from $s$ to $s+1$ (see \eqref{def:theta_s} and Proposition \ref{core}). We then obtain the decay rate of \eqref{resreg1} by a simple interpolation argument.

\bigskip

Regarding the regularity assumption on the nonlinearity, recall that the $H^1$ estimate in Theorem \ref{reg1} holds when $g: z \mapsto zf(|z|^2)$ is of class $\mathscr{C}^1$. As fas as $H^{s_0}$ regularity is concerned, Kato's well-posedness result \cite{kato} in $H^{s_0}(\RR^d)$ assumes $g$ of class $\mathscr{C}^{s_0}$. In Theorem \ref{th_reg_f}, we require a bit more regularity for $g$ to prove an $H^{s_0}$ estimate for $s_0\ge 2$ for the multi-soliton.  From a technical point of view, many estimates rely indeed on the local boundedness of the derivatives (in the sense of distributions) of the functions $\frac{\partial^{s} g}{\partial x^{s'}\partial y^{s-s'}}$, where $s=0,\dots,s_0$ and $s'=0,\dots,s$. The preceding property is typically used at two levels. First, we need the local Lipschitz condition which is satisfied by functions in $W^{1,\infty}_{loc}$: this is for example the case for \eqref{est_loc_lipschitz} in subsection \ref{parag_uniform_Hs0}. In order to obtain the desired $H^{s_0}$ estimate, we need also to integrate by parts a particular term (at least one time) which contains derivatives with respect to the space variable of maximal order $s_0$ of both $u_n-R$ and $g$ in order that $u_n-R$ appears with a derivative of order $s_0-1$, thus can be controlled (see in particular (\ref{est_ipp}) in subsection \ref{parag_uniform_Hs0}). For this, one shall ensure that the distributional derivative of $x_i\mapsto \frac{\partial^{s_0} g}{\partial x^{s}\partial y^{s_0-s}}(R_k(x))$ belong to some Lebesgue space $L^q$; this is in fact the case if the derivative of $x_i\mapsto \frac{\partial^{s_0} g}{\partial x^{s}\partial y^{s_0-s}}$ is bounded on a certain disk centered at the origin. Therefore, we assume that $g$ is an element of $W^{s_0+1,\infty}_{loc}(\CC)$. Notice that this condition is met when $f$ is the pure power non-linearity \eqref{def:nonlinearity} with $s_0=\lfloor p-1\rfloor$ (and also in the particular case when $p$ is an even integer).\\

Besides, we emphasize that assumption $s_0>\frac{d}{2}$ in Theorems \ref{reg} and \ref{th_reg_f} (which is automatically satisfied for $d\in\{1,2,3\}$) seems to be needed to obtain the desired estimates, judging from \eqref{condition_s0}. In order to relax this, one should  work out an argument involving Strichartz type estimates. But to be effective, the dispersive estimates are to be done on the linearized equation around a sum of solitons, that is a sum of potentials which are decoupled and smooth, but large and not decaying in time. Such estimates would actually be very useful for other purposes, for example the stability of multi-solitons. To our knowledge, they are however not (yet) available.

\bigskip

Section 3 is devoted to the proof of the uniqueness result, which combines some ideas of \cite{martel} and of \cite{tsai}.We will consider a solution satisfying (\ref{classexpo}) and \eqref{prop_croissance} and show that it is in fact the multi-soliton constructed in the proof of Theorem \ref{th_reg_f}: we therefore study the difference of these two solutions and show that it is 0. One main tool for this is a Weinstein type functional, which is coercive provided we assume some adequate orthogonality properties. Depending on the stable or $L^2$-critical case considered, these orthogonality conditions differ. The coercivity result available in the latter case (where $f:r\mapsto r^{\frac{2}{d}}$) is the object of Proposition \ref{coerccritic}. The fact that we do the difference with an already constructed multi-soliton which is sufficiently regular is crucial, at least up to dimension 2. In fact, what we truly need is the $H^2(\RR^d)$ decay for $d\ge 2$, and also at several times, that the constructed multi-soliton takes values in $L^\infty(\RR^d)$.

Note also that, finding like us his inspiration in \cite{martel}, Combet \cite{combet}, in the one-dimensional pure power unstable case, has already obtained estimates similar to those we develop in section \ref{uniq} for general $f$ in the stable case. Last, the lack of (backward in time) monotonicity properties of \eqref{NLS} explains somehow the difficulty to obtain unconditional uniqueness, that is to prove uniqueness in the whole class of multi-solitons in the sense of \eqref{def_multisol} (without decay rate); see Remark \ref{rq:mono} for more details.

\subsubsection{Some notations and writing practices used through the text}

Solutions of \eqref{NLS} or functions constructed with such solutions take values in $\CC$. As usual, the modulus of a complex number will be denoted by $|.|$. \\
Our computations are generally done in all dimensions $d$. To that extent, 
\begin{itemize}
\item for any vector $u\in\RR^d$, we denote by $\Rr(u)$ (respectively $\Ii(u)$) the vector of $\RR^d$ which components are the real parts (respectively the imaginary parts) of the components of $u$.
\item $\cdot$ denotes the euclidean scalar product in $\RR^d$ and $|.|$ denotes also the euclidean norm from which it derives.
\item we use the usual notation for multi-indices.
\end{itemize}

As usual, it is also convenient to denote by $C$ some positive constant which can change from one line to the next but which is always independent of the index of any sequence considered. 

The main functional spaces we will work with are the Sobolev spaces $H^s(\RR^d)$ for $s\in\NN^*$ endowed with the usual norms defined by:
\[\forall\;w\in H^s(\RR^d),\quad \|w\|_{H^s}:=\left(\sum_{|\alpha|\le s}\|\partial^\alpha w\|_{L^2}^2\right)^{\frac{1}{2}}.\]
We consider also $H^\infty(\RR^d):=\displaystyle \bigcap_{s\in\NN^*} H^s(\RR^d)$ and the Sobolev spaces $W^{s,\infty}_{loc}(\CC)$ (identified with $W^{s,\infty}_{loc}(\RR^2)$) for $s\in\NN^*$.

Furthermore, many computations are presented formally for ease of reading but can be justified by standard regularization arguments which often involve the local well-posedness of \eqref{NLS} in $H^s$ with continuous dependence on compact sets of time (see \cite[Theorem 1.6]{dai}).

\subsubsection{Acknowledgment}

We would like to thank Rémi Carles for his remarks which improved the quality of the manuscript.

\section{Existence of smooth multi-solitons of \eqref{NLS}}\label{preuveex}

In this section, let us concentrate on the proof of Theorem \ref{th_reg_f}. Let $s_0>\frac{d}{2}$ be an integer and assume that $g:z\mapsto zf(|z|^2)$ is in $W^{s_0+1,\infty}_{loc}(\CC)$ and satisfies (H1).

\subsection{\underline{Step 1}: Uniform $H^1$-estimate for a sequence of solutions}

In order to prove Theorem \ref{reg}, we start from the following proposition, which applies to both stable and unstable cases, and which has already been established in preceding papers. This proposition gives rise to some control in the $H^1$ norm on a constructed sequence of solutions of \eqref{NLS} which turns out to be relevant to achieve our goal.

\begin{Prop}[Martel and Merle \cite{schMM}, Côte, Martel and Merle \cite{cote}]\label{ptdep}
There exist an increasing sequence $(S_n)_{n\in\NN}$ of times such that $S_n  \to +\infty$, a sequence $(\phi_n)\in \big(H^{s_0}(\RR^d)\big)^{\NN}$, and constants $\theta>0$, $C_1>0$, $T_0>0$ with $S_0>T_0$ such that for all $n\in\NN$:
\begin{itemize}
\item $\|\phi_n\|_{H^{s_0}}\le C_1e^{-2\theta S_n}$
\item the maximal solution $u_n$ of \eqref{NLS} such that \[u_n(S_n)=R(S_n)+\phi_n\] belongs to $\mathscr{C}([T_0,S_n],H^{s_0}(\RR^d))$ and satisfies 
$$
\forall\:t\in [T_0,S_n],\qquad\|u_n(t)-R(t)\|_{H^1}\le C_1e^{-2\theta t}.
$$
\end{itemize}
\end{Prop}

\begin{Rq}
Note that the sequence $(\phi_n)$ can be chosen in the following form.
\begin{itemize}
\item For the stable case, we take $\phi_n=0$ for all $n$ (see \cite{schMM}).
\item For the unstable case, we take $\phi_n=i\displaystyle\sum_{k\in\{1,\dots,K\},\pm}b_{n,k}^{\pm}Y_{k}^{\pm}(S_n)$ for all $n$ with $Y_k^{\pm}$ defined by
\begin{equation}\label{eigen}
Y_k^{\pm}: (t,x)\mapsto Y_{\omega_k}^{\pm}\big(x-v_kt-x_k^0\big)e^{i\big(\frac{1}{2}v_k\cdot x+\big(\omega_k-\frac{|v_k|^2}{4}\big)t+\gamma_k\big)}
\end{equation}
and with $b_n=\big(b_{n,k}^{\pm}\big)_{k\in\{1,\dots,K\},\pm}\in\RR^{2K}$ well chosen (see \cite{cote} for full details) so that
\begin{equation}\label{as2}
\forall\; n\in\NN,\quad|b_n|\le e^{-2\theta S_n}.
\end{equation}
\end{itemize}
\end{Rq}

\noindent Some particular estimates will be useful throughout the proof. Firstly we retain \begin{equation}\label{init}
\|u_n(S_n)-R(S_n)\|_{H^{s_0}}\le C_1e^{-2\theta S_n}\\
\end{equation}
(since (\ref{as2}) holds and the quantities $\|Y_{k}^{\pm}(t)\|_{H^s}$ are independent of $t$). We emphasize also that
\begin{equation}\label{est}
\forall\; n\in\NN,\:\forall\;t\in [T_0,S_n],\quad
\|u_n(t)-R(t)\|_{H^1}\le C_1e^{-2\theta t}.
\end{equation}
In addition, the exponential decay property (\ref{exp_decay_Q}) of the ground states $Q_{\omega_k}$ and their derivatives lead to the following assertion, which is also crucial to establish many estimates:
\begin{equation}\label{decay_prop}
\forall\;t\in [T_0,+\infty),\:\forall\; k\neq k',\:\forall\; |\alpha_1|,|\alpha_2|\in \{0,\dots,s_0+2\},\quad\displaystyle\int_{\RR^d}\big|\partial^{\alpha_1}R_k.\partial^{\alpha_2}R_{k'}\big|(t)\;dx\le Ce^{-2\theta t}.
\end{equation}

\subsection[Step 2: Proof of uniform $H^{s}$-estimates for $u_n-R$, $s=1,\dots,s_0$]{\underline{Step 2}: Proof of uniform $H^{s}$-estimates for $u_n-R$, $s=1,\dots,s_0$}\label{parag_uniform_Hs0}

From now on, let $$v_n:=u_n-R.$$ 

\subsubsection{\underline{Step 2.1}: Performance of preliminary uniform $H^{s}$-estimates}

Define $\theta_0:=2\theta$, $\theta_1:=\theta$, and for all $s\ge 2$, 
\begin{equation} \label{def:theta_s}
\theta_s:=\min\left\{\frac{\theta_{s-1}}{2},\frac{2\theta}{d+1}\right\},
\end{equation}
so that $\theta_s=\frac{\theta}{2^{s-2}(d+1)}$ for all $s\ge 2$.
We prove the following statement, which is the core of our main existence result.
\begin{Prop}\label{core} There exists $T_1\ge T_0$ such that for all $s\in\{1,\dots,s_0\}$, there exists $C_s\ge 0$ such that for all $n\in\NN$, if $S_n\ge T_1$ then
\begin{equation}\label{core_bis}
\forall\; t\in [T_1,S_n],\qquad\|v_n(t)\|_{H^s}\le C_se^{-\theta_{s}t}.
\end{equation}
\end{Prop}

To prove Proposition \ref{core}, we resort to a "bootstrap" argument. Recall that for all $s\in\NN^*$, there exists $\mu_s\ge 0$ such that \begin{equation}
\forall\;t\in\RR,\qquad \|Y_{k}^\pm(t)\|_{H^s}\le \mu_s.
\end{equation}
For all $n$, set
\[S_{n}^*:=\inf\big\{t\ge T_0\;|\:\forall\;\tau\in[t,S_n],\:\|v_n(\tau)\|_{H^{s_0}}\le A_{s_0}\big\},\] for some constant $A_{s_0}>2K\mu_{s_0}$. Note that $S_{n}^*$ indeed exists since $v_n(S_n)=i\displaystyle\sum_{k=1}^Kb_{n,k}^\pm Y_{k}^\pm(S_n)$. Hence for all $n\in\NN$, we have 
\[\forall\;t\in (S_n^*,S_n],\qquad\|v_n(t)\|_{H^{s_0}}\le A_{s_0}.\] Due to the continuity of $v_n:[T_0,S_n]\rightarrow H^{s_0}(\RR^d)$ in $S_n^*$, we also have for all $n\in\NN$:
\begin{equation}\label{estimate_a_priori}
\forall\;t\in[S_n^*,S_n],\quad\|v_n(t)\|_{H^{s_0}}\le A_{s_0}.
\end{equation}

\noindent We will show that $S_n^*$ can be chosen independently of $n$ and improve the preceding estimate by showing first:

\begin{Prop}\label{boot}
For all $n\in\NN$, for all $s\in\{1,\dots,s_0\}$, for all $t\in[S_n^*,S_n]$, 
$$\|v_n(t)\|_{H^{s}}\le Ce^{-\theta_{s}t}.$$
\end{Prop}

\begin{proof}
We argue by induction. The existence of $C_1\ge 0$ such that for all $n\in\NN$, 
$$\forall\;t\in[S_n^*,S_n],\qquad \|v_n(t)\|_{H^1}\le C_1e^{-\theta t}$$ is already known. 
Assume that for some $s\in\{2,\dots,s_0\}$, for all $s'\in\{1,\dots,s-1\}$, there exists $C_{s'}\ge 0$ such that for all $n\in\NN$, 
\begin{equation}\label{norms_v}
\forall\;t\in[S_n^*,S_n],\qquad\|v_n(t)\|_{H^{s'}}\le C_{s'}e^{-\theta_{s'}t}.
\end{equation} We aim at showing that the same estimate is valid for $s'=s$. For this purpose, let us consider for all $n\in\NN$ the functional 
\begin{equation}\label{funct}
G_{n,s}:t\mapsto \int_{\RR^d}\left\{\sum_{|\alpha|=s}\dbinom{s}{\alpha}|\partial^\alpha u_n|^2-\sum_{|\beta|=s-1}\dbinom{s-1}{\beta}\Rr\left(u_n^2\left(\partial^\beta\overline{u_n}\right)^2\right)f'(|u_n|^2)\right\}(t)\;dx.
\end{equation}

More precisely we prove, in what follows, how to obtain the following statement, which is essential in the proof of estimate \eqref{norms_v} corresponding to $s'=s$.

\begin{Lem}\label{der_funct} For all $n\in\NN$, and for all $t\in[S_n^*,S_n]$, we have
\begin{equation}\label{diff}
|G_{n,s}(t)-G_{n,s}(S_n)|\le Ce^{-\min\left\{\theta_{s-1},\frac{4\theta}{d+1}\right\}t},
\end{equation} for some constant $C$ independent of $n$, $t$ and $A_{s_0}$.
\end{Lem}

\begin{Rq}\label{rq_choix_funct}
The fundamental reason why it is worth introducing the functional $G_{n,s}$ is that no quadratic term involving $\partial^{\alpha}u_n$ for $|\alpha|=s$ appears in its first derivative and no term $\partial^{\alpha'}u_n$ with $|\alpha'|> s$ appears either. Thus we manage to control $G_{n,s}'(t)$. Nevertheless, we do not claim that the functional $G_{n,s}$ is the only one that can be used to prove \eqref{norms_v}. 
\end{Rq}

\begin{proof}[Proof of Lemma \ref{der_funct}]
We will work on the derivative of $G_{n,s}$ and show in fact that
\begin{equation}\label{der1}
\left|G_{n,s}'(t)\right|\le Ce^{-\min\left\{\theta_{s-1},\frac{4\theta}{d+1}\right\}t}.
\end{equation}

The computations and estimates are established rather in terms of the function $g$ instead of $f$; by this means, they are considerably less burdensome. Besides, in accordance with Remark \ref{rq_choix_funct2} below, the calculations indicate that \eqref{der1} would still be true for more generalized functions $g$ which satisfy $\Ii\;\partial_xg=\Rr\;\partial_yg$. \\

Let us introduce also some further notations. For ease of reading, we will write $u$ instead of $u_n$ or $u_n(t)$ and $G_s$ instead of $G_{n,s}$. In addition, we denote by $u_1$ the real part of $u$ and by $u_2$ its imaginary part.\\ Moreover, for all $j\in\{1,\dots,d\}$, we specify as $e_j$ the $d$-tuple $(0,\dots,1,\dots,0)$ for which all components except the $j$-th one are zero. \\

We divide the proof of Lemma \ref{der_funct} into three steps. \\

\underline{Step 1}: Computation of the derivative of $G_s$. \\

First of all, observe that for all $(x,y)\in\RR^2$
\begin{equation}
\begin{aligned}
\partial_xg(x,y)&=2x(x+iy)f'(x^2+y^2)+f(x^2+y^2)\\
\partial_yg(x,y)&=2y(x+iy)f'(x^2+y^2)+if(x^2+y^2),
\end{aligned}
\end{equation}
so that $\left(\partial_xg+i\partial_yg\right)(z)=2z^2f'(|z|^2)$. In particular, $G_s$ can be rewritten in terms of $g$ as follows: 
$$G_s(t)=\int_{\RR^d}\left\{\sum_{|\alpha|=s}\dbinom{s}{\alpha}|\partial^\alpha u|^2-\frac{1}{2}\sum_{|\beta|=s-1}\dbinom{s-1}{\beta}\Rr\left(\left(\partial_xg(u)+i\partial_yg(u)\right)\left(\partial^\beta\overline{u}\right)^2\right)\right\}(t)\;dx.$$

Let $\alpha=(\alpha_1,\dots,\alpha_d)\in\NN^d$ be such that $|\alpha|=s$. There exists $l(\alpha)\in\{1,\dots,d\}$ such that $\alpha_{l(\alpha)}\ge 1$.
Then, using the fact that $u$ satisfies \eqref{NLS}, the following holds true:
\begin{equation}\label{der_alpha}
\begin{aligned}
\displaystyle\frac{d}{dt}\int_{\RR^d}|\partial^\alpha u|^2\;dx=&-2\;\Ii\displaystyle\int_{\RR^d}\left[\sum_{i=1}^d\partial_{x_i}^{2}\partial^\alpha u+\partial^\alpha\left(g(u)\right)\right]\partial^\alpha\overline{u}\;dx\\
=&-2\;\displaystyle\Ii\int_{\RR^d}\partial^\alpha\left(g(u)\right)\partial^\alpha\overline{u}\;dx\\
=&-2\;\displaystyle\Ii\int_{\RR^d}\left(\partial_xg(u)\Rr\left(\partial^\alpha u\right)+\partial_yg(u)\Ii\left(\partial^\alpha u\right)\right)\partial^\alpha\overline{u}\;dx+I_\alpha\\
=&-2\displaystyle\int_{\RR^d} \Ii\big(u^2(\partial^\alpha\overline{u})^2\big)f'(|u|^2)\;dx+I_\alpha,
\end{aligned}
\end{equation}
where $$I_\alpha=-2\;\Ii\int_{\RR^d}\left(\partial^{\alpha-e_{l(\alpha)}}\left(\partial_xg(u)\right)\partial^{e_{l(\alpha)}}u_1+\partial^{\alpha-e_{l(\alpha)}}\left(\partial_yg(u)\right)\partial^{e_{l(\alpha)}}u_2\right)\partial^\alpha\overline{u}\;dx.$$  Due to Faà di Bruno formula, $I_\alpha$ is also a linear combination of the following terms:
\begin{equation}\label{terme_type}
I_{\alpha,q,r,\tilde\alpha_1,\dots,\tilde\alpha_q}(u):=\int_{\RR^d}\partial^{\tilde\alpha_1}u_{j_1}\dots\partial^{\tilde\alpha_q}u_{j_q}\Ii\left(\frac{\partial^q g}{\partial_x^{r}\partial_y^{q-r}}(u)\partial^\alpha\overline{u}\right)\;dx,
\end{equation}
where $q\in\{2,\dots,s\}$, $r\in\{0,\dots,q\}$, $\sum_{i=1}^q|\tilde\alpha_i|=s$, and for all $i\in\{1,\dots,q\}$, $\tilde\alpha_i\ge 1$ and $j_i\in\{1,2\}$.\\

\noindent Similarly, we have for each multi-index $\beta$ such that $|\beta|=s-1$:
\begin{equation}\label{der_beta}
\begin{aligned}
\MoveEqLeft[4]\displaystyle\frac{d}{dt}\int_{\RR^d}\frac{1}{2}\Rr\left(\left(\partial_xg(u)+i\partial_yg(u)\right)\partial^\beta\overline{u}^2\right)dx\\
=&\;\Ii\displaystyle\int_{\RR^d}\left(\partial_xg(u)+i\partial_yg(u)\right)\left(\sum_{j=1}^d\partial^{\beta+2e_j}\overline{u}+\overline{\partial^\beta(g(u))}\right)\partial^\beta\overline{u}\;dx+J_{1,\beta}\\
=&-\displaystyle\sum_{j=1}^d\Ii\int_{\RR^d}\left(\partial_xg(u)+i\partial_yg(u)\right)\left(\partial^{\beta+e_j}\overline{u}\right)^2\;dx+J_{1,\beta}+J_{2,\beta}+J_{3,\beta}\\
=&-2\displaystyle\sum_{j=1}^d\int_{\RR^d} \Ii\big(u^2(\partial^{\beta+e_j}\overline{u})^2\big)f'(|u|^2)\;dx+J_{1,\beta}+J_{2,\beta}+J_{3,\beta},
\end{aligned}
 \end{equation}
where we denote
\begin{align*}
J_{1,\beta}=&\frac{1}{2}\int_{\RR^d}\Rr\left[\left(\frac{\partial^2g}{\partial x^2}(u)\partial_tu_1+\frac{\partial^2g}{\partial x\partial y}(u)\partial_tu_2+i\frac{\partial^2g}{\partial x\partial y}(u)\partial_tu_1+i\frac{\partial^2g}{\partial y^2}(u)\partial_tu_2\right)\left(\partial^\beta \overline{u}\right)^2\right]\;dx\\
J_{2,\beta}=&\sum_{j=1}^d\Ii\int_{\RR^d}\partial_{x_j}\left(\frac{\partial g}{\partial x}(u)+i\frac{\partial g}{\partial y}(u)\right)\partial^{\beta+e_j}\overline{u}\partial^{\beta}\overline{u}\;dx\\
J_{3,\beta}=&\Ii\int_{\RR^d}\left(\partial_xg(u)+i\partial_yg(u)\right)\overline{\partial^\beta(g(u))}\partial^\beta\overline{u}.
\end{align*}

We observe that 
\begin{equation}\label{prop_multinom}
\sum_{|\alpha|=s}\dbinom{s}{\alpha}\left(\partial^\alpha \overline{u}\right)^2-\sum_{|\beta|=s-1}\dbinom{s-1}{\beta}\sum_{j=1}^d\left(\partial^{\beta+e_j} \overline{u}\right)^2=0.
\end{equation}

\noindent Thus, \begin{equation}\label{der_Fs}
G_{s}'(t)=\sum_{|\alpha|=s}\dbinom{s}{\alpha}I_\alpha-\sum_{|\beta|=s-1}\dbinom{s-1}{\beta}\left(J_{1,\beta}+J_{2,\beta}+J_{3,\beta}\right).
\end{equation}

\begin{Rq}\label{rq_choix_funct2}
Note that, considering \eqref{der_alpha}, \eqref{der_beta}, and \eqref{prop_multinom}, the property that allows us to obtain \eqref{der_Fs} is in fact $\Ii\;\partial_xg=\Rr\;\partial_yg$. Indeed, this assumption suffices to have: for all $\alpha\in\NN^d$ with $|\alpha|=s$,
$$2\;\Ii\left[\left(\partial_xg(u)\Rr\left(\partial^\alpha u\right)+\partial_yg(u)\Ii\left(\partial^\alpha u\right)\right)\partial^{\alpha} \overline{u}\right]=\Ii\left[\left(\partial_xg(u)+i\partial_yg(u)\right)\left(\partial^{\alpha}\overline{u}\right)^2\right].$$
\end{Rq}

\underline{Step 2}: Control of the derivative of $G_s$. \\

Take $\alpha\in\NN^d$ such that $|\alpha|=s$. Let $q$, $r$, and $\tilde\alpha_1,\dots,\tilde\alpha_q$ be as in (\ref{terme_type}), and denote by $I_{\alpha,q,r,\tilde\alpha_1,\dots,\tilde\alpha_q}(R_k)$ the integral defined exactly as in  (\ref{terme_type}) by replacing $u$ by the soliton $R_k$, for all $k=1,\dots,K$.
Then we have
\begin{equation}\label{control_terme_type}
\left|I_{\alpha,q,r,\tilde\alpha_1,\dots,\tilde\alpha_q}(u)-\sum_{k=1}^KI_{\alpha,q,r,\tilde\alpha_1,\dots,\tilde\alpha_q}(R_k)\right|\leq Ce^{-\min\left\{\theta_{s-1},\frac{4\theta}{d+1}\right\}t},
\end{equation} for some constant $C$ independent of $A_{s}$.

In order to prove (\ref{control_terme_type}), one proceeds by decomposition of $I_{\alpha,q,r,\tilde\alpha_1,\dots,\tilde\alpha_q}(u)$ as follows. The basic idea is to make terms in $v$ (which provide the expected exponential term at the right-hand side of (\ref{control_terme_type})) appear. Let us explicit this decomposition:

\begin{align*}
\displaystyle I_{\alpha,q,r,\tilde\alpha_1,\dots,\tilde\alpha_q}(u)=&\displaystyle \int_{\RR^d}\partial^{\tilde\alpha_1}u_{j_1}\dots\partial^{\tilde\alpha_q}u_{j_q}\Ii\left(\left(\frac{\partial^q g}{\partial_x^{r}\partial_y^{q-r}}(u)-\frac{\partial^q g}{\partial_x^{r}\partial_y^{q-r}}(R)\right)\partial^\alpha\overline{u}\right)\;dx&(I_{\alpha,1})\\
&+\int_{\RR^d}\partial^{\tilde\alpha_1}{v}_{j_1}\partial^{\tilde\alpha_2}u_{j_2}\dots\partial^{\tilde\alpha_q}{u}_{j_q}\Ii\left(\frac{\partial^q g}{\partial_x^{r}\partial_y^{q-r}}(R)\partial^\alpha\overline{u}\right)\;dx&(I_{\alpha,2})\\
&+\int_{\RR^d}\partial^{\tilde\alpha_1}{R}_{j_1}\partial^{\tilde\alpha_2}v_{j_2}\partial^{\tilde\alpha_3}u_{j_3}\dots\partial^{\tilde\alpha_q}{u}_{j_q}\Ii\left(\frac{\partial^q g}{\partial_x^{r}\partial_y^{q-r}}(R)\partial^\alpha\overline{u}\right)\;dx&(I_{\alpha,3})\\
&+\dots&(I_{\alpha,\dots})\\
&+\int_{\RR^d}\partial^{\tilde\alpha_1}{R}_{j_1}\partial^{\tilde\alpha_2}R_{j_2}\dots\partial^{\tilde\alpha_{q-1}}R_{j_{q-1}}\partial^{\tilde\alpha_q}{v}_{j_q}\Ii\left(\frac{\partial^q g}{\partial_x^{r}\partial_y^{q-r}}(R)\partial^\alpha\overline{u}\right)\;dx&(I_{\alpha,q+1})\\
&+\int_{\RR^d}\partial^{\tilde\alpha_1}{R}_{j_1}\dots\partial^{\tilde\alpha_q}{R}_{j_q}\Ii\left(\frac{\partial^q g}{\partial_x^{r}\partial_y^{q-r}}(R)\partial^\alpha\overline{v}\right)\;dx&(I_{\alpha,q+2})\\
&+\int_{\RR^d}\partial^{\tilde\alpha_1}{R}_{j_1}\dots\partial^{\tilde\alpha_q}{R}_{j_q}\Ii\left(\frac{\partial^q g}{\partial_x^{r}\partial_y^{q-r}}(R)\partial^\alpha\overline{R}\right)\;dx&(I_{\alpha,q+3})
\end{align*}

Now, we control each preceding term $I_{\alpha,1},\dots, I_{\alpha,q+3}$ occurring in the preceding decomposition by means of the induction assumption and some classical tools in functional analysis, namely Hölder inequality, Sobolev embeddings, and Gagliardo-Nirenberg inequalities. \\

Let us notice that 
\begin{equation}\label{R_bounded}
\sup_{t\in\RR}\|R(t)\|_{L^\infty}<+\infty.
\end{equation}
 Considering that $s_0>\frac{d}{2}$, we deduce then from (\ref{estimate_a_priori}) and the Sobolev embedding $H^{\lfloor\frac{d}{2}\rfloor+1}(\RR^d)\hookrightarrow L^\infty(\RR^d)$ that there exists $A\ge 0$ such that for all $n$,
\begin{equation}\label{u_bounded}
\forall\;t\in[S_n^*,S_n],\qquad\|u_n(t)\|_{L^\infty}\le A.
\end{equation}

Since $\frac{\partial^q g}{\partial_x^{r}\partial_y^{q-r}}$ is $W^{1,\infty}_{loc}$ on $\CC$ (or in other words locally Lipschitz on $\CC$), it results from (\ref{R_bounded}) and (\ref{u_bounded}) that 
\begin{equation}\label{est_loc_lipschitz}
|I_{\alpha,1}|\leq C\int_{\RR^d}|v| |\partial^\alpha u| |\partial^{\tilde\alpha_1}u|\dots |\partial^{\tilde\alpha_q}u|\;dx.
\end{equation}
We now estimate the integral $\int_{\RR^d}|v| |\partial^\alpha u| |\partial^{\tilde\alpha_1}u|\dots |\partial^{\tilde\alpha_q}u|\;dx$ by means of Hölder inequality. For this, we have to be careful concerning the choice of the involved Lebesgue spaces (or in other words the Hölder exponents) considering that $\partial^{\tilde{\alpha}_i}u\in H^{s_0-|\tilde{\alpha}_i|}$.\\ 
We define $\mathcal{I}:=\left\{i\in\left\{1,\dots,q\right\}| \; s_0-|\tilde\alpha_i|<\frac{d}{2}\right\}$, $\mathcal{J}:=\left\{i\in\left\{1,\dots,q\right\}| \; s_0-|\tilde\alpha_i|>\frac{d}{2}\right\}$, and
$$m_i:=\begin{dcases}
\frac{2d}{d-2(s_0-|\tilde\alpha_i|)}&\text{ if } i\in \mathcal{I}\\
\infty&\text{ if } i\in \mathcal{J}.\\
\end{dcases}$$ For $i\in\left\{1,\dots,q\right\}\setminus \left(\mathcal{I}\cup\mathcal{J}\right)$, we take $m_i\in (0,+\infty)$ large enough so that $$\sum_{i=1}^q\frac{1}{m_i}<\frac{1}{2},$$ which is possible since
\begin{equation}\label{condition_s0}
\begin{aligned}
\frac{1}{2}-\left(\sum_{i\in \mathcal{I}}\frac{1}{m_i}+\sum_{i\in \mathcal{J}}\frac{1}{m_i}\right)&=\frac{1}{2}-\sum_{i\in \mathcal{I}}\frac{1}{m_i}=\frac{1}{2}-\sum_{i\in \mathcal{I}}\frac{d-2(s_0-|\tilde\alpha_i|)}{2d}\\
&\ge \frac{1}{2}-q\left(\frac{1}{2}-\frac{s_0}{d}\right)-\frac{s}{d} \geq (1-q)\left(\frac{1}{2}-\frac{s_0}{d}\right) >0.
\end{aligned}
\end{equation} due to our assumption on $s_0$ and the fact that $q>1$.
Then, we observe that for all $i=1,\dots,q$, $\partial^{\tilde\alpha_i}u\in H^{s_0-|\tilde\alpha_i|}(\RR^d)\hookrightarrow L^{m_i}(\RR^d)$  by the classical Sobolev embedding theorem.  Using Hölder inequality, we obtain
\begin{equation}\label{est_I1'}
\begin{aligned}
|I_{\alpha,1}|&\leq C\| \partial^\alpha u\|_{L^2}\prod_{i=2}^q\|\partial^{\tilde\alpha_i}u\|_{L^{m_i}}\|v\|_{L^m} \leq C\|v\|_{L^{m}}.
\end{aligned}
\end{equation}
where $m\geq 2$, and 
$\frac{1}{m}=\frac{1}{2}-\sum_{i=1}^q\frac{1}{m_i}\ge (q-1)\left(\frac{s_0}{d}-\frac{1}{2}\right)>0$
by definition of the $m_i$, $i=1,\dots, q$.
The following Gagliardo-Nirenberg inequality 
\begin{equation}
\|v\|_{L^{m}}\leq C\|v\|_{H^{s_0'}}^\sigma\|v\|_{L^2}^{1-\sigma},
\end{equation}
with $s_0':=\lfloor\frac{d}{2}\rfloor +1\le \frac{d+1}{2}$ and $\sigma:=\frac{d}{s_0'}\left(\frac{1}{2}-\frac{1}{m}\right)$ (which implies $1-\sigma \geq \frac{1}{s_0'}$, since $2s_0'-d\ge 1$)
leads finally to
\begin{equation}\label{est_I1}
\begin{aligned}
|I_{\alpha,1}|&\leq C\|v\|_{L^2}^{1-\sigma} \leq Ce^{-\frac{2\theta}{s_0'}t} \leq Ce^{-\frac{4\theta}{d+1}t}.
\end{aligned}
\end{equation}

\noindent To estimate $I_{\alpha,2},\dots,I_{\alpha,q+1}$, one proceeds as before. For instance, let us explain how to deal with $I_{\alpha,2}$; the same would be done for the other integrals. \\
We choose $m_1'$ such that $H^{s_1-1-|\tilde\alpha_1|}\hookrightarrow L^{m_1'}(\RR^d)$  and $\frac{1}{2}-\frac{1}{m_1'}-\sum_{i=2}^q\frac{1}{m_i}>0$. Then, again due to Hölder inequality, we have:
\begin{equation}\label{est_I2}
\begin{aligned}
|I_{\alpha,2}|&\leq \| \partial^\alpha u\|_{L^2}\| \partial^{\tilde\alpha_1} v\|_{L^{m_1'}}\prod_{i=1}^q\|\partial^{\tilde\alpha_i}u\|_{L^{m_i}}\left\|\frac{\partial^qg}{\partial_x^r\partial_y^{q-r}}(R)\right\|_{L^{\infty}}\\
&\le C\|v\|_{H^{s-1}} \le Ce^{-\theta_{s-1}t}.
\end{aligned}
\end{equation}
Similarly, we check that
\begin{equation}\label{est_I3}
\forall\;i\in\{3,\dots,q+1\},\quad |I_{\alpha,i}|\leq Ce^{-\theta_{s-1}t}.
\end{equation}

\noindent Now, let us deal with $I_{\alpha,q+2}$. By \eqref{decay_prop} and the fact that $\frac{\partial^q g}{\partial_x^{r}\partial_y^{q-r}}\in W^{1,\infty}_{loc}(\CC)$, we have:
$$I_{\alpha,q+2}=\sum_{k=1}^K\int_{\RR^d}\partial^{\tilde\alpha_1}{R}_{k,j_1}\dots\partial^{\tilde\alpha_q}{R}_{k,j_q}\Ii\left(\frac{\partial^q g}{\partial_x^{r}\partial_y^{q-r}}(R_k)\partial^\alpha\overline{v}\right)\;dx+Ce^{-2\theta t}\|v\|_{H^s}.$$
Again by assumption, each partial derivative of $\frac{\partial^qg}{\partial x^r\partial y^{q-r}}(R_k)$ is bounded, thus the integral 
$$\int_{\RR^d}\partial^{\alpha-e_l}\overline{v}\partial^{e_l}\left(\partial^{\tilde{\alpha}_1}R_{k,j_1}\dots\partial^{\tilde{\alpha}_q}R_{k,j_q}\frac{\partial^qg}{\partial_x^r\partial_y^{q-r}}(R_k)\right)\;dx$$ makes sense and one can integrate once by parts to obtain
\begin{equation}\label{est_ipp}
\begin{aligned}
\MoveEqLeft[6] \left|\int_{\RR^d}\partial^{\tilde{\alpha}_1}{R}_{k,j_1}\dots\partial^{\tilde{\alpha}_q}{R}_{k,j_q}\Ii\left(\frac{\partial^q g}{\partial_x^{r}\partial_y^{q-r}}(R_k)\partial^\alpha\overline{v}\right)\;dx\right|\\
&=\left|\Ii\int_{\RR^d}\partial^{\alpha-e_l}\overline{v}\partial^{e_l}\left(\partial^{\tilde{\alpha}_1}R_{k,j_1}\dots\partial^{\tilde{\alpha}_q}R_{k,j_q}\frac{\partial^qg}{\partial_x^r\partial_y^{q-r}}(R_k)\right)\;dx\right|\\
&\leq C\|v(t)\|_{H^{s-1}} \leq Ce^{-\theta_{s-1}t}.
\end{aligned}
\end{equation}
Thus,
\begin{equation}\label{est_Iq+2}
|I_{\alpha,q+2}|\leq Ce^{-\theta_{s-1}t}.
\end{equation}

\noindent Finally, by \eqref{decay_prop} and using once more that $\frac{\partial^qg}{\partial_x^r\partial_y^{q-r}}$ is in $W^{1,\infty}_{loc}(\CC)$,
\begin{equation}\label{est_Iq+3}
\begin{aligned}
\left|I_{\alpha,q+3}-\sum_{k=1}^KI_{\alpha,q,r,\tilde{\alpha}_1,\dots,\tilde{\alpha}_q}(R_k)\right|&\le Ce^{-2\theta t}.
\end{aligned}
\end{equation}
Hence, we conclude from \eqref{est_I1}, \eqref{est_I2}, \eqref{est_I3}, \eqref{est_Iq+2}, and \eqref{est_Iq+3} that \eqref{control_terme_type} holds true. \\ 

The expressions $J_{1,\beta},J_{2,\beta}, J_{3,\beta} $ (given before) consist of terms that can be controlled in a similar manner.  Let us denote by $J_{i,\beta}(R_k)$ the same integral as $J_{i,\beta}$ where $R_k$ replaces $u$  for all $i=1,2,3$ and for all $k=1,\dots,K$.  One can check that 
\begin{equation}\label{term_J1}
\left|J_{1,\beta}-\sum_{k=1}^KJ_{1,\beta}(R_k)\right|\le Ce^{-\min\left\{\theta_{s-1},\frac{4\theta}{d+1}\right\}t};
\end{equation}
\begin{equation}\label{term_J2}
\left|J_{2,\beta}-\sum_{k=1}^KJ_{2,\beta}(R_k)\right|\le Ce^{-\min\left\{\theta_{s-1},\frac{4\theta}{d+1}\right\}t};
\end{equation}
\begin{equation}\label{term_J3}
\left|J_{3,\beta}-\sum_{k=1}^KJ_{3,\beta}(R_k)\right|\le Ce^{-\theta_{s-1}t}.
\end{equation}

\underline{Step 3}: Related functional involving $R_k$.\\

Let $k\in\{1,\dots,K\}$. An immediate induction argument shows that for all multi-index $\alpha\in\NN^d$ such that $|\alpha|=s$, for all multi-index $\alpha'\preceq\alpha$, there exists $z_{\alpha'}\in\CC$ such that 
$$\partial^\alpha R_k(t,x)=\sum_{\alpha'\preceq\alpha}z_{\alpha'}\partial^{\alpha'}Q_{\omega_k}(x-x_k^0-v_kt)e^{i\left(\frac{1}{2}v_k\cdot x+\left(\omega_k-\frac{|v_k|^2}{4}\right)t+\gamma_k\right)}.$$
Therefore 
\begin{align*}
\int_{\RR^d}\left|\partial^\alpha R_k\right|^2(t)\;dx&=\int_{\RR^d}\left|\sum_{\alpha'\preceq \alpha}z_{\alpha'}\partial^{\alpha'}Q_{\omega_k}(x-x^0_k-v_kt)\right|^2\;dx =\int_{\RR^d}\left|\sum_{\alpha'\preceq \alpha}z_{\alpha'}\partial^{\alpha'}Q_{\omega_k}(x)\right|^2\;dx,
\end{align*}
so that 
\begin{equation}\label{prop_der_funct_Rk_1}
\frac{d}{dt}\int_{\RR^d}|\partial^\alpha R_k|^2\;dx=0.
\end{equation}
Furthermore for all multi-index $\beta$ such that $|\beta|=s-1$,
$$R_k^2\left(\partial^\beta\overline{R_k}\right)^2(t,x)=Q_{\omega_k}^2(x-x_k^0-v_kt)\left(\sum_{\beta'\preceq \beta}z_{\beta'}\partial^{\beta'}Q_{\omega_k}(x-x^0_k-v_kt)\right)^2,$$ from which we infer also
\begin{equation}\label{prop_der_funct_Rk_2}
\frac{d}{dt}\int_{\RR^d}\Rr\big(R_k^2\big(\partial^\beta\overline{R_k}\big)\big)^2f'(|R_k|^2)\;dx=0.
\end{equation}
Hence, gathering \eqref{prop_der_funct_Rk_1} and \eqref{prop_der_funct_Rk_2},
\begin{equation}\label{funct_Rk}
0=\frac{d}{dt}\int_{\RR^d}\left\{\sum_{|\alpha|=s}\dbinom{s}{\alpha}|\partial^\alpha R_k|^2-\sum_{|\beta|=s-1}\dbinom{s-1}{\beta}\Rr\left(R_k^2\left(\partial^\beta \overline{R_k}\right)^2\right)f'(|R_k|^2)\right\}\;dx.
\end{equation} 
Considering that this last quantity can be written as exactly the same linear combination of terms as $G_{n,s}'$ (we refer to \eqref{der_Fs}) where we replace just $u$ by $R_k$, we conclude from \eqref{control_terme_type}, \eqref{term_J1}, \eqref{term_J2}, \eqref{term_J3}, and \eqref{funct_Rk} that
\begin{equation}
\forall\; t\in [S_n^*,S_n],\qquad|G_{n,s}'(t)|\le Ce^{-\min\left\{\theta_{s-1},\frac{4\theta}{d+1}\right\}t}.
\end{equation}
Integrating the preceding inequality between $t$ and $S_n$ yields directly Lemma \ref{der_funct}.
\end{proof}

Let us now conclude the proof of Proposition \ref{boot}. \\

We observe that, for all $t\in[S_n^*,S_n]$,
\begin{equation}\label{dec_norm_v}
\begin{aligned}
\MoveEqLeft[3]\displaystyle\sum_{|\alpha|=s}\dbinom{s}{\alpha}\int_{\RR^d}|\partial^\alpha v_n(t)|^2\;dx\\
=&\;\displaystyle\left(G_{n,s}(t)-G_{n,s}(S_n)\right)
+2\sum_{|\alpha|=s}\dbinom{s}{\alpha}\Rr\int_{\RR^d}\partial^{\alpha-e_{l(\alpha)}}v_n(t)\partial^{\alpha+e_{l(\alpha)}}\overline{R}(t)\;dx\\
&+\sum_{|\alpha|=s}\dbinom{s}{\alpha}\left(\int_{\RR^d}|\partial^\alpha R(S_n)|^2\;dx-\int_{\RR^d}|\partial^\alpha R(t)|^2\;dx\right)\\
&+\displaystyle\sum_{|\beta|=s-1}\dbinom{s-1}{\beta}\left(\int_{\RR^d}\Rr\left(u_n^2(\partial^{\beta}\overline{u_n})^2\right)f'(|u_n|^2)(t)-\int_{\RR^d}\Rr\left(R^2(\partial^\beta\overline{R})^2\right)f'(|R|^2)(t)\right)\\
&+\displaystyle\sum_{|\beta|=s-1}\dbinom{s-1}{\beta}\left(\int_{\RR^d}\Rr\left(R^2(\partial^{\beta}\overline{R})^2\right)f'(|R|^2)(t)-\int_{\RR^d}\Rr\left(R^2(\partial^{\beta}\overline{R})^2\right)f'(|R|^2)(S_n)\right)\\
&+\displaystyle\sum_{|\alpha|=s}\dbinom{s}{\alpha}\int_{\RR^d}\left|\partial^\alpha v_n(S_n)\right|^2-2\sum_{|\alpha|=s}\dbinom{s}{\alpha}\Rr\int_\RR \partial^{\alpha-e_{l(\alpha)}}v_n(S_n)\partial^{\alpha+e_{l(\alpha)}}\overline{R}(S_n)\;dx.
\end{aligned}\end{equation} 

Then, by means of \eqref{init}, \eqref{prop_der_funct_Rk_1}, and \eqref{prop_der_funct_Rk_2}, we infer
\begin{multline}\label{dec_norm_v_bis}
\displaystyle\left|\sum_{|\alpha|=s}\dbinom{s}{\alpha}\int_{\RR^d}|\partial^\alpha v_n(t)|^2\;dx\right| \leq \displaystyle\left|G_{n,s}(t)-G_{n,s}(S_n)\right|+C\|v_n\|_{H^{s-1}}+Ce^{-2\theta t}\\
+\displaystyle C\left|\int_{\RR^d}\Rr\left(u_n^2(\partial^{\beta}\overline{u_n})^2\right)f'(|u_n|^2)(t)\;dx-\int_{\RR^d}\Rr\left(R^2(\partial^\beta\overline{R})^2\right)f'(|R|^2)(t)\;dx\right|.
\end{multline} 

Now, from (\ref{decay_prop}), (\ref{diff}), (\ref{funct_Rk}), and from the inequality 
\begin{equation}
\begin{aligned}
\left|\int_{\RR^d}\Rr\big(u_n^2(\partial^{\beta}\overline{u_n})^2\big)f'(|u_n|^2)(t)\;dx-\int_{\RR^d}\Rr\big(R^2(\partial^{\beta}\overline{R})^2\big)f'(|R|^2)(t)\;dx\right|\le C\|v_n(t)\|_{H^{s-1}},
\end{aligned}
\end{equation}
resulting from the local boundedness of the distributional derivative of $z\mapsto \frac{1}{2}\left(\partial_xg+i\partial_y g\right)(z)= z^2f'(|z|^2)$, we deduce the existence of $C_s\ge 0$ such that for all $n\in\NN$,\begin{equation}
\forall\;t\in [S_n^*,S_n],\qquad\|v_n(t)\|_{H^s}^2\le C_s^2e^{-\min\left\{\theta_{s-1},\frac{4\theta}{d+1}\right\}t}.
\end{equation}
This is rewritten as follows:
\begin{equation}
\forall\;t\in [S_n^*,S_n],\qquad\|v_n(t)\|_{H^s}\le C_s e^{-\theta_{s}t},
\end{equation}
which is exactly the expected result. Thus the induction argument implies that for all $n\in\NN$,
\begin{equation}
\forall\;t\in [S_n^*,S_n],\qquad\|v_n(t)\|_{H^{s_0}}\le C_{s_0}e^{-\theta_{s_0}t}.
\end{equation} This puts an end to the proof of Proposition \ref{boot}.
\end{proof}

\bigskip 

Now we explain how to deduce from Proposition \ref{boot} that $S_n^*$ can be chosen independently of $n$, and by this means, we finish the proof of Proposition \ref{core}. We pick up $T_1\ge T_0$ such that $C_{s_0}e^{-\theta_{s_0}T_1}<A_{s_0}$. Let $n\in\NN$ be such that $S_n\ge T_1$, and assume by contradiction that $S_n^*>T_1$. Then by continuity of $v_n$ in $S_n^*$ (and by definition of $S_n^*$ as  infimum), we have $\|v_n(S_n^*)\|_{H^{s_0}}=A_{s_0}$. On the other hand, 
\begin{equation}\label{arg_contradiction}
\begin{aligned}
\|v_n(S_n^*)\|_{H^{s_0}}&\leq C_{s_0}e^{-\theta_{s_0}S_n^*} \leq C_{s_0}e^{-\theta_{s_0}T_1} < A_{s_0},
\end{aligned}
\end{equation}
which yields a contradiction. Thus $S_n^*\le T_1$. Hence, for all $n\in\NN$ such that $S_n\ge T_1$, we have \[\forall\;t\in [T_1,S_n],\qquad\|v_n(t)\|_{H^{s_0}}\le  C_{s_0}e^{-\theta_{s_0}t}.\] 
If necessary we drop the first terms of the sequence $(S_n)$ and re-index it in order to obtain:
\begin{equation}\label{memeT1}
\forall\;n\in\NN,\:\forall\;t\in [T_1,S_n],\qquad\|v_n(t)\|_{H^{s_0}}\le  C_{s_0}e^{-\theta_{s_0}t}.
\end{equation}
Hence, Proposition \ref{core} is established. 

\subsubsection{\underline{Step 2.2}: Independence of $T_1$ with respect to $s$}

Now, we justify that $T_1$ can be chosen independent of $s>\frac{d}{2}$, which is useful to obtain 
\begin{equation}
\forall\; s\in\NN^*,\:\exists\;C_s\ge 0,\:\forall\;t\in[T_1,S_n],\qquad \|v_n(t)\|_{H^s}\le C_se^{-\theta_st},
\end{equation}
in the case where $s_0=\infty$. \\

If $g$ is $\mathscr{C}^\infty$ on $\CC$ as an $\RR$-differentiable function, it is in particular of class $\mathscr{C}^{\lfloor \frac{d}{2}\rfloor +2}$, so that we can apply the previous result: there exists $T_1\ge T_0$ such that for all $n\in\NN$, 
\begin{equation}
\forall\; t\in [T_1,S_n],\qquad\|v_n(t)\|_{H^{\lfloor \frac{d}{2}\rfloor +1}}\le C_{\lfloor \frac{d}{2}\rfloor +1}e^{-\theta_{\lfloor \frac{d}{2}\rfloor +1}t}.
\end{equation}
Let $s\ge \left\lfloor\frac{d}{2} \right\rfloor + 2$ and assume that for all $s'\in\{\lfloor \frac{d}{2}\rfloor +2,\dots,s\}$,
$$\forall\;t\in[T_1,S_n],\qquad \|v_n(t)\|_{H^{s'-1}}\leq C_{s'-1}e^{-\theta_{s'-1}t}.$$
Then define 
$$S_{n,s}^*:=\inf\{t\ge T_1\;|\;\forall\;\tau\in[t,S_n],\:\|v_n(\tau)\|_{H^s}\le A_s\},$$ for some constant $A_s>\max\{2K\mu_s,1\}$ to be determined. We show exactly as before (that is considering the functionals $G_{n,s}$) the existence of $\tilde{C}_s > 0$ independent of $n$, $t$, and $A_s$ such that
\begin{equation}
\forall\;t\in[S_{n,s}^*,S_n],\qquad \|v_n(t)\|_{H^s}^2\leq \tilde{C}_s^2A_s e^{-2\theta_s t},
\end{equation}
or also 
\begin{equation}
\forall\;t\in[S_{n,s}^*,S_n],\qquad \|v_n(t)\|_{H^s}\leq \tilde{C}_sA_s^\frac{1}{2} e^{-\theta_s t}.
\end{equation}
Indeed, the constant $A$ in \eqref{u_bounded} does not depend on $s$ and $A_s\ge 1$ so that we have for example as in \eqref{est_I1'} and then \eqref{est_I1}:
$$|I_{\alpha,1}|\leq CA_se^{-\frac{4\theta}{d+1}t},$$ with $C$ independent of $A_s$. 

\noindent Choosing $A_s>\tilde{C}_s^2e^{-2\theta_s T_1}$ and arguing as in \eqref{arg_contradiction}, we conclude that $S_{n,s}^*=T_1$. Hence, $T_1$ is uniform with respect to $s$. 

\subsubsection{\underline{Step 2.3}: Looking for optimal exponential decay rates in the uniform $H^s$-estimates}

The next result uses and improves that of Proposition \ref{core}.
 
\begin{Prop}\label{prop:core_bis}
For all $s\in\{1,\dots,s_0\}$, there exists $\tilde{A}_s\ge 0$ such that for all $t\in [T_1,S_n]$, \begin{equation}\label{core_bis_mieux}
\|v_n(t)\|_{H^{s}}\le \tilde{A}_se^{-\frac{2\theta}{s+1}t}.
\end{equation}
\end{Prop}

\begin{proof}
Let $s'\in\{1,\dots,s\}$. By \eqref{est}, \eqref{core_bis}, and the following interpolation inequality 
$$\|v_n(t)\|_{H^{s'}}\le \|v_n(t)\|_{L^2}^{\gamma}\|v_n(t)\|_{H^{s}}^{1-\gamma},$$ with $\gamma=\frac{s-s'}{s}$, we have for all $t\in [T_1,S_n]$, 
$$\|v_n(t)\|_{H^{s'}}\le Ce^{-2\theta\frac{s-s'}{s}t}.$$

\noindent Now, set
$$T_{n}^*:=\inf\{t\ge T_1\;|\;\forall\;\tau\in[t,S_n],\;\|v_n(\tau)\|_{H^{s}}\le \tilde{A}_se^{-\mu \tau}\},$$ for some $\mu\in(0,2\theta)$ and for some $\tilde{A}_s\ge 1$ to be determined later.

\noindent Let $t$ belong to $[T_n^*,S_n]$. Then by the proof set up before,
$$\|v_n(t)\|_{H^{s}}^2\le C\left(\|v_n(t)\|_{H^{s-1}}+e^{-2\theta t}\right).$$
In addition, we obtain once again by interpolation
\begin{align*}
\|v_n(t)\|_{H^{s-1}}&\le \|v_n(t)\|_{L^2}^{\frac{1}{s}}\|v_n(t)\|_{H^{s}}^{\frac{s-1}{s}} \le C\tilde{A}_{s}^{1-\frac{1}{s}}e^{-\frac{2\theta}{s}t}e^{-\mu\frac{s-1}{s}t}.
\end{align*}
Since $\mu\le 2\theta$, we have $\frac{2\theta+\mu(s-1)}{s}\le 2\theta$, and so there exists $\tilde{C}\ge 0$ (independent of $\tilde{A}_s$) such that 
\begin{equation}\label{ineg_bootstrap}
\begin{aligned}
\|v_n(t)\|_{H^{s}}^2
\le \tilde{C}\tilde{A}_{s}^{1-\frac{1}{s}}e^{-\frac{2\theta}{s}t}e^{-\mu\frac{s-1}{s}t}.\\
\end{aligned}
\end{equation}

Now, choose $$\mu:=\frac{2\theta}{s+1}\quad \text{and}\quad \tilde{A}_s>\tilde{C}^{\frac{s}{s+1}}.$$ 
By a similar argument as that set up to prove Proposition \ref{core}, we see that $T_n^*=T_1$. Indeed, if we had $T_n^*>T_1$, then by the definition of $T_n^*$ and by continuity of $v_n$ in $T_n^*$, we would obtain $$\tilde{A}_s^2e^{-2\mu T_n^*}=\|v_n(T_n^*)\|_{H^{s}}^2\le  \tilde{C}\tilde{A}_s^{1-\frac{1}{s}}e^{-\frac{2\theta}{s}T_n^*}e^{-\mu\frac{s-1}{s}T_n^*},$$
thus, by the choice of $\mu$,
$$
\tilde{A}_s^2\le \tilde{C}\tilde{A}_s^{1-\frac{1}{s}},
$$
which is a contradiction. Consequently, estimate \eqref{core_bis_mieux} does indeed hold.
\end{proof}

\subsection{\underline{Step 3}: Conclusion of the proof of Theorem \ref{th_reg_f}}

We construct now the multi-soliton $u$ using the same arguments as those of Martel \cite[paragraph 2, Step 2]{martel} and Martel and Merle \cite[Paragraph 2]{schMM}. The crucial point is the following lemma, obtained by a compactness argument.

\begin{Lem}
There exist $\varphi\in H^{s_0}(\RR^d)$ and a subsequence $(u_{n_k}(T_1))_k$ of $(u_n(T_1))_n$ such that \[\|u_{n_k}(T_1)-\varphi\|_{H^{s_0}}\underset{k\to+\infty}{\longrightarrow} 0.\]
\end{Lem}

\noindent Note that the main ingredients to show this lemma are:
\begin{itemize}
\item the uniform $H^{s_0}$-estimate obtained in Step 2.
\item the following $L^2$-compactness assertion: for all $\epsilon>0$, there exists $\mathcal{K}$  a compact subset of $\RR^d$ such that $$\forall\;n\in\NN,\qquad\int_{\mathcal{K}}|u_n(T_1,x)|^2\;dx\le \epsilon.$$
\end{itemize} 

\noindent Then by local well-posedness of \eqref{NLS} in $H^{s_0}(\RR^d)$ with continuous dependence on compact sets of time \cite[Theorem 1.6]{dai}, the solution $u$ of \eqref{NLS} such that $u(T_1)=\varphi$ is defined in $H^{s_0}(\RR^d)$ and for all $t\ge T_1$, $\left\|u_{n_k}(t)-\varphi(t) \right\|_{H^{s_0}}\to 0$ as $k\to +\infty$. Thus $\varphi$ turns out to be the desired multi-soliton. Besides, the quantities $\|u(t)-R(t)\|_{H^s}$ decrease exponentially; this result is obtained by passing to the limit as $k$ tends to $+\infty$ in the $H^{s}$-uniform estimates given by Proposition \ref{prop:core_bis}, that is for all $s=1,\dots, s_0$, for $k$ large enough:
$$\|u_{n_k}(t)-R(t)\|_{H^{s}}\le C_se^{-\frac{2\theta}{s+1}t}.$$
This yields precisely \eqref{resreg1}. \\

Note that in the case where $g$ is $\mathscr{C}^\infty$ (for example when we consider the pure power non-linearity with $p$ an odd integer), we obtain \eqref{resreg2} as a consequence of \eqref{resreg1}, by interpolating the corresponding $H^s$-estimates, and by the independence of $T_1$ with respect to $s$ proved in Step 2.2.

\section{Conditional uniqueness for multi-solitons of \eqref{NLS}}\label{uniq}

In this section, we prove the uniqueness result stated in Theorem \ref{th_uni_f}, that is for $d\le 3$. 
The strategy developed here would also work to prove Proposition \ref{uni:4d} under the corresponding stronger assumptions. 

Our uniqueness result holds due to the coercivity properties of the linearized operators around ground states, namely assumption (H3) when $f$ is not the $L^2$-critical non-linearity and \eqref{coer_crit} in Proposition \ref{coerccritic} in the $L^2$-critical pure power non-linearity case. The proof follows essentially the same lines in these two cases; the differences are only rooted in the use of the appropriate coercivity result.

We first develop the proof in the stable case, assuming $\tilde{f}:z\mapsto f(|z|^2)$ of class $\mathscr{C}^2$ satisfying 
\eqref{prop_croissance}. This covers in particular the $L^2$-subcritical assumption with $3\leq p<1+\frac{4}{d}$ and $d=1$ in Theorem \ref{uni}. In subsection \ref{subsec_L2crit}, we explain how to modify the calculations in order to perform the proof in the $L^2$-critical case, that is assuming $p=1+\frac{4}{d}$ and $f:r\mapsto r^{\frac{p-1}{2}}$, which will extend the uniqueness result as stated in Theorems \ref{uni} and \ref{th_uni_f}. \\

Let us denote $\varphi$ the multi-soliton of \eqref{NLS} constructed in Theorem \ref{reg1} for $d=1$ and in Theorem \ref{th_reg_f} for $d\ge 2$ (which is possible to consider by hypothesis). Set $\gamma:=\frac{2\theta}{3}$, where $\theta$ is defined in Theorem \ref{reg1} and let $T_1>0$ such that $\varphi$ belongs to $\mathscr{C}([T_1,+\infty),H^1(\RR))$ and
\begin{equation}\label{varphi_1}
\forall\;t\ge T_1,\qquad \|\varphi(t)-R(t)\|_{H^1}\le Ce^{-\gamma t}
\end{equation} for $d=1$, and such that $\varphi$ belongs to $\mathscr{C}([T_1,+\infty),H^{s_0}(\RR^d))$ and
\begin{equation}\label{varphi_3d}
\forall\;t\ge T_1,\qquad \|\varphi(t)-R(t)\|_{H^{s_0}}\le Ce^{-\gamma t}
\end{equation} for $d\ge 2$ (where $s_0=\lfloor\frac{d}{2}\rfloor+1$.)
In particular, due to the Sobolev embedding $H^{\left\lfloor\frac{d}{2}\right\rfloor+1}(\RR^d)\hookrightarrow L^\infty(\RR^d)$ for ($d\ge 1$), we emphasize that, for all $d\ge 1$, $\varphi\in\mathscr{C}([T_1,+\infty[,L^{\infty}(\RR^d))$ and
\begin{equation}\label{varphi_3d'}
\forall\;t\ge T_1,\qquad \|\varphi(t)-R(t)\|_{L^\infty}\le Ce^{-\gamma t}.
\end{equation}

Now, let us take $u$ in the class of multi-solitons satisfying (\ref{classexpo}) and define $z:=u-\varphi$ the difference of the two multi-solitons so that \begin{equation}\label{derz}
\partial_tz=i\left(\Delta z+f\left(|z+\varphi|^2\right)z+\left(f\left(|z+\varphi|^2\right)-f\left(|\varphi|^2\right)\right)\varphi\right),
\end{equation}
and
\begin{equation}\label{est_z_N}
\|z(t)\|_{H^1}=\mathrm{O}\left(\frac{1}{t^N}\right),\qquad\text{as } t\rightarrow +\infty,
\end{equation}
for some integer $N\ge 1$ to be determined later.\\

We will show that $z=0$. The idea is to practice some kind of modulation of the variable $z$ in order to ensure some orthogonality relations, needed to make use of the coercivity properties mentioned before. In other words, we obtain a new function (denoted by $\tilde{z}$) which seems to be adapted to the proof; this is the aim of subsection \ref{subsec_coerc}. Then, the control of the modulation parameters allows us to obtain an estimate of $\|z(t)\|_{H^1}$ in terms of $\|\tilde{z}(t)\|_{H^1}$; this combined with the estimate of the derivative of some kind of Weinstein functional $F_{\tilde{z}}$ (that we introduce in paragraph \ref{estenergysect}) enables us finally to see that $\tilde{z}=0$.

\subsection{Change of function to ensure a coercivity property in the stable case} \label{subsec_coerc}

\subsubsection{Introduction of a new variable}\label{change_var}

We introduce a new function $\tilde{z}$ on $[T,+\infty)\times\RR^d$ for $T$ sufficiently large by
\begin{equation}\label{def_tilde_z}
\forall\;(t,x)\in[T,+\infty)\times\RR^d,\quad\tilde{z}(t,x):=z(t,x)+\sum_{k=1}^K\left(ia_k(t)R_k(t,x)+b_k(t)\cdot\nabla R_k(t,x)\right),
\end{equation}
where $a_k(t)\in\RR$ and $b_k(t)\in\RR^d$ are chosen so that 
\begin{equation}\label{rel}
\forall\;k\in\{1,\dots,K\},\quad\forall\;i\in\{1,\dots,d\},\qquad \begin{dcases}
\Ii\int_{\RR^d}\tilde{z}\overline{R_k}\;dx=0\\
\Rr\int_{\RR^d}\tilde{z}\partial_{x_i}\overline{R_k}\;dx=0.
\end{dcases}
\end{equation}
Existence of $a_k(t)$ and $b_k(t)$ for $t$ large enough follows from:
\begin{Lem}\label{ex}
For $t$ large enough, and for all $k=1,\dots,K$, $a_k(t)$ and $b_k(t)$ are uniquely determined. Moreover, $t\mapsto a_k(t)$ and $t\mapsto b_k(t)$  are differentiable in the sense of distributions and 
\begin{equation}\label{estcoef}
|a_k(t)|,|b_k(t)|\le C\|z(t)\|_{L^2},
\end{equation}
\begin{equation}\label{estdercoef}
|a_k'(t)|,|b_k'(t)|\le C\|z(t)\|_{H^1}.
\end{equation}
\end{Lem}

\begin{proof}[Proof of Lemma \ref{ex}]
Let us introduce the symmetric block matrix 
\[M(t):=\left[\begin{array}{cccc}
A_{0,0}(t)&B_{1,1}(t)&\cdots&B_{1,d}(t)\\
^tB_{1,1}(t)&A_{1,1}(t)&\cdots&A_{1,d}(t)\\
\vdots&\vdots&\ddots&\vdots\\
^tB_{1,d}(t)&A_{d,1}(t)&\cdots&A_{d,d}(t)
\end{array}\right],\]
where $A_{i,j}(t)$ and $B_{1,j}(t)$ are $K\times K$-matrices with real entries defined by 
\begin{align*}
A_{0,0} & =\left[\Rr\int_{\RR^d}R_k\overline{R_l}\;dx\right]_{(k,l)}, \\
\forall\;(i,j)\in\{1,\dots,d\}^2,\quad A_{i,j} & =\left[\Rr\int_{\RR^d}\partial_{x_i}R_k\partial_{x_j}\overline{R_l}\;dx\right]_{(k,l)},\\
\forall\;j\in\{1,\dots,d\},\quad B_{1,j} & = \left[\Ii\int_{\RR^d}\partial_{x_j}R_k\overline{R_l}\;dx\right]_{(k,l)}.
\end{align*}
Set also \[x(t)=^t\left[a_1,\dots,a_K,b_{1,1},\dots,b_{K,1},\dots,b_{1,d},\dots,b_{K,d}\right]\] and 
\[ y(t)=-^t\left[y_0,y_1,\dots,y_d\right], \]
where 
\[ y_0=\left[\Ii\int_{\RR^d}z\overline{R_1}\;dx,\dots,\Ii\int_{\RR^d}z\overline{R_K}\;dx\right] \]
and for all $i=1,\dots,d$,
\[ y_i=\left[\Rr\int_{\RR^d}z\partial_{x_i}\overline{R_1}\;dx,\dots,\Rr\int_{\RR^d}z\partial_{x_i}\overline{R_K}\;dx\right]. \]
Then relations (\ref{rel}) rewrite clearly \[M(t)x(t)=y(t).\] Consequently, we have to show that $\det M(t)\neq 0$ for $t$ large enough to ensure existence and uniqueness of $a_k(t)$ and $b_k(t)$ for those values of $t$. To do this, observe that 
\begin{align*}
\Rr\int_{\RR^d}R_k(t)\overline{R_l}(t)\;dx & =\begin{dcases}
\int_{\RR^d}Q_{\omega_k}^2\;dx&\text{if } k=l\\
\mathrm{O}\left(e^{-\gamma t}\right)&\text{if } k\neq l,\\
\end{dcases} \\
\Rr\int_{\RR^d}\partial_{x_i}R_k(t)\partial_{x_j}\overline{R_l}(t)\;dx & =\begin{dcases}
\int_{\RR^d}\left\{\partial_{x_i}Q_{\omega_k}\partial_{x_j}Q_{\omega_k}+\frac{v_{i,k}v_{j,k}}{4}Q_{\omega_k}^2\right\}\;dx&\text{if } k=l\\
\mathrm{O}\left(e^{-\gamma t}\right)&\text{if } k\neq l,\\
\end{dcases} \\
\Ii\int_{\RR^d}\partial_{x_i}R_k(t)\overline{R_l}(t)\;dx & =\begin{dcases}
\frac{v_{i,k}}{2}\int_{\RR^d}Q_{\omega_k}^2\;dx&\text{if } k=l\\
\mathrm{O}\left(e^{-\gamma t}\right)&\text{if } k\neq l. \\
\end{dcases}
\end{align*}

Let us now compute $\det(M(t))$. For all $k=1,\dots, K$, let $L_k$ denote the $k$-th line of the block matrix 
$$\left[A_{0,0}(t)\:B_{1,1}(t)\:\cdots\:B_{1,d}(t)\right].$$
For all $i=1,\dots,d$, and for all $k=1,\dots,K$, replacing the $k$-th line $L_{i,k}$ of the block matrix $\left[^tB_{1,i}(t)\:A_{i,1}(t)\:\cdots\:A_{i,d}(t)\right]$ by $L_{i,k}-\frac{v_{i,k}}{2}L_k$, we obtain $\det(M(t))=\det(N(t))$ where 
$$N(t):=\left[\begin{array}{cccc}
A_{0,0}(t)&B_{1,1}(t)&\cdots&B_{1,d}(t)\\
C_{1}(t)&D_{1,1}(t)&\cdots&D_{1,d}(t)\\
\vdots&\vdots&\ddots&\vdots\\
C_d(t)&D_{d,1}(t)&\cdots&D_{d,d}(t)
\end{array}\right]$$
and 
$C_i(t)$ has entries zero on the diagonal and $O(e^{-\gamma t})$ everywhere else and
$D_{i,j}$ has entries $$\int_{\RR^d}\partial_{x_i}Q_{\omega_k}\partial_{x_j}Q_{\omega_k}\;dx$$ on the diagonal and $O(e^{-\gamma t})$ everywhere else.

\noindent Thus \begin{equation}\label{detM}
\det(M(t))=\left(\prod_{k=1}^K\int_{\RR^d}Q_{\omega_k}^2(x)\;dx\right)\det(D(t))+\mathrm{O}(e^{-\gamma t}),\end{equation}
where $D(t)$ is the sub-matrix of $N(t)$ with block matrices $D_{i,j}(t)$.

\noindent We observe that $D(t)$ admits a limit as $t\to +\infty$ which we denote by $D(\infty)$ and which corresponds to the block matrix 
$$\left[\begin{array}{ccc}
D_{1,1}(\infty)&\cdots&D_{1,d}(\infty)\\
\vdots&\ddots&\vdots\\
D_{d,1}(\infty)&\cdots&D_{d,d}(\infty)
\end{array}\right]$$
where $D_{i,j}({\infty})$ is a diagonal matrix with entries $\int_{\RR^d}\partial_{x_i}Q_{\omega_k}\partial_{x_j}Q_{\omega_k}\;dx$.
Due to the continuity of the determinant, $\det(D(t))\to \det D(\infty)$ as $t\to +\infty$. Thus,
\begin{equation}\label{detM_1}
\det M(t)\to \left(\prod_{k=1}^K\int_{\RR^d}Q_{\omega_k}^2(x)\;dx\right)\det(D(\infty)),\qquad\text{as } t\to +\infty.
\end{equation}

\noindent Moreover for all $Y=(y_{i,k})\in\RR^{dK}$ different from $0$, we have
\begin{align*}
^tYD(\infty)Y&=\sum_{i,j=1}^d\sum_{k=1}^K\left(\int_{\RR^d}\partial_{x_i}Q_{\omega_k}\partial_{x_j} Q_{\omega_k}\;dx\right) y_{i,k}y_{j,k}\\
&=\sum_{k=1}^K\int_{\RR^d}\left(\sum_{i=1}^dy_{i,k}\partial_{x_i}Q_{\omega_k}\right)^2\;dx
\end{align*}
which is a positive quantity for large values of $t$ since for all $k$, the $d$ functions $\partial_{x_i}Q_{\omega_k}$, $i=1,\dots,d$ are linearly independent (this can be seen using that $Q_{\omega_k}$ is radial but it is in fact also related to a more general result corresponding to Proposition \ref{indep} in Appendix).
 
\noindent Hence, $\det(D\infty))>0$ and also $\det(M(t))>0$ for large values of $t$ by \eqref{detM_1}. In particular, $M(t)$ is invertible for large values of $t$. Applying Cramer's formula, we obtain an explicit expression of $a_k(t)$ and $b_k(t)$ in terms of $z(t)$, from which we derive the content of Lemma \ref{ex}. Let us justify it. \\
The entries of $M(t)$ are bounded functions of $t$ so that the transpose of the comatrix of $M(t)$ is bounded   too (with respect to $t$). In addition, we have proved the existence of $c>0$ such that for $t$ large, $\det M(t)>c$. Hence, there exists $c_0>0$ such that for all $t$ sufficiently large,
$$|x(t)|\le c_0|y(t)|.$$
This immediately implies \eqref{estcoef}.\\
We moreover observe that the entries of $M(t)$ are $\mathscr{C}^1$ functions of $t$ (by \eqref{exp_decay_Q} and Lebesgue's dominated convergence theorem); in particular $t\mapsto \det M(t)$ is $\mathscr{C}^1$. Then, the differentiability of $a(t)$ and $b(t)$ and estimate \eqref{estdercoef} follow from the differentiability in the sense of distributions (and the expressions of the differentials) of $t\mapsto \Ii \int_{\RR^d}z(t)\overline{R_k}(t)\;dx$ and $t\mapsto \Rr\int_{\RR^d}z(t)\partial_{x_i}(t)\overline{R_k}\;dx$ for $i\in\{1,\dots,d\}$ and $k\in\{1,\dots,K\}$. Let us explain how to show the differentiability of $t\mapsto \int_{\RR^d}z(t,x)\overline{R_k}(t,x)\;dx$. This is essentially due to a density argument and the local well-posedness of \eqref{derz} with continuous dependence on compact sets of time (as for \eqref{NLS}). Let us consider a $\mathscr{C}^1$ function $\phi$ defined for large values of $t$ and with compact support, say included in $[t_0,t_1]$. Since $z(t_0)\in H^1(\RR^d)$, there exists $(z_n(t_0))\in \mathscr{C}^\infty_c(\RR^d)$ converging to $z(t_0)$ in the sense of the $H^1$-norm. The solution $z_n$ of \eqref{derz} with initial data $z_n(t_0)$ in time $t_0$ is defined on $[t_0,t_1]$) for $n$ large, belongs to $\mathscr{C}\left([t_0,t_1],\mathscr{S}(\RR^d)\right)$, and satisfies 
\begin{equation}\label{auxz_n}
\sup_{t\in [t_0,t_1]}\|z_n(t)-z(t)\|_{H^1}\to 0,\qquad\text{as } n\to +\infty.
\end{equation}
Now, by Fubini theorem and the differentiability of $t\mapsto z_n(t,x)R_k(t,x)$ for all $x\in\RR^d$, we obtain
\begin{align*}
\MoveEqLeft[2]
\int_{t_0}^{t_1}\left(\int_{\RR^d}z_n(t,x)R_k(t,x)\;dx\right)\phi '(t)\;dt\\
&=\int_{\RR^d}\left(\int_{t_0}^{t_1}z_n(t,x)R_k(t,x)\phi '(t)\;dt\right)\;dx\\
&=-i\int_{\RR^d}\left(\int_{t_0}^{t_1}\left((\Delta z_n+g(z_n+\phi)-g(\phi))R_k+z_n\partial_tR_k\right)(t,x)\phi(t)\;dt\right)\;dx\\
&=-i\int_{t_0}^{t_1}\phi(t)\left(\int_{\RR^d}\left((\Delta z_n+g(z_n+\phi)-g(\phi))R_k+z_n\partial_tR_k\right)(t,x)\;dx\right)\;dt\\
&=-i\int_{t_0}^{t_1}\phi(t)\left(\int_{\RR^d}\left(\nabla z_n\cdot R_k+(g(z_n+\phi)-g(\phi))R_k+z_n\partial_tR_k\right)(t,x)\;dx\right)\;dt.
\end{align*}
(We recall that $g(z)=zf(|z|^2)$ for $z\in\CC$.) Passing to the limit as $n\to +\infty$ by using \eqref{auxz_n} leads to
\begin{multline}
\int_{t_0}^{t_1}\left(\int_{\RR^d}z(t,x)R_k(t,x)\;dx\right)\phi '(t)\;dt\\
=-i\int_{t_0}^{t_1}\phi(t)\left(\int_{\RR^d}\left(\nabla z\cdot R_k+g(z+\phi)-g(\phi)R_k+z\partial_tR_k\right)(t,x)\;dx\right)\;dt.
\end{multline}
Thus $t\mapsto \int_{\RR^d}z\overline{R_k}(t,x)\;dx$ is differentiable in the sense of distributions; its differential is $$t\mapsto \int_{\RR^d}\left(\nabla z\cdot R_k+(g(z+\phi)-g(\phi))R_k+z\partial_tR_k\right)(t,x)\;dx$$ and is thus bounded by $\|z(t)\|_{H^1}$.\\
This finishes proving the lemma.
\end{proof}

Even if it means taking a larger $T_1$, we can suppose that the preceding lemma holds on $[T_1,+\infty)$.
Then it results also immediately that 
\begin{equation}\label{tild}
\forall\;t\ge T_1,\qquad \|z(t)\|_{H^1}\leq\|\tilde{z}(t)\|_{H^1}+C\sum_{k=1}^K\big(|a_k(t)|+|b_k(t)|\big).
\end{equation}
and 
\begin{equation}\label{tild'}
\forall\;t\ge T_1,\qquad \|\tilde{z}(t)\|_{H^1}\leq C\|z(t)\|_{H^1}.
\end{equation}

\subsubsection{The statement of a coercivity property in terms of the new variable}\label{wein}

In this paragraph, we come to some crucial inequality, on which the proof is essentially based. First of all, let us define some notations, and particularly well-chosen cut-off functions.

By a classical argument given in \cite[Claim 1]{schMM}, we can assume (without loss of generality) that \begin{equation}\label{vit}
v_{1,1}<v_{2,1}<\dots<v_{K,1}.
\end{equation} 
\noindent Now let $A_0\in\Big]0,\frac{1}{2}\min_{k\in\{2,\dots,K\}}\{v_{k,1}-v_{k-1,1}\}\Big[$ and
define \[\begin{array}{cccl}
\psi:&\RR&\rightarrow&\RR\\
&x&\mapsto&\begin{dcases}
1&\text{ if } x<-A_0\\
\left(\int_{-A_0}^{A_0}e^{-\frac{A_0^2}{A_0^2-{y}^2}}\;dy\right)^{-1}\int_x^{A_0}e^{-\frac{A_0^2}{A_0^2-y^2}}\;dy&\text{ if } -A_0\le x\le A_0\\
0&\text{ if } x>A_0,
\end{dcases}
\end{array}\]
which is obviously a smooth bounded non-increasing function. 

\noindent For all $k\in\{1,\dots,K-1\}$, let 
$\sigma_k:=\frac{1}{2}\left(v_{k,1}+v_{k+1,1}\right)$ and $\xi_k:=\frac{1}{2}\left(x^0_{k,1}+x^0_{k+1,1}\right)$. Then define on $\RR\times\RR^d$
\[\begin{array}{cccll}
\psi_0:&(t,x)&\mapsto& 0\\
\psi_k:&(t,x)&\mapsto&\displaystyle\psi\Big(\frac{x_1-\xi_k-\sigma_kt}{t}\Big)&\text{for }k\in\{1,\dots,K-1\}\\ 
\psi_K:&(t,x)&\mapsto& 1
\end{array}\]
and also $K$ functions on $\RR\times \RR^d$ by
\[\forall\;k\in\{1,\dots,K\},\quad \phi_k:=\psi_k-\psi_{k-1}.\]

We can check that, for large values of $t$, $\phi_k(t,\cdot)$ has a smooth profile localized at the "neighborhood" of the $k$-th solitary wave; more precisely we have
\[ \phi_1(t,x)=\begin{dcases}1 &\text{ if  }x_1<\xi_1+(\sigma_1-A_0)t\\
0& \text{ if  } x_1>\xi_{1}+(\sigma_{1}+A_0)t,
\end{dcases}
\]
for all $k=2,\dots,K-1$:
\[ \phi_k(t,x)=\begin{dcases}1 &\text{ if  } \xi_{k-1}+(\sigma_{k-1}+A_0)t<x_1<\xi_k+(\sigma_k-A_0)t\\
0& \text{ if  } x_1<\xi_{k-1}+(\sigma_{k-1}-A_0)t \text{ or } x_1>\xi_{k}+(\sigma_{k}+A_0)t,
\end{dcases}
\]
and
\[ \phi_K(t,x)=\begin{dcases}1 &\text{ if  } \xi_{K-1}+(\sigma_{K-1}+A_0)t<x_1\\
0& \text{ if  } x_1<\xi_{K-1}+(\sigma_{K-1}-A_0)t.
\end{dcases}
\]
Besides, for large values of $t$, the following inequalities hold owing to the decay properties of $R_k$ and the support properties of $\phi_j$ and its derivatives.

\begin{Lem}\label{lem:phij}
Even if it means reducing $\gamma>0$ so that $$\gamma<\min\left\{\frac{\sqrt{\omega_k}}{4}\left(\frac{v_{j,1}-v_{j-1,1}}{2}-A_0\right),\:k=1,\dots,K,\; j=2,\dots,K\right\},$$ we have:
\begin{align}
\label{prel1}
\forall\; j\neq k,\qquad \big(|R_k(t,x)|+|\partial_{x_1}R_k(t,x)|\big)|\phi_j(t,x)| & \le Ce^{-\gamma t}e^{-\frac{\sqrt{\omega_k}}{4}|x-v_kt|} \\
\label{prel2}
\forall\; j,\qquad \big(|R_j(t,x)|+|\partial_{x_1}R_j(t,x)|+|\partial_{t}R_j(t,x)|\big)|\phi_j(t,x)-1| & \le Ce^{-\gamma t}e^{-\frac{\sqrt{\omega_j}}{4}|x-v_jt|} \\\label{prel3}
\forall\; j,\qquad |\partial_{x_1}\phi_j(t,x)|+|\partial^3_{x_1}\phi_j(t,x)|+|\partial_{t}\phi_j(t,x)| & \le \frac{C}{t}. \\
\label{prel4}
\forall\; j,k,\qquad \big(|R_k(t,x)|+|\partial_{x_1}R_k(t,x)|\big)\left|\partial_{x_1}\phi_j(t,x)\right| & \le Ce^{-\gamma t}e^{-\frac{\sqrt{\omega_k}}{4}|x-v_kt|}.
\end{align}
\end{Lem}

\begin{proof}
The proof, postponed in Appendix, is similar to that of Combet \cite[Proof of Lemma 3.9, Appendix A]{combet}. 
\end{proof}

Let us introduce the following Weinstein energy functional which is inspired from Martel, Merle and Tsai \cite{tsai} for dimensions 1 to 3: 
\begin{multline} 
{H}(t):=\sum_{k=1}^K\int_{\RR^d}\left\{|\nabla \tilde{z}|^2-\left(f\left(|R_k|^2\right)|\tilde{z}|^2+2\Rr(\overline{R_k}\tilde{z})^2f'\left(|R_k|^2\right)\right)
\right. \\
\left. +\left(\omega_k+\frac{|v_k|^2}{4}\right)|\tilde{z}|^2-v_k\cdot\Ii\left(\nabla \tilde{z}\overline{\tilde{z}}\right)\right\}\phi_k(t,x)\;dx.
\end{multline}

\noindent One of the main features concerning $H$ is the following coercivity property, which turns out to be a key ingredient in our matter.

\begin{Prop}\label{prop_coerc}
There exists $C>0$ such that 
\begin{equation}\label{coerc}
\forall\; t\ge T_1,\quad C\|\tilde{z}(t)\|_{H^1}^2-\frac{1}{C}\sum_{k=1}^K\left(\Rr\int_{\RR^d}\tilde{z}(t)\overline{R_k}(t)\;dx\right)^2\le H(t).
\end{equation}
\end{Prop}

\begin{proof}
This result follows from our assumption (H3), from \eqref{rel}, and an immediate adaptation to all dimensions of the proof given for the one-dimensional case in \cite[appendix B]{tsai} which consists in localizing in some sense each version of (H3) for all $k=1,\dots, K$.
\end{proof}

\subsection{Proof of some needed estimates}\label{subsec_estimates}

This subsection, which is probably the most technical one, precises the tools and estimates which will allow us to make use of Proposition \ref{prop_coerc} and actually to conclude the proof of uniqueness in subsection \ref{end}. It consists in giving some controls of $H(t)$, of the scalar products $$\Rr\int_{\RR^d}\tilde{z}(t)\overline{R_k}(t)\;dx,$$ and also of the modulation parameters $a_k(t)$ and $b_k(t)$.

\subsubsection{Control of $H$}\label{estenergysect}

We typically improve the a priori control of $H$ by $\mathrm{O}\left(\|\tilde{z}\|_{H^1}^2\right)$ by differentiation of the functional. Actually, for the sake of simplification, we will compute the derivative of the following related functional $\tilde{H}:[T_1,+\infty)\rightarrow \RR$ defined by $\forall\;t\ge T_1$,
\begin{multline}\label{energyprim}
\tilde{H}(t)=\sum_{k=1}^K\int_{\RR^d}\left\{|\nabla \tilde{z}|^2-\left(F(|\tilde{z}+\varphi|^2)-F(|\varphi|^2)-2\Rr(\tilde{z}\overline{\varphi})f(|\varphi|^2)\right) \right. \\
 \left. +\left(\omega_k+\frac{|v_k|^2}{4}\right)|\tilde{z}|^2-v_k\cdot\Ii\left(\nabla \tilde{z}\overline{\tilde{z}}\right)\right\}\phi_k(t,x)\;dx.
\end{multline} 
(Recall that $F(r)=\int_0^rf(\rho)\;d\rho$ \eqref{pointxi}.)

The next proposition, which compares $H$ and $\tilde{H}$, justifies that it suffices to control $\tilde{H}$ in order to obtain a similar estimate for $H$.

\begin{Prop}\label{func} We have
\begin{equation}\label{comp}
H(t)=\tilde{H}(t)+\mathrm{O}\left(\|\tilde{z}(t)\|_{H^1}^3+e^{-\gamma t}\|\tilde{z}(t)\|_{H^1}^2\right).
\end{equation}
\end{Prop}   

\begin{proof}
Let us first observe that $\tilde{F}:z\mapsto F(|z|^2)$ is $\mathscr{C}^3$ on $\CC$. Indeed, since $\tilde{f}$ is $\mathscr{C}^2$ on $\CC$, the function $f$ is $\mathscr{C}^2$ on $(0,+\infty)$ and thus, $\tilde{F}$ is $\mathscr{C}^3$ on $\CC\setminus\{0\}$. Moreover, for all $z=(\Rr(z),\Ii(z))=(x,y)\in\CC\setminus\{0\}$, we obtain by differentiation of $\tilde{f}$ and $\tilde{F}$:
\begin{align*}
\partial_{x}\tilde{F}(z)&=2xf(|z|^2)=2x\tilde{f}(z)\\
\partial_{y}\tilde{F}(z)&=2yf(|z|^2)=2y\tilde{f}(z)\\
\partial_{xx}\tilde{F}(z)&=2f(|z|^2)+4x^2f'(|z|^2)=2\tilde{f}(z)+2x\partial_x\tilde{f}(z)\\
\partial_{xy}\tilde{F}(z)&=4xyf'(|z|^2)=2y\partial_x\tilde{f}(z)\\
\partial_{yy}\tilde{F}(z)&=2f(|z|^2)+4y^2f'(|z|^2)=2\tilde{f}(z)+2y\partial_y\tilde{f}(z)\\
\partial_{xxx}\tilde{F}(z)&=4\partial_x\tilde{f}(z)+2x\partial_{xx}\tilde{f}(z)\\
\partial_{xxy}\tilde{F}(z)&=2y\partial_{xx}\tilde{f}(z)\\
\partial_{xyy}\tilde{F}(z)&=2x\partial_{yy}\tilde{f}(z)\\
\partial_{yyy}\tilde{F}(z)&=4\partial_y\tilde{f}(z)+2y\partial_{yy}\tilde{f}(z).
\end{align*}
Since $\tilde{f}$ is $\mathscr{C}^2$, the partial differentials of $\tilde{F}$ up to order 3 admit limits as $(x,y)\to (0,0)$ in $\RR^2$, from which we deduce that $\tilde{F}$ is $\mathscr{C}^3$.
 
Then we have the following Taylor expansion: for $t\ge T_1$ and $x\in\RR^d$,
\begin{multline} 
\left(\tilde{F}(\tilde{z}+\varphi)-\tilde{F}(\varphi)-\Rr(\tilde{z})\partial_x\tilde{F}(\varphi)-\Ii(\tilde{z})\partial_y\tilde{F}(\varphi)\right)(t,x)\\
=\frac{1}{2}\left((\Rr \tilde{z})^2\partial_{xx}\tilde{F}+2\Rr \tilde{z}\Ii \tilde{z}\partial_{xy}\tilde{F}+(\Ii \tilde{z})^2\partial_{yy}\tilde{F}\right)(t,x)+\rho(t,x),
\end{multline}
where 
\begin{align*}|\rho(t,x)|&\le \frac{1}{6}|\tilde{z}|^3\sup_{\omega\in [\varphi(t,x),(\tilde{z}+\varphi)(t,x)]}\|D^3_{\omega}\tilde{F}\|\\
&\le C|\tilde{z}|^3\sup_{\omega\in [\varphi(t,x),(\tilde{z}+\varphi)(t,x)]}\left(\|D_{\omega}\tilde{f}\|+|\tilde{z}|\|D^2_{\omega}\tilde{f}\|\right)\\
&\le C|\tilde{z}|^3\left(1+\tilde{z}|^{\frac{4}{d}-1}\right)
\end{align*}
by assumption \eqref{prop_croissance}. We then note that the preceding Taylor expansion rewrites
\begin{equation}\label{taylor}
F(|\tilde{z}+\varphi|^2)-F(|\varphi|^2)-2\Rr(\tilde{z}\overline{\varphi})f(|\varphi|^2)=|\tilde{z}|^2f(|\varphi|^2)+2\Rr(\tilde{z}\overline{\varphi})^2f'(|\varphi|^2)+\mathrm{O}(|\tilde{z}|^3+|\tilde{z}|^{\frac{4}{d}+2}),
\end{equation}
uniformly with respect to both variables $t$ and $x$. Let us underline that, for $d\ge 2$, one can not claim whether $z(t)$ or $\tilde{z}(t)$ belong to $L^\infty(\RR^d)$ and even less whether $z$ or $\tilde{z}$ belong to $L^\infty\left([T_1,+\infty),L^{\infty}(\RR^d)\right)$, which prevents us from simplifying $\mathrm{O}\left(|\tilde{z}|^3+|\tilde{z}|^{\frac{4}{d}+2}\right)$ by $\mathrm{O}\left(|\tilde{z}|^3\right)$. \\
\noindent
Moreover, we have noticed that
$$|\tilde{z}|^2f(|\varphi|^2)+2\Rr(\tilde{z}\overline{\varphi})^2f'(|\varphi|^2)=\frac{1}{2}D^2_\varphi \tilde{F}(\tilde{z},\tilde{z})$$
so that
$$
\sum_{k=1}^K\int_{\RR^d}\left\{|\tilde{z}|^2f\left(|\varphi|^2\right)+2\Rr(\tilde{z}\overline{\varphi})^2f'\left(|\varphi|^2\right)\right\}\phi_k\;dx\\
=\sum_{k=1}^K\int_{\RR^d}D^2_{\varphi}\tilde{F}(\tilde{z},\tilde{z})\phi_k\;dx.
$$
We now observe that for all $k=1,\dots,K$,
\begin{align*}
\MoveEqLeft[4]
\left|\int_{\RR^d}\left(D^2_{\varphi}\tilde{F}(\tilde{z},\tilde{z})-D^2_{R_k}\tilde{F}(\tilde{z},\tilde{z})\right)\phi_k\;dx\right|\\
&\le \int_{\RR^d}|(\varphi-R_k)\phi_k||\tilde{z}|^2\sup_{\omega\in [R_k,\varphi]}\|D_{\omega}^3\tilde{F}\|\;dx\\
&\le \int_{\RR^d}|(\varphi-R)\phi_k||\tilde{z}|^2\sup_{\omega\in [R_k,\varphi]}\|D_{\omega}^3\tilde{F}\|\;dx+\sum_{j\ne k}\int_{\RR^d}|R_j\phi_k||\tilde{z}|^2\sup_{\omega\in [R_k,\varphi]}\|D_{\omega}^3\tilde{F}\|\;dx.
\end{align*}
We note that $\sup_{\omega\in [R_k,\varphi]}\|D_{\omega}^3\tilde{F}\|$ is bounded by a constant $C$ independent of $t$ and $x$ because $R_k$, $\varphi$ belong to $L^\infty([T_1,+\infty),L^\infty(\RR^d))$ and $\omega\mapsto D_{\omega}^3\tilde{F}$ is continuous on $\CC$.
Then (\ref{varphi_3d}) and (\ref{prel1}) lead to
\begin{multline}\label{aux}
\sum_{k=1}^K\int_{\RR^d}\left\{|\tilde{z}|^2f\left(|\varphi|^2\right)+2\Rr(\tilde{z}\overline{\varphi})^2f'\left(|\varphi|^2\right)\right\}\phi_k\;dx\\
=\sum_{k=1}^K\int_{\RR^d}\left\{|\tilde{z}|^2f\left(|R_k|^2\right)+2\Rr(\tilde{z}\overline{R_k})^2f'\left(|R_k|^2\right)\right\}\phi_k\;dx
+\mathrm{O}\left(e^{-\gamma t}\|\tilde{z}\|_{H^1}^2\right).
\end{multline}

We finally obtain (\ref{comp}) as a direct consequence of (\ref{taylor}), (\ref{aux}), the Sobolev embeddings $H^1(\RR^d)\hookrightarrow L^3(\RR^d)$ (indeed available for $d\le 3$) and $H^1(\RR^d)\hookrightarrow L^{\frac{4}{d}+2}(\RR^d)$, and the fact that $\frac{4}{d}\ge 1$.

\end{proof}

Now, we state and prove the crucial

\begin{Prop}\label{estenergy}
The derivative of $\tilde{H}$ is given by
\begin{equation}\label{estenergy1}
\frac{d}{dt}\tilde{H}(t)= \text{Main}(t)+\mathrm{O}\big(e^{-\gamma t}\|\tilde{z}(t)\|_{H^1}\|z(t)\|_{H^1}+\|\tilde{z}(t)\|_{H^1}\|z(t)\|_{H^1}^2\big),\qquad \text{as } t\rightarrow +\infty.
\end{equation}
where 
\begin{equation}\label{terme_dominant}
\begin{aligned}
\text{Main}(t):=&\;\displaystyle
\sum_{k=1}^K\left(\omega_k+\frac{|v_k|^2}{4}\right)\left(\int_{\RR^d}|\tilde{z}|^2\partial_t\phi_k\;dx+2\int_{\RR^d}\Ii(\partial_{x_1}\tilde{z}\overline{\tilde{z}})\partial_{x_1}\phi_k\;dx\right)\notag\\
&-\displaystyle\sum_{k=1}^Kv_{k,1}\left(
2\int_{\RR^d}|\nabla \tilde{z}|^2\partial_{x_1}\phi_k\;dx-\frac{1}{2}\int_{\RR^d}|\tilde{z}|^2\partial^3_{x_1}\phi_k\;dx\right)\notag\\
&-\displaystyle\sum_{k=1}^Kv_k\cdot\int_{\RR^d}\Ii\left(\nabla \tilde{z}\overline{\tilde{z}}\right)\partial_{t}\phi_k\;dx\notag\\
=&\;\displaystyle\mathrm{O}\left(\frac{1}{t}\|\tilde{z}(t)\|_{H^1}^2\right).
\end{aligned}
\end{equation}
\end{Prop}

\begin{Rq} \label{rq:mono}
The bound $\mathrm{O}\left(\frac{1}{t}\|\tilde{z}(t)\|_{H^1}^2\right)$ in above is the one which constrains us to prove uniqueness in the class satisfying \eqref{classexpo}. In order to prove unconditionnal uniqueness, one would need to improve this bound to $\mathrm{O}\left(\alpha(t) \|\tilde{z}(t)\|_{H^1}^2\right)$, where $\alpha(t)$ is integrable in time (in the KdV context, \cite{martel} proves it with $\alpha(t) = e^{-\gamma t}$ for some $\gamma>0$).
\end{Rq}

Let us begin with some preliminaries (Lemma \ref{ineqs} and Lemma \ref{dertildez} below), which are needed to obtain  Proposition \ref{estenergy}.

\begin{Lem}\label{ineqs}
There exists $C>0$ such that:
$$\big| f(|z+\varphi|^2)-f(|\varphi|^2)\big|\le C\big(|z|^{\frac{4}{d}}+|z|\big)$$
$$\left|f(|z+\varphi|^2)-f(|\varphi|^2)-2\Rr(z\overline{\varphi})f'(|\varphi|^2)\right|\le C\left(|z|^2+|z|^{\frac{4}{d}}\right).$$
\end{Lem}

\begin{proof}
By the mean value theorem applied to $\tilde{f}$, 
$$\left|\tilde{f}(z+\varphi)-\tilde{f}(\varphi)\right|(t,x)\le |z(t,x)|\sup_{\omega\in [\varphi(t,x),(z+\varphi)(t,x)]}\|D_{\omega}\tilde{f}\|$$
By \eqref{prop_croissance} and the boundedness of $\varphi$ with respect to $t$ and $x$, it results 
$$\left|\tilde{f}(z+\varphi)-\tilde{f}(\varphi)\right|(t,x)\le C\left(|z(t,x)|+|z(t,x)|^{\frac{4}{d}}\right).$$
Similarly, the second estimate stated in Lemma \ref{ineqs} is obtained by direct application of Taylor formula for $\tilde{f}$ at order 2 and by using 
$$\sup_{\omega\in [\varphi(t,x),(z+\varphi)(t,x)]}\|D^2_{\omega}\tilde{f}\|\le C(1+|\tilde{z}|^{\frac{4}{d}-2}).$$
\end{proof}

\begin{Lem}[Expression of $\partial_t\tilde{z}$]\label{dertildez}
We have
\begin{equation}
\begin{aligned}
\partial_t\tilde{z}=& \;i\Big(\Delta\tilde{z}+f(|z+\varphi|^2)z+\big(f(|z+\varphi|^2)-f(|\varphi|^2)\big)\varphi\Big)\\
&+i\displaystyle\sum_{k=1}^K\big\{ia_kf(|R_k|^2)R_k+b_k\cdot\nabla\big(f(|R_k|^2)R_k)\big)\big\}+\sum_{k=1}^K\big\{ia_k'R_k+b_k'\cdot\nabla R_k\big\}\\
=&\;\displaystyle i\Big(\Delta\tilde{z}+f(|\varphi|^2)\tilde{z}+\big(f(|\tilde{z}+\varphi|^2)-f(|\varphi|^2)\big)\varphi\Big)+\displaystyle\sum_{k=1}^K\big\{ia_k'R_k+b_k'\cdot\nabla R_k\big\}+\epsilon,
\end{aligned}
\end{equation}
where $\epsilon$ is a function of $t$ and $x$ such that 
\begin{multline}\label{est_g}
\displaystyle\int_{\RR^d}|\epsilon|\big(|\tilde{z}|+|\nabla\tilde{z}|\big)\;dx+\left|\int_{\RR^d}\epsilon\Delta\tilde{z}\;dx\right|\\
\le C\big(e^{-\gamma t}\|z\|_{H^1}\|\tilde{z}\|_{H^1}+\|z\|_{H^1}^2\|\tilde{z}\|_{H^1}\big).
\end{multline}
\end{Lem}

\begin{proof}
The first equality concerning $\partial_t\tilde{z}$ is quite immediate. Let us precise how to obtain the second equality. Decomposing
\[
f(|z+\varphi|^2)z=f(|\varphi|^2)z+\big(f(|z+\varphi|^2)-f(|\varphi|^2)\big)z,
\]
and using the expression of $z$ in terms of $\tilde{z}$ and the $a_k$ and $b_k$, $k=1,\dots, K$ given by \eqref{def_tilde_z}, we have that 
\begin{equation}\label{lem_dertildez_1}
\begin{aligned}
\MoveEqLeft[4] f(|z+\varphi|^2)z+\sum_{k=1}^K\left\{ia_kf(|R_k|^2)R_k+b_k\cdot\nabla\left(f(|R_k|^2)R_k)\right)\right\} \\
=&\;f(|\varphi|^2)\tilde{z}+\displaystyle\sum_{k=1}^K\big\{ia_k\big(f(|R_k|^2)-f(|\varphi|^2)\big)R_k+b_k\cdot\nabla R_k\big(f(|R_k|^2)-f(|\varphi|^2)\big)\big\}\\
&+\displaystyle\sum_{k=1}^Kb_k\cdot\nabla(f(|R_k|^2))R_k+\big(f(|z+\varphi|^2)-f(|\varphi|^2)\big)z
\end{aligned}
\end{equation}

\noindent Moreover, using the second estimate obtained in Lemma \ref{ineqs}, 
\begin{equation}\label{lem_dertildez_2}
\begin{aligned}
\MoveEqLeft[4]
f(|z+\varphi|^2)-f(|\varphi|^2)\\
&=f(|\tilde{z}+\varphi|^2)-f(|\varphi|^2)-2\Rr\big((\tilde{z}-z)\overline{\varphi}\big)f'(|\varphi|^2)+\mathrm{O}(|z|^2+|z|^{\frac{4}{d}})\\
&=\displaystyle f(|\tilde{z}+\varphi|^2)-f(|\varphi|^2)-2\sum_{k=1}^K\Rr\big(b_k\cdot\nabla R_k\overline{R_k}\big)f'(|R_k|^2)+h+\mathrm{O}(|z|^2+|z|^{\frac{4}{d}})\\
&=\displaystyle f(|\tilde{z}+\varphi|^2)-f(|\varphi|^2)-\sum_{k=1}^Kb_k\cdot\nabla\big(f(|R_k|^2\big)+\tilde{h},
\end{aligned}
\end{equation}
where $h$ and $\tilde{h}$ satisfy the same property (\ref{est_g}) as $\epsilon$ due to (\ref{varphi_3d}) and (\ref{prel1}). Lemma \ref{dertildez} is now a consequence of \eqref{lem_dertildez_1}, \eqref{lem_dertildez_2}, and the fact that $\frac{4}{d}+1\ge 2$ (for $d\le 3$).
\end{proof}

We are now in a position to prove \eqref{estenergy1}.

\begin{proof}[Proof of Proposition \ref{estenergy}]
The proof decomposes essentially into two parts. We first differentiate successively each term constituting $\tilde{H}$ by means of Lemma \ref{dertildez}. For this, integrations by parts are sometimes necessary in order not to keep terms carrying second spatial derivatives for $z$. Then we put together suitable terms in the expression of $\frac{d}{dt}\tilde{H}$ in order to get better estimates than the a priori control by $\mathrm{O}\left(\|\tilde{z}(t)\|_{H^1}^2\right)$. Besides, we put annotations for the different terms we have to work on for ease of reading; terms associated with the same letter $\mathrm{A}$, $\mathrm{B}$, or $\mathrm{C}$ are to be gathered. \\

\underline{Step 1}: Differentiation of $\tilde{H}$
\begin{itemize}
\item
Using Lemma \ref{dertildez}, 
one computes 
\begin{align*}
\MoveEqLeft[2]
\displaystyle\frac{d}{dt}\int_{\RR^d}|\nabla \tilde{z}|^2\;dx\\
&=\displaystyle \;2\;\Rr\int_{\RR^d}\nabla \tilde{z}_t\cdot\nabla\overline{\tilde{z}}\;dx\\
&=\displaystyle \;2\;\Ii\; \int_{\RR^d}f(|\varphi|^2)\tilde{z}\Delta\overline{\tilde{z}}\;dx+2\displaystyle\;\Ii\; \int_{\RR^d}\big(f(|\tilde{z}+\varphi|^2)-f(|\varphi|^2)\big)\varphi\Delta\overline{\tilde{z}}\;dx\\
&\:\:\:-2\displaystyle\;\Rr\int_{\RR^d}\sum_{k=1}^K\big\{ia_k'R_k+b_k'\cdot\nabla R_k\big\}\Delta\overline{\tilde{z}}\;dx+\mathrm{O}\big((e^{-\gamma t}+\|z\|_{H^1})\|z\|_{H^1}\|\tilde{z}\|_{H^1}\big).
\end{align*}

Similarly one obtains directly
\begin{align*}
\MoveEqLeft[4]
\displaystyle\frac{d}{dt}\int_{\RR^d}\Big\{F(|\tilde{z}+\varphi|^2)-F(|\varphi|^2)-2\Rr(\tilde{z}\overline{\varphi})f(|\varphi|^2)\Big\}\;dx\\
=&\;\displaystyle \;2\;\Rr\int_{\RR^d}\varphi_t\overline{\varphi}\big(f(|\tilde{z}+\varphi|^2)-f(|\varphi|^2)-2\Rr(\tilde{z}\overline{\varphi})f'(|\varphi|^2)\big)\;dx\\
&\displaystyle+2\;\Rr\int_{\RR^d}\tilde{z}_t\overline{\tilde{z}}f(|\tilde{z}+\varphi|^2)\:dx+2\;\Rr\int_{\RR^d}(\varphi_t\overline{\tilde{z}}+\tilde{z}_t\overline{\varphi})\big(f(|\tilde{z}+\varphi|^2)-f(|\varphi|^2)\big)\;dx\\
=&-\displaystyle 2\;\Ii\; \int_{\RR^d}\Delta \tilde{z}\overline{\tilde{z}}f(|\tilde{z}+\varphi|^2)\;dx\\
&-2\displaystyle\;\Ii\; \int_{\RR^d}\Delta \varphi\overline{\varphi}\big(f(|\tilde{z}+\varphi|^2)-f(|\varphi|^2)-2\Rr(\tilde{z}\overline{\varphi})f'(|\varphi|^2)\big)\;dx\\
&-2\displaystyle\;\Ii\int_{\RR^d}\big(f(|\tilde{z}+\varphi|^2)-f(|\varphi|^2)\big)\varphi\overline{\tilde{z}}(f(|\tilde{z}+\varphi|^2)\;dx\\
&-2\displaystyle\;\Ii\int_{\RR^d}\Delta\tilde{z}\overline{\varphi}\big(f(|\tilde{z}+\varphi|^2)-f(|\varphi|^2)\big)\;dx\\
&-2\displaystyle\;\Ii\int_{\RR^d}\big(\Delta\varphi+f(|\varphi|^2)\varphi\big)\overline{\tilde{z}}\big(f(|\tilde{z}+\varphi|^2)-f(|\varphi|^2)\big)\;dx\\
&-2\displaystyle\;\Ii\int_{\RR^d}f(|\varphi|^2)\tilde{z}\overline{\varphi}\big(f(|\tilde{z}+\varphi|^2)-f(|\varphi|^2)\big)\;dx\\
&+2\displaystyle\;\Rr\int_{\RR^d}\sum_{k=1}^K\big(ia_k'R_k+b_k'\cdot \nabla R_k\big)\left(\overline{\tilde{z}}f(|\tilde{z}+\varphi|^2)+\overline{\varphi}\big(f(|\tilde{z}+\varphi|^2)-f(|\varphi|^2)\big)\right)\;dx\\
&+\mathrm{O}\big((e^{-\gamma t}+\|z\|_{H^1})\|z\|_{H^1}\|\tilde{z}\|_{H^1}\big).
\end{align*}

Thus, we have at this point
\begin{align}
\MoveEqLeft[4]
\displaystyle\frac{d}{dt}\int_{\RR^d}\big\{|\nabla \tilde{z}|^2-\big(F(|\tilde{z}+\varphi|^2)-F(|\varphi|^2)-2\Rr(\tilde{z}\overline{\varphi})f(|\varphi|^2)\big)\big\}\;dx \nonumber \\
=&\;\displaystyle 2\;\displaystyle\Ii \int_{\RR^d}\Delta\varphi\overline{\varphi}\big(f(|\tilde{z}+\varphi|^2)-f(|\varphi|^2)-2\Rr(\tilde{z}\overline{\varphi})f'(|\varphi|^2)\big)\;dx \tag{$B_1$} \\
&+2\;\displaystyle\Ii \int_{\RR^d}\big(\Delta\varphi+f(|\varphi|^2)\varphi\big)\overline{\tilde{z}}\big(f(|\tilde{z}+\varphi|^2)-f(|\varphi|^2)\big)\;dx \tag{$A_1$} \\
&-2\;\displaystyle\Rr\int_{\RR^d}\sum_{k=1}^K\big(ia_k'R_k+b_k'\cdot\nabla R_k\big)\big(\Delta\overline{\tilde{z}}+\overline{\tilde{z}}f(|\tilde{z}+\varphi|^2)\big)\;dx \tag{$C_1$} \\
&-2\;\displaystyle\Rr\sum_{k=1}^K\int_{\RR^d}b_k'\cdot\nabla R_k\overline{R_k}\big(f(|\tilde{z}+\varphi|^2)-f(|\varphi|^2)\big)\;dx \tag{$C_2$} \\
&+\;\mathrm{O}\left(\left(e^{-\gamma t}+\|z\|_{H^1}\right)\|z\|_{H^1}\|\tilde{z}\|_{H^1}\right).\notag
\end{align}

\item We differentiate then the next term appearing in the expression of $\tilde{H}$. For all $k=1,\dots,K$,
\begin{align}
\MoveEqLeft[4]\displaystyle\left(\omega_k+\frac{|v_k|^2}{4}\right)\frac{d}{dt}\int_{\RR^d}|\tilde{z}|^2\phi_k\;dx\notag\\
=&\displaystyle\left(\omega_k+\frac{|v_k|^2}{4}\right)\int_{\RR^d}|\tilde{z}|^2\partial_t\phi_k\;dx+2\left(\omega_k+\frac{|v_k|^2}{4}\right)\int_{\RR^d}\Ii\left(\partial_{x_1}\tilde{z}\overline{\tilde{z}}\right)\partial_{x_1}\phi_k\;dx\tag{$Main_{1,k}$}\\
&-2\displaystyle\left(\omega_k+\frac{|v_k|^2}{4}\right)\int_{\RR^d}\Ii(\varphi\overline{\tilde{z}})(f(\tilde{z}+\varphi|^2)-f(|\varphi|^2))\phi_k\;dx\tag{$A_{2,k}$}\\
&+\;2\displaystyle\left(\omega_k+\frac{|v_k|^2}{4}\right)\Rr\int_{\RR^d}\big(ia_k'R_k+b_k'\cdot\nabla R_k\big)\overline{\tilde{z}}\phi_k\;dx\tag{$C_{3,k}$}\\
&+\mathrm{O}\big((e^{-\gamma t}+\|z\|_{H^1})\|z\|_{H^1}\|\tilde{z}\|_{H^1}\big).\notag
\end{align}

\item To finish with, using integrations by parts and \eqref{prel4}, we obtain for all $k=1,\dots,K$,
\begin{align}
\MoveEqLeft[4]\displaystyle\frac{d}{dt}\int_{\RR^d}\Ii(v_k\cdot\nabla \tilde{z}\overline{\tilde{z}})\phi_k\;dx\notag\\
=&-2\;\displaystyle v_k\cdot\Ii\int_{\RR^d}\tilde{z}_t\nabla\overline{\tilde{z}}\phi_k\;dx-v_k\cdot\Ii\int_{\RR^d}\tilde{z}_t\overline{\tilde{z}}\nabla\phi_k\;dx+v_k\cdot\Ii\int_{\RR^d}\nabla{\tilde{z}}\overline{\tilde{z}}\partial_t\phi_k\;dx\notag\\
=&\;\displaystyle 2\;v_{k,1}\int_{\RR^d}\left|\nabla \tilde{z}\right|^2\partial_{x_1}\phi_k\;dx-\frac{v_{k,1}}{2}\int_{\RR^d}|\tilde{z}|^2\partial^3_{x_1}\phi_k\;dx+v_k\cdot\int_{\RR^d}\Ii\left(\nabla \tilde{z}\overline{\tilde{z}}\right)\partial_{t}\phi_k\;dx\tag{$Main_{2,k}$}\\
&+\displaystyle 2\;v_k\cdot\int_{\RR^d}\Rr\left(\nabla\varphi\overline{\tilde{z}})\big(f(|\tilde{z}+\varphi|^2)-f(|\varphi|^2)\right)\phi_k\;dx\tag{$A_{3,k}$}\\
&\displaystyle+\;v_k\cdot\int_{\RR^d}\nabla(f(|\varphi|^2))|\tilde{z}|^2\phi_k\;dx\tag{$B_{2,k}$}\\
&-2\displaystyle\;v_k\cdot\Rr\int_{\RR^d}\nabla\left(\varphi\overline{\tilde{z}}\right)\left(f(|\tilde{z}+\varphi|^2)-f(|\varphi|^2)\right)\phi_k\;dx\tag{$B_{3,k}$}\\
&-2\;v_k\cdot\displaystyle\;\Ii\int_{\RR^d}\big(ia_k'R_k+b_k'\cdot\nabla R_k\big)\nabla\overline{\tilde{z}}\phi_k\;dx\tag{$C_{4,k}$}\\
&+\mathrm{O}\big((e^{-\gamma t}+\|z\|_{H^1})\|z\|_{H^1}\|\tilde{z}\|_{H^1}\big).\notag
\end{align}

\end{itemize}

We now continue the proof by showing how the corresponding terms put together can yield estimation (\ref{estenergy1}). \\

\underline{Step 2}: Estimate concerning $\tilde{H}'$\\

We first deal with the terms $A_1$, $A_{2,k}$, and $A_{3,k}$ ($k=1,\dots,K$). We see that 

\begin{equation}\label{A1'}
\begin{aligned}
\MoveEqLeft[4]
\left|A_1-2\sum_{k=1}^K\Ii\int_{\RR^d}\big(\Delta R_k+f(|R_k|^2)R_k\big)\overline{\tilde{z}}\big(f(|\tilde{z}+\varphi|^2)-f(|\varphi|^2)\big)\;dx\right|\\
\leq&
\displaystyle \;2\;\left|\int_{\RR^d}\Delta(\varphi-R)\overline{\tilde{z}}\left(f(|\tilde{z}+\varphi|^2)-f(|\varphi|^2)\right)\;dx\right|\\
&+\displaystyle\int_{\RR^d}\big|f(|\varphi|^2)-f(|R|^2)\big|\big|\varphi\overline{\tilde{z}}\big|\big|f(|\tilde{z}+\varphi|^2)-f(|\varphi|^2)\big|\;dx\\
&+\displaystyle\int_{\RR^d}\big|f(|R|^2)(\varphi-R)\overline{\tilde{z}}\big|\big|f(|\tilde{z}+\varphi|^2)-f(|\varphi|^2)\big|\;dx\\
&+\displaystyle\sum_{k=1}^K\int_{\RR^d}\left|\left(f(|R|^2)-f(|R_k|^2)\right)R_k\right|\left|\overline{\tilde{z}}\right|\big|f(|\tilde{z}+\varphi|^2)-f(|\varphi|^2)\big|\;dx.
\end{aligned}
\end{equation}
As for the proof of Lemma \ref{ineqs}, by the mean value theorem, we observe also that
\begin{equation}\label{A1''}
\begin{dcases}
\big|f(|\varphi|^2)-f(|R|^2)\big|\le C\left(|\varphi-R|^{\frac{4}{d}}+|\varphi-R|\right)\\
\big|\left(f(|R|^2)-f(|R_k|^2)\right)R_k\big|\le C\sum_{j\neq k}|R_jR_k|.
\end{dcases}
\end{equation}
Moreover, we deduce from Lemma \ref{ineqs} that 

\begin{equation}\label{A1'''}
\begin{aligned}
\left|\int_{\RR^d}\Delta(\varphi-R)\overline{\tilde{z}}\left(f(|\tilde{z}+\varphi|^2)-f(|\varphi|^2)\right)\;dx\right|&\leq Ce^{-\gamma t}\big(\|\tilde{z}\|_{H^1}^{\frac{4}{d}+1}+\|\tilde{z}\|_{H^1}^2\big)\\
&\le Ce^{-\gamma t}\|\tilde{z}\|_{H^1}^2.
\end{aligned}
\end{equation}

Let us establish the preceding inequality in each dimension $d=1,2,3$.

\begin{itemize}

\item For $d=1$, 
we firstly make use of one integration by parts:
\begin{align*}
\MoveEqLeft[4]
\int_{\RR^d}\Delta(\varphi-R)\overline{\tilde{z}}\left(f(|\tilde{z}+\varphi|^2)-f(|\varphi|^2)\right)\;dx\\
&=-\int_{\RR^d}\nabla(\varphi-R)\cdot\nabla\overline{\tilde{z}}\left(f(|\tilde{z}+\varphi|^2)-f(|\varphi|^2)\right)\;dx\\
&\:\:-\int_{\RR^d}\nabla(\varphi-R)\cdot\nabla\left(f(|\tilde{z}+\varphi|^2)-f(|\varphi|^2)\right)\overline{\tilde{z}}\;dx.
\end{align*}
We thus obtain by Lemma \ref{ineqs}:
\begin{multline}
\left|\int_{\RR^d}\Delta(\varphi-R)\overline{\tilde{z}}\left(f(|\tilde{z}+\varphi|^2)-f(|\varphi|^2)\right)\;dx\right|\\
\le C\int_{\RR^d}|\nabla(\varphi-R)||\nabla \tilde{z}|(|z|+|z|^{\frac{4}{d}})+C\int_{\RR^d}|\nabla(\varphi-R)||\nabla \varphi|(|z|+|z|^{\frac{4}{d}})\;dx.
\end{multline}
Then we note that for all $\psi\in H^1(\RR^d)$, (by the embedding $H^1(\RR)\hookrightarrow L^\infty(\RR)$)
\begin{align*}
\int_{\RR^d}|\nabla\psi||\nabla \tilde{z}||\tilde{z}|\;dx & \le \|\nabla \tilde{z}\|_{L^2}\|\nabla \psi\|_{L^2}\|\tilde{z}\|_{L^\infty}\le C\|\tilde{z}\|_{H^1}^2\|\psi\|_{H^1}, \\
\int_{\RR^d}|\nabla\tilde{z}||\nabla\varphi||\tilde{z}|^{\frac{4}{d}}\;dx & \le \|\tilde{z}\|_{H^1}^{\frac{4}{d}}\|\nabla \tilde{z}\|_{L^2}\|\nabla \varphi\|_{L^2}\le C\|\tilde{z}\|_{H^1}^{\frac{4}{d}+1}\|\varphi\|_{H^1},\\
\int_{\RR^d}|\nabla\psi||\nabla\varphi||\tilde{z}|^{\frac{4}{d}}\;dx & \le \|\tilde{z}\|_{H^1}^{\frac{4}{d}}\|\nabla \psi\|_{L^2}\|\nabla \varphi\|_{L^2}\le C\|\tilde{z}\|_{H^1}^{\frac{4}{d}}\|\psi\|_{H^1}\|\varphi\|_{H^1}.
\end{align*}

\item For $d=2$: for all $\psi\in H^2(\RR^d)$, (by the embeddings $H^1(\RR^2)\hookrightarrow L^q(\RR^2)$ for each $q\in[2,+\infty)$)
\begin{align*}
\int_{\RR^d}|\nabla\psi||\nabla \tilde{z}||\tilde{z}|\;dx & \le \|\nabla \tilde{z}\|_{L^2}\|\nabla \psi\|_{L^4}\|\tilde{z}\|_{L^4}\le C\|\tilde{z}\|_{H^1}^2\|\psi\|_{H^2}, \\
\int_{\RR^d}|\nabla\psi||\nabla\tilde{z}||\tilde{z}|^{\frac{4}{d}}\;dx & \le \|\nabla\tilde{z}\|_{L^2}\|\tilde{z}\|_{L^{\frac{16}{d}}}^{\frac{4}{d}}\|\nabla \psi\|_{L^4}\le  C\|\tilde{z}\|_{H^1}^{\frac{4}{d}+1}\|\psi\|_{H^2}\\
\int_{\RR^d}|\nabla\psi||\nabla\varphi||\tilde{z}|^{\frac{4}{d}}\;dx & \le \|\tilde{z}\|_{L^{\frac{8}{d}}}^{\frac{4}{d}}\|\nabla \psi\|_{L^4}\|\nabla \varphi\|_{L^4}\le C\|\tilde{z}\|_{H^1}^{\frac{4}{d}}\|\psi\|_{H^2}\|\varphi\|_{H^2}.
\end{align*}

\item For $d=3$, $|z|+|z|^{\frac{4}{d}}\le |z|+|z|^{\frac{4}{3}}\le 2(|z|+|z|^2)$. We have for all $\psi\in H^2(\RR^d)$, (by the embedding $H^1(\RR^3)\hookrightarrow L^6(\RR^3)$)

\begin{align*}
\int_{\RR^d}|\nabla\psi||\nabla \tilde{z}||\tilde{z}|\;dx & \le \|\nabla \tilde{z}\|_{L^2}\|\nabla \psi\|_{L^4}\|\tilde{z}\|_{L^4}\le C\|\tilde{z}\|_{H^1}^2\|\psi\|_{H^2}, \\
\int_{\RR^d}|\nabla\psi||\nabla\tilde{z}||\tilde{z}|^{2}\;dx & \le \|\nabla\tilde{z}\|_{L^2}\|\tilde{z}\|_{L^6}^{2}\|\nabla \psi\|_{L^6}\le  C\|\tilde{z}\|_{H^1}^{3}\|\psi\|_{H^2}\\
\int_{\RR^d}|\nabla\psi||\nabla\varphi||\tilde{z}|^{2}\;dx & \le \|\nabla\psi\|_{L^2}\|\tilde{z}\|_{L^6}^{3}\|\nabla \varphi\|_{L^6}\le C\|\tilde{z}\|_{H^1}^{2}\|\psi\|_{H^2}\|\varphi\|_{H^2}.
\end{align*}

\end{itemize}

\noindent Hence, gathering \eqref{A1'}, \eqref{A1''}, and\eqref{A1'''}, using \eqref{varphi_3d} and the fact that $\frac{4}{d}+1\ge 2$, it results
\begin{equation}\label{A1}
\left|A_1-2\sum_{k=1}^K\Ii\int_{\RR^d}\big(\Delta R_k+f(|R_k|^2)R_k\big)\overline{\tilde{z}}\big(f(|\tilde{z}+\varphi|^2)-f(|\varphi|^2)\big)\;dx\right|\le Ce^{-\gamma t}\left(\|\tilde{z}\|_{H^1}^2\right).
\end{equation}
In a similar way, for all $k=1,\dots,K$,
\begin{equation}\label{A2}
\begin{aligned}
\MoveEqLeft[6]\left|A_{2,k}+2\left(\omega_k+\frac{|v_k|^2}{4}\right)\int_{\RR^d}\Ii(R_k\overline{\tilde{z}})\big(f(|\tilde{z}+\varphi|^2)-f(|\varphi|^2)\big)\phi_k\;dx\right|\\
\leq &\;\displaystyle 2\left(\omega_k+\frac{|v_k|^2}{4}\right)\int_{\RR^d}\big|(\varphi-R)\overline{\tilde{z}}\big(f(|\tilde{z}+\varphi|^2)-f(|\varphi|^2)\big)\phi_k\big|\;dx\\
&\;+\displaystyle 2\left(\omega_k+\frac{|v_k|^2}{4}\right)\sum_{j\neq k}\int_{\RR^d}\big|R_j\overline{\tilde{z}}\big(f(|\tilde{z}+\varphi|^2)-f(|\varphi|^2)\big)\phi_k\big|\;dx\\
\leq&\;C\displaystyle\int_{\RR^d}\big|\varphi-R\big|\big(|\tilde{z}|^{\frac{4}{d}+1}+|\tilde{z}|^2\big)\;dx+ C\displaystyle\sum_{j\neq k}\int_{\RR^d}\big|R_j\phi_k\big|\big(|\tilde{z}|^{\frac{4}{d}+1}+|\tilde{z}|^2\big)\;dx\\
\leq &\;Ce^{-\gamma t}\|\tilde{z}\|_{H^1}^2.
\end{aligned}
\end{equation}
We have then for all $k=1,\dots,K$,
\begin{equation}\label{A3}
\begin{aligned}
\MoveEqLeft[4]
\left|A_{3,k}-2v_k\cdot\int_{\RR^d}\Rr(\nabla R_k\overline{\tilde{z}})\big(f(|\tilde{z}+\varphi|^2)-f(|\varphi|^2)\big)\phi_k\;dx\right|\\
&\leq 
C\displaystyle\int_{\RR^d}|\nabla(\varphi-R)|\big(|\tilde{z}|^p+|\tilde{z}|^2\big)\phi_k\;dx+ C\displaystyle\sum_{j\neq k}\int_{\RR^d}\big|\nabla R_j\phi_k\big|\big(|\tilde{z}|^{\frac{4}{d}+1}+|\tilde{z}|^2\big)\;dx\\
&\leq Ce^{-\gamma t}\|\tilde{z}\|_{H^1}^2.
\end{aligned}
\end{equation}
Let us gather \eqref{A1}, \eqref{A2}, and \eqref{A3}. Observing that 
\begin{equation}\label{derRk}
\begin{dcases}
\partial_tR_k=-v_k\cdot \nabla R_k+i\left(\omega_k+\frac{|v_k|^2}{4}\right)R_k\\
\partial_tR_k=i\big(\Delta R_k+f(|R_k|^2)R_k\big),\\
\end{dcases}
\end{equation}
we notice that \[\Ii\big((\Delta R_k+f(|R_k|^2)R_k)\overline{\tilde{z}}\big)-\left(\omega_k+\frac{|v_k|^2}{4}\right)\Ii\left(R_k\phi_k\overline{\tilde{z}}\right)-v_k\cdot\Rr\left(\nabla R_k\phi_k\overline{\tilde{z}}\right)=-\Rr\left(\partial_tR_k(1-\phi_k)\overline{\tilde{z}}\right).\] As a consequence of (\ref{prel2}), we obtain a control of 
\begin{align*}
\MoveEqLeft[4]\int_{\RR^d}\bigg\{2\Ii\big((\Delta R_k+f(|R_k|^2)R_k)\overline{\tilde{z}}\big)-2\left(\omega_k+\frac{|v_k|^2}{4}\right)\Ii\left(R_k\phi_k\overline{\tilde{z}}\right)\\
&-2v_k\cdot\Rr\left(\nabla R_k\phi_k\overline{\tilde{z}}\right)\big(f(|\tilde{z}+\varphi|^2)-f(|\varphi|^2)\big)\bigg\}\;dx
\end{align*} by $Ce^{-\gamma t}\|\tilde{z}\|_{H^1}^2.$

\noindent Finally, 
\begin{equation}\label{estA}
\left|A_1+\sum_{k=1}^K\left\{A_{2,k}- A_{3,k}\right\}\right|\le Ce^{-\gamma t}\|\tilde{z}\|_{H^1}^2.
\end{equation}

Let us focus now on the terms identified by the letter $B$. We observe that 
\begin{equation}\label{B2k}
B_{2,k}=-v_k\cdot\int_{\RR^d}f(|\tilde{z}+\varphi|^2)\nabla(|\tilde{z}|^2)\phi_k\;dx+\mathrm{O}\left((e^{-\gamma t}+\|z\|_{H^1})\|z\|_{H^1}\|\tilde{z}\|_{H^1}\right).
\end{equation}

\noindent Then, we obtain 
\begin{equation}\label{B3k}
\begin{aligned}
B_{2,k}+B_{3,k}=&-v_k\cdot\Rr\int_{\RR^d}\nabla\left(|\tilde{z}+\overline{\varphi}|^2\right)f(|\tilde{z}+\overline{\varphi}|^2)\phi_k\;dx +v_k\cdot\Rr\int_{\RR^d}\nabla\left(|\varphi|^2\right)f(|\tilde{z}+\overline{\varphi}|^2)\phi_k\;dx\\
&+2v_k\cdot\Rr\int_{\RR^d}\nabla(\varphi\overline{\tilde{z}})f(|\varphi|^2)\phi_k\;dx +\mathrm{O}\left((e^{-\gamma t}+\|z\|_{H^1})\|z\|_{H^1}\|\tilde{z}\|_{H^1}\right).
\end{aligned}
\end{equation}

Notice next that $$v_k\cdot\Rr\left(\nabla R_k\overline{R_k}\right)=\Ii\left(\Delta R_k\overline{R_k}\right),$$ which allows us to rewrite \[B_1-2\;\sum_{k=1}^Kv_k\cdot\Rr\int_{\RR^d}\nabla R_k\overline{R_k}\big(f(|\tilde{z}+\varphi|^2)-f(|\varphi|^2)-2\Rr(\tilde{z}\overline{\varphi})f'(|\varphi|^2)\big)\;dx\] as a sum of quantities in which the differences between $\varphi$ and $R$ or the products $\nabla R_k\overline{R_j}$ for $j\neq k$ appear. Use moreover 
\begin{enumerate}
\item on the one hand, the second estimate proven in Lemma \ref{ineqs}
\item on the other, the following inequalities:

\begin{itemize}
\item for $d=1$: for all $\psi\in H^{1}(\RR^d)$,
\begin{align*}
\int_{\RR^d}|\nabla\psi||\nabla \tilde{z}||\tilde{z}|\;dx & \le \|\nabla \tilde{z}\|_{L^2}\|\nabla \psi\|_{L^2}\|\tilde{z}\|_{L^\infty}\le C\|\tilde{z}\|_{H^1}^2\|\psi\|_{H^1}, \\
\int_{\RR^d}|\nabla\psi|^2|\tilde{z}|^2\;dx & \le \|\tilde{z}\|_{H^1}^2\|\nabla \psi\|_{L^2}^2 \le C\|\tilde{z}\|_{H^1}^2\|\psi\|_{H^1}^2, \\
\int_{\RR^d}|\nabla\psi||\nabla\varphi||\tilde{z}|^{\frac{4}{d}}\;dx & \le \|\tilde{z}\|_{H^1}^{\frac{4}{d}}\|\nabla \psi\|_{L^2}\|\nabla \varphi\|_{L^2}\le C\|\tilde{z}\|_{H^1}^{\frac{4}{d}}\|\psi\|_{H^1}\|\varphi\|_{H^1}, \\
\end{align*}
\item for $d\in\{2,3\}$: for all $\psi\in H^2(\RR^d)$,
\begin{align*}
\int_{\RR^d}|\nabla\psi||\nabla \tilde{z}||\tilde{z}|\;dx & \le \|\nabla \tilde{z}\|_{L^2}\|\nabla \psi\|_{L^4}\|\tilde{z}\|_{L^4}\le C\|\tilde{z}\|_{H^1}^2\|\psi\|_{H^2}, \\
\int_{\RR^d}|\nabla\psi|^2|\tilde{z}|^2\;dx & \le \|\tilde{z}\|_{L^4}^2\|\nabla \psi\|_{L^4}^2\le C\|\tilde{z}\|_{H^1}^2\|\psi\|_{H^2}^2, \\
\int_{\RR^d}|\nabla\psi||\nabla\varphi||\tilde{z}|^4\;dx & \le \|\tilde{z}\|_{L^6}^{4}\|\nabla \psi\|_{L^6}\|\nabla \varphi\|_{L^6}\le C\|\tilde{z}\|_{H^1}^4\|\psi\|_{H^2}\|\varphi\|_{H^2}.
\end{align*}
\end{itemize}
\end{enumerate}

\begin{Rq}
Regarding the higher dimensions in order to prove Proposition \ref{uni:4d}, one would make use of the following inequality, valid for $d\ge 3$: for all $\psi\in H^{\lfloor\frac{d}{2}\rfloor+1}(\RR^d)$, for all $\tilde{\psi}\in H^{1}(\RR^d)$,
$$
\int_{\RR^d}|\nabla\psi||\nabla \tilde{\psi}||\tilde{z}|\;dx \le \|\tilde{z}\|_{L^{\frac{2d}{d-2}}}\|\nabla \tilde{\psi}\|_{L^2}\|\nabla\psi\|_{L^d}\le C\|\tilde{z}\|_{H^1}\|\tilde{\psi}\|_{H^2}\|\psi\|_{H^{\lfloor\frac{d}{2}\rfloor+1}}. 
$$
\end{Rq}

\noindent Then we conclude that
\begin{equation}\label{B1}
\begin{aligned}
\MoveEqLeft\left|B_1-2\;\sum_{k=1}^Kv_k\cdot\Rr\int_{\RR^d}\nabla R_k\overline{R_k}\big(f(|\tilde{z}+\varphi|^2)-f(|\varphi|^2)-2\Rr(\tilde{z}\overline{\varphi})f'(|\varphi|^2)\big)\;dx\right| \\
& \le Ce^{-\gamma t}\big(\|\tilde{z}\|_{H^1}^{\frac{4}{d}+1}+\|\tilde{z}\|_{H^1}^2\big) \le Ce^{-\gamma t}\|\tilde{z}\|_{H^1}^2.
\end{aligned}
 \end{equation}

\noindent Thus,  from \eqref{B3k} and \eqref{B1}, we deduce that:
\begin{align}\label{estB}
\MoveEqLeft[4]
\left|B_1-\displaystyle\sum_{k=1}^K\big(B_{2,k}+B_{3,k}\big)\right|\\
&\leq \: Ce^{-\gamma t}\|\tilde{z}\|_{H^1}^2+C\left((e^{-\gamma t}+\|z\|_{H^1})\|z\|_{H^1}\|\tilde{z}\|_{H^1}\right) \notag\\
&\:\:\:+\sum_{k=1}^K\left|v_k\cdot\Rr\int_{\RR^d}\nabla\left(|R_k|^2-|\varphi|^2\right)f(|\tilde{z}+\varphi|^2)\phi_k\;dx\right| \notag\\
&\:\:\:+\sum_{k=1}^K\left|-v_k\cdot\int_{\RR^d}\nabla\left(|\varphi|^2\right)f(|\varphi|^2)\;dx+v_k\cdot\Rr\int_{\RR^d}\nabla(|\tilde{z}+\varphi|^2)f(|\tilde{z}+\varphi|^2)\;dx\right|\notag\\
&\leq\displaystyle\; C\left(e^{-\gamma t}+\|z\|_{H^1}\right)\|z\|_{H^1}\|\tilde{z}\|_{H^1}.\notag\\
\end{align}
Notice that we have used \eqref{tild'},
\[ \int_{\RR^d}\nabla\left(|\varphi|^2\right)f(|\varphi|^2)\;dx=0, \quad \text{and} \quad \int_{\RR^d}\nabla\left(|\tilde{z}+\varphi|^2\right)f(|\tilde{z}+\varphi|^2)\;dx=0. \]

To finish with, we have to obtain estimates for the terms with the letter $C$ involving $a_k'$ and $b_k'$. \\
Due to (\ref{estdercoef}), (\ref{prel2}), and (\ref{derRk}), we compute
\begin{equation}\label{C1'}
\begin{aligned}
\MoveEqLeft[4]
-2\;\sum_{k=1}^Ka_k'\int_{\RR^d}\left\{\Rr\left(iR_k\Delta\overline{\tilde{z}}\right)+\Rr\left(iR_k\overline{\tilde{z}}\right)-\left(\omega_k+\frac{|v_k|^2}{4}\right)\Rr\left(iR_k\overline{\tilde{z}}\phi_k\right)-v_k\cdot\Ii\left(iR_k\nabla\overline{\tilde{z}}\phi_k\right)\right\}dx\\
=&\;2\;\displaystyle\sum_{k=1}^Ka_k'\;\Ii\int_{\RR^d}\left\{\overline{\tilde{z}}\left(\Delta R_k+R_kf\left(|R_k|^2\right)-\left(\omega_k+\frac{|v_k|^2}{4}\right)R_k\phi_k-iv_k\cdot\nabla R_k\phi_k\right)\right\}\;dx\\
&+\mathrm{O}\big(e^{-\gamma t}\|\tilde{z}\|_{H^1}\|z\|_{H^1}\big)\\
=&\;\mathrm{O}\left(e^{-\gamma t}\|\tilde{z}\|_{H^1}\|z\|_{H^1}\right).
\end{aligned}
\end{equation}
On the other hand,
\begin{equation}\label{C1''}
\begin{aligned}
\MoveEqLeft[4] -2\;\displaystyle\sum_{k=1}^Kb_k'\cdot\int_{\RR^d}\bigg\{\Rr(\nabla R_k\Delta\overline{\tilde{z}})+\Rr(\nabla R_k\overline{\tilde{z}}f(|\tilde{z}+\varphi|^2))+\Rr\big(\nabla R_k\overline{R_k}\big(f(|\tilde{z}+\varphi|^2)-f(|\varphi|^2)\big)\big)\\
&\displaystyle\qquad\qquad-\left(\omega_k+\frac{|v_k|^2}{4}\right)\Rr\left(\nabla R_k\overline{\tilde{z}}\phi_k\right)-\Ii\left(v_k\cdot R_k\nabla\overline{\tilde{z}}\phi_k\right)\bigg\}\;dx\\
=&\;2\;\displaystyle\sum_{k=1}^Kb_k'\cdot\Rr\int_{\RR^d}\left\{\nabla\overline{\tilde{z}}\left(\Delta R_k+R_kf\left(|R_k|^2\right)-\left(\omega_k+\frac{|v_k|^2}{4}\right)R_k\phi_k-iv_k\cdot\nabla R_k\phi_k\right)\right\}\;dx\\
&\displaystyle+2\sum_{k=1}^Kb_k'\cdot\int_{\RR^d}\left\{\Rr\left(R_k\overline{\tilde{z}}\nabla\left(f\left(|R_k|^2\right)\right)\right)-\Rr\left(\nabla R_k\overline{R_k}\right)\left(f\left(|\tilde{z}+\varphi|^2\right)-f\left(|\varphi|^2\right)\right)\right\}\;dx\\
&+\mathrm{O}\left((e^{-\gamma t}+\|\tilde{z}\|_{H^1})\|\tilde{z}\|_{H^1}\|z\|_{H^1}\right)\\
=&\;\mathrm{O}\left((e^{-\gamma t}+\|\tilde{z}\|_{H^1})\|\tilde{z}\|_{H^1}\|z\|_{H^1}\right),
\end{aligned}
\end{equation}
again due to  (\ref{estdercoef}), (\ref{prel2}), and (\ref{derRk}).
Consequently, \eqref{C1'} and \eqref{C1''} lead to
\begin{equation}\label{estC}
C_1+C_2+\sum_{k=1}^K\{C_{3,k}-C_{4,k}\}=\mathrm{O}\big(\left(e^{-\gamma t}+\|\tilde{z}\|_{H^1}\right)\|\tilde{z}\|_{H^1}\|z\|_{H^1}\big).
\end{equation}
Proposition \ref{estenergy} follows from Step 1, from estimates (\ref{estA}), (\ref{estB}), and (\ref{estC}), and from the observation that $$Main=\sum_{k=1}^K\left(Main_{1,k}-Main_{2,k}\right)=\mathrm{O}\left(\frac{1}{t}\|\tilde{z}(t)\|_{H^1}^2\right)$$ by \eqref{prel3}.
\end{proof}

\subsubsection{Control of the $R_k$ directions}

\noindent We have the following estimate which expresses that the variation in time of the real scalar products $\Rr\int_{\RR^d}\tilde{z}(t)\overline{R_k}(t)\;dx$ (which appear in \eqref{coerc}) is essentially of order two in $z(t)$.

\begin{Lem}\label{derre}
For all $t\ge T_1$,
\begin{equation}\label{derre1}
\left|\frac{d}{dt}\Rr\int_{\RR^d}\tilde{z}(t)\overline{R_k}(t)\;dx\right|\le C\left(e^{-\gamma t}\|z(t)\|_{H^1}+\|z(t)\|_{H^1}^2\right).
\end{equation}
\end{Lem}

\begin{proof}

We notice that
\begin{equation}\label{prod_scal_tilde_z}
\begin{aligned}
\displaystyle\Rr\int_{\RR^d}\tilde{z}\overline{R_k}\;dx&=\displaystyle\Rr\int_{\RR^d}z\overline{R_k}\;dx+\sum_{j=1}^K\Rr\int_{\RR^d}ia_jR_j\overline{R_k}\;dx+\sum_{j=1}^K\Rr\int_{\RR^d}b_j\cdot \nabla R_j\overline{R_k}\;dx\\
&=\displaystyle\Rr\int_{\RR^d}z\overline{R_k}\;dx+a_k\Rr\int_{\RR^d}i|R_k|^2\;dx+b_k\cdot\Rr\int_{\RR^d}\nabla R_k\overline{R_k}\;dx+\eta(t)\\
&=\displaystyle\Rr\int_{\RR^d}z\overline{R_k}\;dx+\eta(t),\\
\end{aligned}
\end{equation}
where $\eta$ is a complex-valued function defined on a neighborhood of $+\infty$, differentiable in the sense of distributions, and such that $\eta'(t)=\mathrm{O}\left(e^{-\gamma t}\|z(t)\|_{H^1}\right)$.
Moreover,
\begin{equation}\label{prod_scal_z}
 \begin{aligned}
\displaystyle\frac{d}{dt}\Rr\int_{\RR^d}z\overline{R_k}\;dx=&\;\displaystyle\Rr\int_{\RR^d}i\left(\Delta z+f\left(|z+\varphi|^2\right)z+\left(f\left(|z+\varphi|^2\right)-f\left(|\varphi|^2\right)\right)\varphi\right)\overline{R_k}\;dx\\
&+\Rr\int_{\RR^d}-iz\left(\Delta \overline{R_k}+f\left(|R_k|^2\right)\overline{R_k}\right)\;dx\\
=&\;\Rr\int_{\RR^d}i\left( f \left(|\varphi|^2\right)-f \left(|R_k|^2\right) \right) z\overline{R_k}\;dx\\
&+\Rr\int_{\RR^d}i\left( f\left(|z+\varphi|^2\right)-f \left(|\varphi|^2\right)\right) \varphi\overline{R_k}\;dx+\mathrm{O}\left( \|z(t)\|_{H^1}^2 \right),
\end{aligned}
\end{equation}
where we have used $p\ge 2$ and
$$
\int_{\RR^d}\Delta z\overline{R_k}\;dx=\int_{\RR^d} z\Delta\overline{R_k}\;dx\qquad\text{and}\qquad
\big|f(|z+\varphi|^2)-f(|\varphi|^2)\big|\le C(|z|+|z|^{\frac{4}{d}}).$$

\noindent By means of \[\big|f(|\varphi|^2)-f(|R_k|^2)\big|\le C\big(|\varphi-R|^2+|\varphi-R|+\sum_{j\neq k}|R_jR_k|\big)\] (which is consequence of the application of the mean value theorem, as for the proof of Lemma \ref{ineqs}), and by means of \eqref{decay_prop} and \eqref{varphi_3d}, we see that 
\begin{equation}\label{aux_1}
\left|\Rr\int_{\RR^d}i\left(f(|\varphi|^2)-f(|R_k|^2)\right)z\overline{R_k}\;dx\right|\le Ce^{-\gamma t}\|z(t)\|_{L^2}.
\end{equation}
Similarly by Lemma \ref{ineqs}, and noting in addition that \[\Rr\left(i\varphi\overline{R_k}\right)=\Rr\left(i((\varphi-R)+(R-R_k))\overline{R_k}\right),\] 
we have also
\begin{equation}\label{aux_2}
\left|\Rr\int_{\RR^d}i\left(f(|z+\varphi|^2)-f(|\varphi|^2)\right)\varphi\overline{R_k}\;dx\right|\le Ce^{-\gamma t}\|z(t)\|_{H^1}.
\end{equation}
To put it in a nutshell, Lemma  \ref{derre} is now a direct consequence of \eqref{prod_scal_tilde_z}, \eqref{prod_scal_z}, \eqref{aux_1}, and \eqref{aux_2}.
\end{proof}

As a consequence of the preceding lemma, we state: 

\begin{Cor} For all $t\ge T_1$,
\begin{equation}\label{derre2}
\left|\frac{d}{dt}\left(\Rr\int_{\RR^d}\tilde{z}(t)\overline{R_k}(t)\;dx\right)^2\right|\le C\left(e^{-\gamma t}\|z(t)\|_{H^1}+\|z(t)\|_{H^1}^2\right)\|\tilde{z}(t)\|_{H^1}.
\end{equation}
\end{Cor}

\subsubsection{Control of the modulation parameters}

At this point, recalling inequality (\ref{tild}), it remains us to obtain estimates for $|a_k(t)|$ and $|b_k(t)|$.  This is the object of the following result. 

\begin{Lem}\label{inegab} 
For all $t\ge T_1$,
\begin{equation}\label{inegab1}
 |a_k'(t)|+|b_{k}'(t)|\le C\big(e^{-\gamma t}\|z(t)\|_{H^1}+\|z(t)\|_{H^1}^2+\|z\|_{H^1}^{\frac{4}{d}}+\|\tilde{z}(t)\|_{L^2}\big).
\end{equation}
\end{Lem}

\begin{proof}
Due to \[\begin{dcases}
\Delta\tilde{z}=\displaystyle\Delta z+\sum_{j=1}^K\big(ia_j\Delta R_j+b_j\cdot\nabla(\Delta R_j)\big)\\
\partial_t\tilde{z}=\displaystyle\partial_tz+\sum_{k=1}^K\big(ia_k'R_k+b_k'\cdot\nabla R_k\big)+\sum_{k=1}^K\big(ia_k\partial_tR_k+b_k\cdot\nabla \partial_tR_k\big)\\
\partial_tz=i\big(\Delta z+f(|z+\varphi|^2)z+\big(f(|z+\varphi|^2)-f(|\varphi|^2)\big)\varphi\big),
\end{dcases}\]
differentiation with respect to $t$ of equality $\Ii\int_{\RR^d}\tilde{z}\overline{R_k}\;dx=0$ \eqref{rel} implies:
\begin{align*}
0=&\;\displaystyle\Ii\int_{\RR^d}\partial_t\tilde{z}\overline{R_k}\;dx+\Ii\int_{\RR^d}\tilde{z}\partial_t\overline{R_k}\;dx\\
=&\;\Rr\int_{\RR^d}\big(\Delta z+f(|z+\varphi|^2)z+\big(f(|z+\varphi|^2)-f(|\varphi|^2)\big)\varphi\big)\overline{R_k}\;dx\\
&+\displaystyle\sum_{j=1}^K\Ii\left(ia_j'\int_{\RR^d}R_j\overline{R_k}\;dx+b_j'\cdot\int_{\RR^d}\nabla R_j\overline{R_k}\;dx+ia_j\int_{\RR^d}\partial_tR_j\overline{R_k}\;dx+b_j\cdot\int_{\RR^d}\nabla \partial_tR_j\overline{R_k}\;dx\right)\\
&-\displaystyle\Rr\int_{\RR^d}\left(z+\sum_{j=1}^K\big(ia_jR_j+b_j\cdot\nabla R_j\big)\right)\big(\Delta \overline{R_k}+f(|R_k|^2)\overline{R_k}\big)\;dx,
\end{align*}
or equivalently
\begin{equation}\label{eg1}
\begin{aligned}
0=&\;\displaystyle\mathrm{O}\left(\|z(t)\|_{H^1}^2\right)+\Rr\int_{\RR^d}\big(f(|\varphi|^2)-f(|R_k|^2)\big)z\overline{R_k}\;dx+\Rr\int_{\RR^d}\big(f(|z+\varphi|^2)-f(|\varphi|^2)\big)\varphi\overline{R_k}\;dx\\
&+\displaystyle\sum_{j=1}^K\left(a_j'\Rr\int_{\RR^d}R_j\overline{R_k}\;dx+b_j'\cdot\Ii\int_{\RR^d}\nabla R_j\overline{R_k}\;dx\right)\\
&\displaystyle+\sum_{j=1}^K\left(-a_j\Ii\int_{\RR^d}\big(\Delta R_j+f(|R_j|^2)R_j\big)\overline{R_k}\;dx+b_j\cdot\Rr\int_{\RR^d}\big(\Delta R_j+f(|R_j|^2)R_j\big)\nabla \overline{R_k}\;dx\right)\\
&-\displaystyle\sum_{j=1}^K\left(-a_j\Ii\int_{\RR^d}R_j\left(\Delta\overline{R_k}+f(|R_k|^2)\overline{R_k}\right)\;dx+b_j\cdot\Rr\int_{\RR^d}\nabla R_j\left(\Delta \overline{R_k}+f(|R_k|^2)\overline{R_k}\right)\;dx\right)\\
=&\displaystyle\int_{\RR^d}\big(f(|z+\varphi|^2)-f(|\varphi|^2)\big)|R_k|^2\;dx+a_k'\int_{\RR^d}Q_{\omega_k}^2\;dx+\frac{1}{2}b_k'\cdot v_k\int_{\RR^d}Q_{\omega_k}^2\;dx\\
&+\mathrm{O}\left(\|z(t)\|_{H^1}^2+e^{-\gamma t}\|z(t)\|_{H^1}\right).
\end{aligned}
\end{equation}

\noindent Similarly, exploiting the $d$-dimensional equality $\Rr\int_{\RR^d}\tilde{z}\nabla\overline{R_k}\;dx=0$ \eqref{rel}, we see that

\begin{align*}
0=&\;\displaystyle\Ii\int_{\RR^d}\left(\Delta\tilde{z}-\sum_{j=1}^Kia_j\Delta R_j-\sum_{j=1}^Kb_j\cdot\nabla\Delta R_j\right)\nabla\overline{R_k}\;dx+\mathrm{O}\left(\|z(t)\|_{L^2}^2\right)\\
&-\displaystyle\Ii\int_{\RR^d}f(|\varphi|^2)\left(\tilde{z}-\sum_{j=1}^Kia_j R_j-\sum_{j=1}^Kb_j\cdot\nabla R_j\right)\nabla\overline{R_k}\;dx\\
&-\displaystyle\Ii\int_{\RR^d}\big(f(|z+\varphi|^2)-f(|\varphi|^2)\big)\varphi\nabla\overline{R_k}\;dx\\
&\displaystyle+\sum_{j=1}^K\left[a_j'\Ii\int_{\RR^d}\nabla R_k\overline{R_j}\;dx+b_j'\cdot\Rr\int_{\RR^d}\nabla R_j\nabla\overline{R_k}\right]\\
&-\displaystyle\sum_{j=1}^Ka_j\Rr\int_{\RR^d}\Delta R_j\nabla\overline{R_k}\;dx-\sum_{j=1}^Ka_j\Rr\int_{\RR^d}f(|R_j|^2)R_j\nabla\overline{R_k}\;dx\\
&-\displaystyle\sum_{j=1}^Kb_j\cdot\Ii\int_{\RR^d}\nabla(\Delta R_j)\nabla\overline{R_k}\;dx-\sum_{j=1}^Kb_j\cdot\Ii\int_{\RR^d}\nabla\big(f(|R_j|^2)R_j\big)\nabla\overline{R_k}\;dx\\
&+\displaystyle\Ii\int_{\RR^d}\tilde{z}\big(\nabla(\Delta \overline{R_k})+\nabla\big(f(|R_k|^2)\overline{R_k}\big)\big)\;dx,
\end{align*}
or equivalently, using 
\begin{align*}
\int_{\RR^d}\nabla\big(f(Q_{\omega_k}^2)\big)Q_{\omega_k}^2\;dx&=-\sum_{i=1}^d\int_{\RR^d}f(Q_{\omega_k}^2)\partial_{x_i}(Q_{\omega_k}^2)\;dx\\
&=0,
\end{align*}
\begin{equation}\label{eg2} 
\begin{aligned}
0=&\;\mathrm{O}\left(\|z(t)\|_{H^1}^2+e^{-\gamma t}\|z(t)\|_{H^1}\right)+\displaystyle\int_{\RR^d}\big(f(|z+\varphi|^2)-f(|\varphi|^2)\big)\Ii\big(\nabla R_k\overline{R_k}\big)\;dx\\
&\displaystyle+\Ii\int_{\RR^d}\tilde{z}\nabla(f(|R_k|^2))\overline{R_k}\;dx+\frac{a_k'}{2}v_k\int_{\RR^d}Q_{\omega_k}^2\;dx+\Rr\int_{\RR^d}\big[\nabla \overline{R_k}^t\nabla R_k\big]\times b_k'\;dx\\
=&\;\mathrm{O}\left(\|z(t)\|_{H^1}^2+e^{-\gamma t}\|z(t)\|_{H^1}\right)+\displaystyle\int_{\RR^d}\big(f(|z+\varphi|^2)-f(|\varphi|^2)\big)\Ii\big(\nabla R_k\overline{R_k}\big)\;dx\\
&\displaystyle+\Ii\int_{\RR^d}\tilde{z}\nabla\left(f(|R_k|^2)\right)\overline{R_k}\;dx+\frac{a_k'}{2}v_k\int_{\RR^d}Q_{\omega_k}^2\;dx\\
&+\int_{\RR^d}\big[\nabla Q_{\omega_k}{^t\nabla Q_{\omega_k}}+\frac{1}{4}v_k{^t{v_k}}Q_{\omega_k}^2\big]\times b_k'\;dx,\\
\end{aligned}
\end{equation}

\noindent Using Lemma \ref{ineqs} and \eqref{decay_prop}, we obtain
\begin{align*}
\MoveEqLeft[2]\displaystyle\int_{\RR^d}\big(f(|z+\varphi|)^2-f(|\varphi|^2)\big)\Ii\big(\nabla R_k\overline{R_k}\big)\;dx\\
&=2\displaystyle\int_{\RR^d}\Rr(z\overline{\varphi})f'(|\varphi|^2)\Ii\big(\nabla R_k\overline{R_k}\big)\;dx+\mathrm{O}\left(\|z\|_{H^1}^2+\|z\|_{H^1}^{\frac{4}{d}}\right)\\
&=2\displaystyle\int_{\RR^d}\Rr(z\overline{R_k})f'(|R_k|^2)\Ii\big(\nabla R_k\overline{R_k}\big)\;dx+\mathrm{O}\left(e^{-\gamma t}\|z\|_{H^1}+\|z\|_{H^1}^2+\|z\|_{H^1}^{\frac{4}{d}}\right).\\
\end{align*}
\noindent Recalling the definition of $\tilde{z}$ \eqref{def_tilde_z}, this reads also as follows:
\begin{align*}
\MoveEqLeft[2]\int_{\RR^d}\big(f(|z+\varphi|)^2-f(|\varphi|^2)\big)\Ii\left(\nabla R_k\overline{R_k}\right)\;dx\\
&=2\int_{\RR^d}\Rr\left(\tilde{z}\overline{R_k}\right)f'(|R_k|^2)\Ii\left(\nabla R_k\overline{R_k}\right)\;dx+\mathrm{O}\left(e^{-\gamma t}\|z\|_{H^1}+\|z\|_{H^1}^2\right)\\
&=\mathrm{O}\left(\|\tilde{z}\|_{L^2}+e^{-\gamma t}\|z\|_{H^1}+\|z\|_{H^1}^2+\|z\|_{H^1}^{\frac{4}{d}}\right).
\end{align*}

\noindent Considering that $\Ii\left(\nabla R_k\overline{R_k}\right)=v_k|R_k|^2$, we then deduce from (\ref{eg1}) and (\ref{eg2}) that
$$\int_{\RR^d}\left[\nabla Q_{\omega_k}{^t\nabla Q_{\omega_k}}\right]\;dx\times b_k'=\mathrm{O}\left(\|\tilde{z}\|_{L^2}+e^{-\gamma t}\|z\|_{H^1}+\|z\|_{H^1}^2+\|z\|_{H^1}^{\frac{4}{d}}\right)\frac{v_k}{2}.$$ Since $\int_{\RR^d}\left[\nabla Q_{\omega_k}{^t\nabla Q_{\omega_k}}\right]\;dx$ is a positive definite symmetric matrix (by Proposition \ref{indep}), it is invertible, so that inequality \eqref{inegab1} holds for $|b_k'(t)|$. Then we conclude that the same inequality is true for $|a_k'(t)|$ by \eqref{eg1}.
\end{proof}

\subsection{End of the proof}\label{end}

We now conclude the proof of our uniqueness result, that is Theorem \ref{th_uni_f} by gathering \eqref{coerc} and the different controls obtained in subsection \ref{subsec_estimates}. 

Let us begin with a control of $\|z\|_{H^1}$ in terms of $\|\tilde{z}\|_{H^1}$ which relies on the integrability of $t\mapsto\|z(t)\|_{H^1}+\|z(t)\|_{H^1}^{\frac{4}{d}-1}$ in the neighborhood of $+\infty$ (provided $N$ is chosen sufficiently large); once more, we observe here that the condition $d\le 3$ is important.

\begin{Prop}\label{controlz}
For $t$ large enough,
\begin{equation}
\|z(t)\|_{H^1}\le C\left(\sup_{s\ge t}\|\tilde{z}(s)\|_{H^1}+\int_t^{+\infty}\|\tilde{z}(u)\|_{H^1}\;du\right).
\end{equation}
\end{Prop}

\begin{proof}
Recall that we have already seen \eqref{tild}:
$$
\|z\|_{H^1}\le C\left(\|\tilde{z}\|_{H^1}+\sum_{k=1}^K\big(|a_k|+|b_k|\big)\right)
$$
\noindent Therefore, using the control of the modulation parameters obtained before, that is (\ref{inegab1}), 
\begin{multline}\label{estz0}
\|z(t)\|_{H^1}\le C\|\tilde{z}(t)\|_{H^1}\\
+C\left(\int_t^{+\infty}\|\tilde{z}(s)\|_{H^1}\;ds+\int_t^{+\infty}\left(\|z(s)\|_{H^1}^2+\|z(s)\|_{H^1}^{\frac{4}{d}}\right)\;ds+\int_t^{+\infty}e^{-\gamma s}\|z(s)\|_{H^1}\;ds\right).
\end{multline} 
Since $t\mapsto\|z(t)\|_{H^1}+\|z(t)\|_{H^1}^{\frac{4}{d}-1}$ is integrable in the neighborhood of $+\infty$, we have for $t$ large enough:
\begin{equation}\label{estz1}
\int_t^{+\infty}\left(\|z(s)\|_{H^1}^2+\|z(s)\|_{H^1}^{\frac{4}{d}}\right)\;ds\le \left(\int_t^{+\infty}\left(\|z(s)\|_{H^1}+\|z(s)\|_{H^1}^{\frac{4}{d}-1}\right)\;ds\right)\sup_{s\ge t}\|z(s)\|_{H^1}.
\end{equation}
Similarly we have
\begin{equation}\label{estz2}
\int_t^{+\infty}e^{-\gamma s}\|z(s)\|_{H^1}\;ds\le \frac{e^{-\gamma t}}{\gamma}\sup_{s\ge t}\|z(s)\|_{H^1}.
\end{equation}
It follows from (\ref{estz0}), (\ref{estz1}), and (\ref{estz2}) that for $t$ large enough,
\begin{multline}\label{estz'}
\sup_{s\ge t}\|z(s)\|_{H^1}\le C\left(\sup_{s\ge t}\|\tilde{z}(s)\|_{H^1}+\int_t^{+\infty}\|\tilde{z}(s)\|_{H^1}\;ds\right)\\
+C\left(\int_t^{+\infty}\left(\|z(s)\|_{H^1}+\|z(s)\|_{H^1}^{\frac{4}{d}-1}\right)\;ds+e^{-\gamma t}\right)\sup_{s\ge t}\|z(s)\|_{H^1}.
\end{multline}
Hence, for $t$ large enough,
\begin{equation}\label{estz}
\sup_{s\ge t}\|z(s)\|_{H^1}\le C\left(\sup_{s\ge t}\|\tilde{z}(s)\|_{H^1}+\int_t^{+\infty}\|\tilde{z}(u)\|_{H^1}\;du\right),
\end{equation}
which ends the proof of Proposition \ref{controlz}.
\end{proof}

Now, we deduce the following

\begin{Lem}\label{lem_tilde_z}
There exists $T>0$ such that for all $t\ge T$, $\tilde{z}(t)=0$.
\end{Lem}

\begin{proof}
By means of (\ref{coerc}), (\ref{estenergy1}), and (\ref{derre2}), we can write for $t$ large enough
\begin{equation}\label{ineg_tilde_z_1}
\begin{aligned}
\displaystyle\|\tilde{z}(t)\|_{H^1}^2&\leq\displaystyle C\int_t^{+\infty}\left(\frac{1}{s}\|\tilde{z}(s)\|_{H^1}^2+e^{-\gamma s}\|z(s)\|_{H^1}\|\tilde{z}(s)\|_{H^1}+\|z(s)\|_{H^1}^2\|\tilde{z}(s)\|_{H^1}\right)\;ds
\\
&\leq\displaystyle C\left[\int_t^{+\infty}\left(\frac{1}{s}\|\tilde{z}(s)\|_{H^1}+e^{-\gamma s}\|z(s)\|_{H^1}+\|z(s)\|_{H^1}^2\right)\;ds\right]\sup_{s\ge t}\|\tilde{z}(s)\|_{H^1}.\\
\end{aligned}
\end{equation}
We deduce from the preceding line that for $t$ large enough
\begin{equation}\label{ineg_tilde_z_2}
\|\tilde{z}(t)\|_{H^1}\le C\int_t^{+\infty}\left(\frac{1}{s}\|\tilde{z}(s)\|_{H^1}+e^{-\gamma s}\|z(s)\|_{H^1}+\|z(s)\|_{H^1}^2\right)\;ds.
\end{equation}
Using (\ref{estz1}), (\ref{estz2}), and \eqref{estz}, this leads to the fact that for $t$ large enough 
\begin{multline*}
\|\tilde{z}(t)\|_{H^1}\le C\int_{t}^{+\infty}\frac{1}{s}\|\tilde{z}(s)\|_{H^1}\;ds\\
+C\left(e^{-\gamma t}+\int_{t}^{+\infty}\|z(s)\|_{H^1}\;ds\right)\left(\sup_{s\ge t}\|\tilde{z}(s)\|_{H^1}+\int_t^{+\infty}\|\tilde{z}(u)\|_{H^1}\;du\right).
\end{multline*}
Thus, for large values of $t$, 
\begin{equation}\label{ineg_tilde_z_3}
\|\tilde{z}(t)\|_{H^1}\le C\left[\int_{t}^{+\infty}\frac{1}{s}\|\tilde{z}(s)\|_{H^1}\;ds+\left(e^{-\gamma t}+\int_{t}^{+\infty}\|z(s)\|_{H^1}\;ds\right)\left(\int_t^{+\infty}\|\tilde{z}(u)\|_{H^1}\;du\right)\right].
\end{equation}
Since by assumption $\|z(t)\|_{H^1}=\mathrm{O}\left(\frac{1}{t^\alpha}\right)$ with $\alpha>2$ and since $e^{-\gamma t}\le \frac{1}{t^{\alpha-1}}$ for $t$ large enough, there exist $\tilde{C}\ge 0$ and $T\ge 1$ such that for all $t\ge T$,
\begin{equation}\label{inegtildez}
\|\tilde{z}(t)\|_{H^1}\le \tilde{C}\left[\int_{t}^{+\infty}\frac{1}{s}\|\tilde{z}(s)\|_{H^1}\;ds+\frac{1}{t^{\alpha-1}}\left(\int_t^{+\infty}\|\tilde{z}(u)\|_{H^1}\;du\right)\right].
\end{equation}

\begin{Rq}\label{rq:indep_z}
Note that in \eqref{inegtildez}, $\tilde{C}$ seems to depend on $z$ (or equivalently on $u$) but in fact it does not (even if it means changing $T$ which does actually depend on $z$). Indeed, $\tilde{C}$ depends only on the constants appearing in \eqref{varphi_3d}, \eqref{tild}, \eqref{tild'}, \eqref{coerc} (linked with the parameters used to define the solitons $R_k$), on the constants appearing in \eqref{comp}, \eqref{estenergy1}, \eqref{derre2}, \eqref{inegab1} (linked with $f$, with the parameters used to define the solitons, and with $\|z(t)\|_{H^1}$ which can be chosen less or equal to 1 provided $t$ is sufficiently large, depending on $z$), and also on universal constants which enable us to pass from \eqref{estz'} to \eqref{estz}, from \eqref{ineg_tilde_z_1} to \eqref{ineg_tilde_z_2}, and from \eqref{ineg_tilde_z_3} to \eqref{inegtildez} (on condition that $t$ is once more sufficiently large, which depends on $z$). Thus one should read the following assertion:
there exists $\tilde{C}>0$ such that for all $z$ satisfying \eqref{derz} and $\|z(t)\|_{H^1}=\mathrm{O}\left(\frac{1}{t^\alpha}\right)$, there exists $T(z)>0$ such that for all $t\ge T(z)$, 
\eqref{inegtildez} holds.
\end{Rq}

Now take $N$ in \eqref{est_z_N} such that $N>4\tilde{C}+1$ (in this way, $N$ does not depend on $z$, as emphasized in Remark \ref{rq:indep_z}). Even if it means taking a larger $T$, we can assume \[\exists\;c\ge 0,\:\forall\;t\ge T,\quad \|\tilde{z}(t)\|_{H^1}\leq\frac{c}{t^N}.\]
\noindent Then $A:=\displaystyle\sup_{t\ge T}\{t^N\|\tilde{z}(t)\|_{H^1}\}$ is well defined. Let us pick up $\tilde{T}\ge T$ such that $\tilde{T}^N\|\tilde{z}(\tilde{T})\|_{H^1}\ge \frac{A}{2}$. \\ \\
Now, replacing $t$ by $\tilde{T}$ in (\ref{inegtildez}), we obtain 
\begin{equation}
\begin{aligned}
\frac{A}{2\tilde{T}^N}&\leq\tilde{C}A\left(\frac{1}{N\tilde{T}^N}+\frac{1}{(N-1)\tilde{T}^{N+\alpha-2}}\right) \leq\frac{2\tilde{C}A}{(N-1)\tilde{T}^N}.
\end{aligned}
\end{equation}

\noindent Supposing $A\neq 0$ would lead to a contradiction because of the choice of $N$. \\
\noindent Consequently \[\forall\; t\ge T,\qquad \|\tilde{z}(t)\|_{H^1}=0. \] 
\end{proof}

We deduce from Proposition \ref{controlz} and Lemma \ref{lem_tilde_z} that \[\forall\; t\ge T,\qquad \|z(t)\|_{H^1}=0.\] The local well-posedness in $H^1(\RR^d)$ of \eqref{NLS} implies then $u=\varphi$. Hence Theorem \ref{th_uni_f} is proved.

\subsection{Uniqueness result for the critical pure-power case}\label{subsec_L2crit}

In this paragraph, let $d\ge 1$ and $f:r\mapsto r^{\frac{2}{d}}$. 
Our proof of uniqueness in the class of multi-solitons $u$ such that $\|u(t)-R(t)\|_{H^1}\underset{t\to+\infty}{=}\mathrm{O}\left(\frac{1}{t^N}\right)$ (for some $N\in\NN^*$ sufficiently large to be determined later) and in the $L^2$-critical case consists in exploiting the same ideas as for the subcritical case. Nevertheless it is this time based on Proposition \ref{coerccritic}, stated below and proved in Appendix. 

\begin{Prop}\label{coerccritic}
Assume that $f(r)=r^{\frac{2}{d}}$ and let $\omega>0$. There exists $\mu>0$ such that for all $w=w_1+iw_2\in H^1(\RR,\CC)$,
\begin{multline} \label{coer_crit} 
H(w) \ge \mu \| w \|_{H^1}^2 - \frac{1}{\mu} \left( \int_{\RR^d} w_1 Q_\omega dx \right)^2 \\
-  \frac{1}{\mu} \left(\sum_{i=1}^d \left( \int_{\RR^d} w_1 \partial_{x_i} Q_\omega\;dx \right)^2 + \left( \int_{\RR^d} w_1\left(x\cdot\nabla Q_\omega\right)^2\;dx \right)^2 + \left( \int_{\RR^d} w_2 Q_\omega\;dx \right)^2 \right).
\end{multline}
\end{Prop}

\noindent In order not to be too redundant, we only explicit the main modifications of the proof given for the stable case. 

\subsubsection{Change of variable}

We still consider $\varphi$, $u$, and $z$ as defined at the beginning of Section \ref{uniq}. In order to apply Proposition \ref{coerccritic} (which states a coercivity property available in the critical case), one has to take into account a third family of directions indexed by $k=1,\dots,K$. 
More precisely, let us introduce $y_k(t,x):=x-v_kt-x_k^0\in\RR^d$, for all $k=1,\dots,K$, and
\begin{multline*}
\tilde{z}(t,x):=z(t,x)+\sum_{k=1}^K\left\{ia_k(t)R_k(t,x)+b_k(t)\cdot\nabla R_k(t,x)\right\}\\
+\sum_{k=1}^Kc_k(t)\left(\frac{d}{2}R_k+y_k\cdot\nabla R_k-\frac{i}{2}v_k\cdot y_kR_k\right)(t,x),
\end{multline*} where $a_k$, $b_k$, and $c_k$ are well defined on $[T_1,+\infty)$ (even if it means taking a larger $T_1$) with values respectively in $\RR$, $\RR^d$, and $\RR$ such that 
\begin{equation}
\forall\;t\ge T_1,\quad \begin{dcases}
\Ii\int_{\RR^d}\tilde{z}(t)\overline{R_k(t)}\;dx=0\\
\Rr\int_{\RR^d}\tilde{z}(t)\overline{\nabla R_k(t)}\;dx=0\\
\Rr\int_{\RR^d}\tilde{z}(t)\overline{\left(\frac{d}{2}R_k+y_k\cdot \nabla R_k-\frac{i}{2}v_k\cdot y_kR_k\right)(t)}\;dx=0.
\end{dcases}
\end{equation}

As for the stable case, we can prove that $a_k(t)$, $b_k(t)$, and $c_k(t)$, $k=1,\dots,K$, are uniquely determined by the preceding orthogonality conditions. This time, we have to show indeed that the following block matrix is invertible:
$$\tilde{M}(t):= \left[\begin{array}{ccccc}
A_{0,0}(t)&B_{1,1}(t)&\cdots&B_{1,d}(t)&Z_{0}(t)\\
^tB_{1,1}(t)&A_{1,1}(t)&\cdots&A_{1,d}(t)&Z_1(t)\\
\vdots&\vdots&\ddots&\vdots&\vdots\\
^tB_{1,d}(t)&A_{d,1}(t)&\cdots&A_{d,d}(t)&Z_d(t)\\
^tZ_0(t)&^tZ_1(t)&\cdots&^tZ_d(t)&W(t)\\
\end{array}\right],$$
where $A_{0,0}$, $A_{i,j}$, $B_{1,j}$ are defined in paragraph \ref{change_var}, $Z_{i}(t)$ ($i=0,\dots,d$) has entries zero on its diagonal and $\mathrm{O}(e^{-\gamma t})$ elsewhere, and $W(t)$ possesses the coefficients $\int_{\RR^d}\left(x\cdot\nabla Q_{\omega_k}\right)^2\;dx$, $k=1,\dots, K$ on its diagonal and $\mathrm{O}(e^{-\gamma t})$ elsewhere. \\
For the sake of completeness, let us justify how to determine the coefficients of $Z_{i}(t)$ ($i=1,\dots,d$) which are the less obvious ones to compute.
By the orthogonality condition $\Rr\int_{\RR^d}\tilde{z}(t)\nabla\overline{R_k}(t)\;dx=0$, the coefficient $Z_i(k,l)$ of $Z_i$ located at line $k$ and column $l$ is equal to 
$$Z_i(k,l)=\frac{d}{2}\Rr\int_{\RR^d}R_k\partial_{x_i}\overline{R_l}\;dx+\Rr\int_{\RR^d}y_k\cdot\nabla R_k\partial_{x_i}\overline{R_l}\;dx+\frac{1}{2}v_k\cdot\Ii\int_{\RR^d}y_kR_k\partial_{x_i}\overline{R_l}\;dx.$$
Thus for $k\neq l$, we have $Z_i(k,l)=\mathrm{O}(e^{-\gamma t})$  by \eqref{decay_prop}, and for $k=l$, we obtain:
\begin{align*}
Z_i(k,l)&=\Rr\int_{\RR^d}y_k\cdot\left(\nabla R_k-\frac{i}{2}v_k R_k\right)\partial_{x_i}\overline{R_k}\;dx =\int_{\RR^d}x\cdot\nabla Q_{\omega_k}\partial_{x_i}Q_{\omega_k}\\
&=\sum_{j=1}^d\int_{\RR^{d-1}}\left(\int_{\RR}x_j\partial_{x_j}Q_{\omega_k}\partial_{x_i}Q_{\omega_k}\;dx_i\right)dx_1\dots\hat{dx_{i}}\dots dx_d =0,
\end{align*}
since for all $j\in\{1,\dots,d\}$, $x_i\mapsto x_j\partial_{x_j}Q_{\omega_k}(x_1,\dots,x_d)\partial_{x_i}Q_{\omega_k}(x_1,\dots,x_d)$ is an odd integrable function on $\RR$ in view of the fact that $Q_{\omega_k}$ is radial. \\

\noindent Hence, we obtain that
$$\det\tilde{M}(t)=\det M(t)\prod_{k=1}^K\int_{\RR^d}\left(x\cdot\nabla Q_{\omega_k}\right)^2\;dx+\mathrm{O}(e^{-\gamma t})$$ is strictly positive for $t$ large enough (see \eqref{non_zero} in Appendix). \\

Moreover, for all $t\ge T_1$, \begin{equation} |a_k(t)|,|b_k(t)|,|c_k(t)|\leq\|z(t)\|_{L^2},\end{equation} and
\begin{equation}
|a_k'(t)|,|b_k'(t)|,|c_k'(t)|\leq\|z(t)\|_{H^1}.
\end{equation}

\begin{Rq}\label{rq:L2crit}
The consideration of $\tilde{z}$ turns out to be appropriate judging by the properties stated in Lemma \ref{estcrit}.
Besides let us note that the particular non-linearity satisfies the ODE $dxf'(x)=2f(x)$ in the $L^2$-critical case; this will be truly useful to control the third family of directions associated with the coefficients $c_k$.
\end{Rq}

 First of all, let us begin with the useful computation of the derivative of $\tilde{z}$ with respect to the time variable.

\begin{Lem} We have
\begin{align*}
\partial_t\tilde{z}=&\;\displaystyle i\Big(\Delta\tilde{z}+f(|\varphi|^2)\tilde{z}+\big(f(|\tilde{z}+\varphi|^2)-f(|\varphi|^2)\big)\varphi\Big)+\displaystyle\sum_{k=1}^K\big\{ia_k'R_k+b_k'\cdot\nabla R_k\big\}\\
&-2i\displaystyle\sum_{k=1}^Kc_k\omega_kR_k+\sum_{k=1}^Kc_k'\left(\frac{d}{2}R_k+y_k\cdot\nabla R_k-\frac{i}{2}v_k\cdot y_k R_k(t,x)\right)+\epsilon_1,
\end{align*}
where $\epsilon_1$ is a function of $t$ and $x$ such that \[\displaystyle\int_{\RR^d}|\epsilon_1|\big(|\tilde{z}|+|\nabla\tilde{z}|+|\Delta\tilde{z}|\big)\;dx\le C\big(e^{-\gamma t}\|z\|_{H^1}\|\tilde{z}\|_{H^1}+\|z\|_{H^1}^2\|\tilde{z}\|_{H^1}\big).\]
\end{Lem}

\begin{proof}
Note that $\partial_t\tilde{z}$ decomposes like
\begin{equation}\label{dertildez1}
\begin{aligned}
\partial_t\tilde{z}=&\;\displaystyle\partial_tz+\sum_{k=1}^K\left\{a_k'(\cdots)+b_k'(\cdots)+a_k(\cdots)+b_k(\cdots)\right\}\\
&+\sum_{k=1}^Kc_k'\left(\frac{d}{2}R_k+y_k\cdot\nabla R_k-\frac{i}{2}v_k\cdot y_kR_k\right)\\
&\displaystyle+\sum_{k=1}^Kc_k\bigg(i\frac{d}{2}(\Delta R_k+f(|R_k|^2))R_k)-v_k\cdot\nabla R_k+\frac{i}{2}|v_k|^2R_k+iy_k\cdot\nabla\big(\Delta R_k+f(|R_k|^2)R_k\big)\\
&\qquad\qquad\quad\displaystyle+\frac{1}{2}v_k\cdot y_k\big(\Delta R_k+f(|R_k|^2)R_k\big)\bigg).\\
\end{aligned}\end{equation}

\noindent Now, we want to express $\partial_tz$ given by \eqref{derz} in terms of $\tilde{z}$, as already made in Lemma \ref{dertildez}. For this, let us observe that:
\begin{gather}\label{aux1}
\Delta\big(y_k\cdot \nabla R_k\big) =y_k\cdot \nabla (\Delta R_k)+2\Delta R_k, \\
\label{aux2}
\Delta\left(-\frac{i}{2}v_k\cdot y_k R_k\right) =-\frac{i}{2}v_k\cdot y_k\Delta R_k-iv_k\cdot\nabla R_k, \\
\begin{aligned}
\label{aux3}
f(|z+\varphi|^2)z & =\;\big(f(|z+\varphi|^2)-f(|\varphi|^2)\big)z+f(|\varphi|^2)\tilde{z} \\
&\quad -f(|\varphi|^2)\sum_k\left\{a_k(\cdots)+b_k(\cdots)+c_k\left(\frac{d}{2}R_k+y_k\cdot\nabla R_k-\frac{i}{2}v_k\cdot y_kR_k(t,x)\right)\right\},
\end{aligned}
\end{gather} and
\begin{equation}\label{aux4}
\begin{aligned}
\big(f(|z+\varphi|^2)-f(|\varphi|^2)\big)\varphi &= \displaystyle\big(f(|\tilde{z}+\varphi|^2)-f(|\varphi|^2)\big)\varphi+2\Rr\big((z-\tilde{z})\overline{\varphi}\big)f'(|\varphi|^2)\varphi+\epsilon_1\\
& =\big(f(|\tilde{z}+\varphi|^2)-f(|\varphi|^2)\big)\varphi \\
&\:\:\:- 2\sum_{k=1}^K\Rr\left[\left(b_k(\cdots)+c_k\left(\frac{d}{2}R_k+y_k\cdot\nabla R_k\right)\right)\overline{R_k}\right]f'(|R_k|^2)R_k+\epsilon_1.
\end{aligned}
\end{equation}
Inserting each equality (\ref{aux1}), (\ref{aux2}), (\ref{aux3}), and (\ref{aux4}) in (\ref{dertildez1}) leads to 
\begin{align*}
\partial_t\tilde{z}=&\;\displaystyle i\left(\Delta\tilde{z}+f(|\varphi|^2)\tilde{z}+\big(f(|\tilde{z}+\varphi|^2)-f(|\varphi|^2)\big)\varphi\right)+\displaystyle\sum_{k=1}^K\big\{ia_k'R_k+b_k'\cdot\nabla R_k\big\}\\
&\displaystyle+\sum_{k=1}^Kc_k'\left(\frac{d}{2}R_k+y_k\cdot\nabla R_k-\frac{i}{2}v_k\cdot y_k R_k(t,x)\right)\\
&+\displaystyle\sum_{k=1}^Kc_k\left(-2v_k\cdot\nabla R_k+\frac{i}{2}|v_k|^2R_k-2i\Delta R_k-id|R_k|^2f'(|R_k|^2)R_k\right)+\epsilon_1.\\
\end{align*}
Finally, we conclude by means of 
\begin{equation}\label{fpart}
d\;|R_k|^2f'(|R_k|^2)=2f(|R_k|^2)
\end{equation}
(which indeed holds in the $L^2$-critical case as mentioned in Remark \ref{rq:L2crit}) and the two possibilities given in \eqref{derRk} to write $\partial_tR_k$.
\end{proof}

\subsubsection{Control of $\tilde{H}$ and of the modulation parameters}

\noindent Take again $\tilde{H}$ as defined at the end of paragraph \ref{wein} and consider still $Main(t)$ as in Proposition \ref{estenergy}. Then we can state

\begin{Lem}[Control of the derivative of the Weinstein functional]\label{estcrit}
The following assertion holds true: 
\begin{align*}
\frac{d\tilde{H}}{dt}(t) & =Main(t)-4\sum_{k=1}^Kc'_k(t)\left(\omega_k-\frac{|v_k|^2}{4}\right)\Rr\int_{\RR^d}\tilde{z}\overline{R_k}(t)\;dx \\
& \qquad +\mathrm{O}\left(\left(e^{-\gamma t}+\|z(t)\|_{H^1}\right)\|\tilde{z}(t)\|_{H^1}\|z(t)\|_{H^1}\right).
\end{align*}
\end{Lem} 

\begin{proof}
Take again the proof of Proposition \ref{estenergy}. Concerning the expression of the derivative of $\tilde{H}$, observe that everything is kept unchanged in the present context  except that we have to take care of the additional terms involving the parameters $c_k(t)$ and $c_k'(t)$, for all $k=1,\dots,K$. \\

\noindent Let us define the $\CC$-linear endomorphism $\mathcal{L}_k$ of $H^1(\RR^d)$ by 
\[\mathcal{L}_k(v):=-\Delta v-f(|R_k|^2)v+\left(\omega_k+\frac{|v_k|^2}{4}\right)v+iv_k\cdot\nabla v.\]

\noindent Observe that $\mathcal{L}_k(R_k)=0$ and for all $v,w\in H^1(\RR^d)$,
\begin{equation}\label{prop_bilinearite}
\Rr\int_{ \RR^d}v\overline{\mathcal{L}_k(w)}\phi_k\;dx=\Rr\int_{ \RR^d}\overline{w}\mathcal{L}_k(v)\phi_k\;dx-2\Rr\int_{ \RR^d}\overline{w}\nabla v\cdot\nabla\phi_k\;dx+\Rr\; i\int_{ \RR^d}v\overline{w}v_k\cdot\nabla\phi_k\;dx,
\end{equation}

Using  \eqref{prel2}, \eqref{prel4}, and \eqref{prop_bilinearite}, we deduce that for all $v\in H^1(\RR^d)$:
\begin{equation}\label{prop_bilinearite_Rk}
\begin{aligned}
\Rr\int_{ \RR^d}R_k\overline{\mathcal{L}_k(v)}\;dx&=\Rr\int_{ \RR^d}R_k\overline{\mathcal{L}_k(v)}\phi_k\;dx+\mathrm{O}\left(e^{-\gamma t}\|v\|_{H^1}\right)\\
&=\Rr\int_{ \RR^d}v\overline{\mathcal{L}_k(R_k)}\phi_k\;dx+\mathrm{O}\left(e^{-\gamma t}\|v\|_{H^1}\right) =\mathrm{O}\left(e^{-\gamma t}\|v\|_{H^1}\right).
\end{aligned}
\end{equation}

\noindent The term associated with $c_k(t)$ in the expression of $\frac{d}{dt}\tilde{H}$ writes $$2\;c_k(t)\Rr\int_{\RR^d}(-2i\omega_kR_k)\overline{\mathcal{L}_k(\tilde{z})}\;dx.$$ By \eqref{prop_bilinearite_Rk}, it is thus bounded by $Ce^{-\gamma t}\|\tilde{z}\|_{H^1}\|z\|_{H^1}$.\\

It remains us to obtain the term associated with $c_k'$ in $\frac{d}{dt}\tilde{H}$. This term corresponds to $\mathcal{I}_{1,k}-\mathcal{I}_{2,k}$, where
$$\mathcal{I}_{1,k}=2\;\Rr\int_{\RR^d}\left(\frac{d}{2}R_k+y_k\cdot\nabla R_k-\frac{i}{2}v_k\cdot y_kR_k(t,x)\right)\overline{\mathcal{L}_k(\tilde{z})}\phi_k\;dx$$
and $$\mathcal{I}_{2,k}=2\;\Rr\int_{\RR^d}\left(\frac{d}{2}R_k+y_k\cdot\nabla R_k\right)\overline{\varphi}\big(f(|\tilde{z}+\varphi|^2)-f(|\varphi|^2)\big)\;dx.$$

\noindent Let us concentrate first on $\mathcal{I}_{1,k}$. By \eqref{prop_bilinearite}, 
\[\mathcal{I}_{1,k}=2\;\Rr\int_{\RR^d}\mathcal{L}_k\left(\frac{d}{2}R_k+y_k\cdot\nabla R_k-\frac{i}{2}v_k\cdot y_kR_k(t,x)\right)\overline{\tilde{z}}\;dx+\mathrm{O}\left(e^{-\gamma t}\|\tilde{z}\|_{H^1}\right).\] 
Moreover,
\begin{align*}
\MoveEqLeft[4]\Rr\int_{\RR^d}\overline{\tilde{z}}\mathcal{L}_k\left(\frac{d}{2}R_k\right)\;dx=0; \\
\MoveEqLeft[4]\Rr\int_{\RR^d}\overline{\tilde{z}}\mathcal{L}_k(y_k\cdot\nabla R_k)\;dx\\
=&\;\displaystyle\Rr\int_{\RR^d}\overline{\tilde{z}}\big(-y_k\cdot\nabla(\Delta R_k)-2\Delta R_k\big)\;dx-\Rr\int_{\RR^d}\overline{\tilde{z}}f(|R_k|^2)y_k\cdot\nabla R_k\;dx\\
&\displaystyle+\Rr\int_{\RR^d}i\overline{\tilde{z}}v_k\cdot\nabla(y_k\cdot\nabla R_k)\;dx+\Rr\int_{\RR^d}\overline{\tilde{z}}\left(\omega_k+\frac{|v_k|^2}{4}\right)y_k\cdot\nabla R_k\;dx\\
=&\;\displaystyle\Rr\int_{\RR^d}\left(\nabla\overline{\tilde{z}}\cdot y_k+d\overline{\tilde{z}}\right)\Delta R_k\;dx-2\;\Rr\int_{\RR^d}\overline{v}\Delta R_k\;dx\\
&\displaystyle+\Rr\int_{\RR^d}\left(\nabla\overline{\tilde{z}}\cdot y_k+d\overline{\tilde{z}}\right)f(|R_k|^2)R_k\;dx+\Rr\int_{\RR^d}\overline{\tilde{z}}y_k\cdot\nabla(f(|R_k|^2))R_k\;dx\\
&\displaystyle-\Rr\int_{\RR^d}i\left(\nabla\overline{\tilde{z}}\cdot y_k\right)\big(v_k\cdot\nabla R_k\big)\;dx-(d-1)\Rr\int_{\RR^d}i\overline{\tilde{z}}v_k\cdot\nabla R_k\;dx\\
&-\Rr\int_{\RR^d}\left(\nabla\overline{\tilde{z}}\cdot  y_k+d\overline{\tilde{z}}\right)\left(\omega_k+\frac{|v_k|^2}{4}\right)R_k\;dx+\mathrm{O}\left(e^{-\gamma t}\|\tilde{z}\|_{L^2}\right)\\
=&\displaystyle-2\;\Rr\int_{\RR^d}\overline{\tilde{z}}\Delta R_k\;dx+\Rr\int_{\RR^d}i\overline{\tilde{z}}v_k\cdot\nabla R_k\;dx\\
&\displaystyle+2\;\Rr\int_{\RR^d}\overline{\tilde{z}}f'(|R_k|^2)y_k\cdot\Rr\left(\overline{R_k}\nabla R_k\right)R_k\;dx+\mathrm{O}(e^{-\gamma t}\|\tilde{z}\|_{L^2}); \\
\MoveEqLeft[4]\displaystyle\Rr\int_{\RR^d}\overline{\tilde{z}}\mathcal{L}_k\left(\frac{i}{2}v_k\cdot y_k R_k\right)\;dx\\
=&\;\displaystyle\Rr\int_{\RR^d}\overline{\tilde{z}}\left(-\frac{i}{2}(v_k\cdot y_k)\Delta R_k-iv_k\cdot\nabla R_k\right)\;dx+\Rr\int_{\RR^d}\overline{\tilde{z}}\left(\omega_k+\frac{|v_k|^2}{4}\right)\frac{i}{2}v_k\cdot y_kR_k\;dx\\
&\;\displaystyle-\frac{1}{2}\Rr\int_{\RR^d}\overline{\tilde{z}}v_k\cdot\nabla\left(v_k\cdot y_kR_k\right)\;dx-\Rr\int_{\RR^d}\frac{i}{2}\overline{\tilde{z}}f(|R_k|^2)v_k\cdot y_kR_k\;dx+\mathrm{O}\left(e^{-\gamma t}\|\tilde{z}\|_{L^2}\right)\\
=&\displaystyle-\Rr\int_{\RR^d}i\overline{\tilde{z}}v_k\cdot\nabla R_k\;dx-\frac{|v_k|^2}{2}\Rr\int_{\RR^d}\overline{\tilde{z}}R_k\;dx+\mathrm{O}\left(e^{-\gamma t}\|\tilde{z}\|_{L^2}\right).
\end{align*}
Note that to establish the three preceding equalities, we have used once again $\mathcal{L}_k(R_k)=0$. Thus, gathering the preceding calculations, we infer 
\begin{equation}\label{terme_ck'1}
\begin{aligned}
c_k'\mathcal{I}_{1,k}=&\;\displaystyle 2\;c_k'\Rr\int_{\RR^d}\overline{\tilde{z}}R_ky_k\cdot\nabla\big(f(|R_k|^2)\big)\;dx_k \displaystyle -4\;c_k'\Rr\int_{\RR^d}\overline{\tilde{z}}\left(\Delta R_k+iv_k\cdot\nabla R_k-\frac{|v_k|^2}{4}R_k\right)\;dx\\
&+\mathrm{O}\left(e^{-\gamma t}\|\tilde{z}\|_{L^2}\|z\|_{H^1}\right).
\end{aligned}
\end{equation}
Now, let us focus on the second integral $\mathcal{I}_{2,k}$. On the one hand, by means of a Taylor expansion, by \eqref{decay_prop}, \eqref{varphi_3d'}, and \eqref{fpart}, we obtain 
\begin{equation}\label{terme_ck'2}
\begin{aligned}
\MoveEqLeft[4]
\displaystyle d\int_{\RR^d}|R_k|^2\big(f(|\tilde{z}+\varphi|^2)-f(|\varphi|^2)\big)\;dx\\
&=\displaystyle d\int_{\RR^d}|R_k|^2\times 2\Rr(\tilde{z}R_k)f'(|R_k|^2)\;dx+\mathrm{O}\big(e^{-\gamma t}\|\tilde{z}\|_{L^2}+\|\tilde{z}\|_{L^2}^2\big)\\
&=\displaystyle 4\int_{\RR^d}f(|R_k|^2)\Rr(\tilde{z}R_k)\;dx+\mathrm{O}\big(e^{-\gamma t}\|\tilde{z}\|_{L^2}+\|\tilde{z}\|_{L^2}^2\big).
\end{aligned}
\end{equation}
On the other, we observe that
\begin{equation}\label{terme_ck'3}
\begin{aligned}
\MoveEqLeft[4]2\;\Rr\int_{\RR^d}y_k\cdot\nabla R_k\overline{R_k}\big(f(|\tilde{z}+\varphi|^2)-f(|\varphi|^2)\big)\;dx\\
&=4\;\Rr\int_{\RR^d}y_k\cdot\nabla R_k\overline{R_k}\Rr\left(\tilde{z}\overline{R_k}\right)f'(|R_k|^2)\;dx+\mathrm{O}\big(e^{-\gamma t}\|\tilde{z}\|_{L^2}+\|\tilde{z}\|_{L^2}^2\big).
\end{aligned}
\end{equation}
Thus, we deduce from \eqref{terme_ck'2}, \eqref{terme_ck'3}, and \eqref{defvarphi} that
\begin{equation}\label{terme_I_2,k}
\begin{aligned}
c_k'\mathcal{I}_{2,k}=&\;c_k'\Rr\int_{\RR^d}\overline{\tilde{z}}\left(\Delta R_k+iv_k\cdot\nabla R_k-\frac{|v_k|^2}{4}R_k\right)\;dx\\
&+2\;c_k'\Rr\int_{\RR^d}\overline{\tilde{z}}R_ky_k\cdot\nabla\left(f(|R_k|^2)\right)\;dx+\mathrm{O}\big(e^{-\gamma t}\|\tilde{z}\|_{L^2}+\|\tilde{z}\|_{L^2}^2\big).
\end{aligned}
\end{equation}

From \eqref{terme_ck'1} and \eqref{terme_I_2,k}, we conclude that the term associated with $c_k'$ in $\frac{d}{dt}\tilde{H}$ is equal to 
$$-4\;c_k'\Rr\int_{\RR^d}\overline{\tilde{z}}\left(\Delta R_k+iv_k\cdot\nabla R_k+f(|R_k|^2)R_k-\frac{|v_k|^2}{4}R_k\right)\;dx+\mathrm{O}(e^{-\gamma t}\|\tilde{z}\|_{L^2}\|z\|_{H^1}),$$
that is to
$$4\;c_k'\left(\omega_k-\frac{|v_k|^2}{4}\right)\Rr\int_{\RR^d}\overline{\tilde{z}}R_k\;dx+\mathrm{O}(e^{-\gamma t}\|\tilde{z}\|_{L^2}\|z\|_{H^1}).$$
This finishes the proof of Lemma \ref{estcrit}. 
\end{proof}

\begin{Lem}[Control of the modulation parameters]\label{lem_modpar}
We have for all $k=1,\dots,K$:
$$\forall\;t\ge T_1,\quad \displaystyle|a_k'(t)|+|b_k'(t)|+|c_k'(t)|\le C\big(e^{-\gamma t}\|z\|_{H^1}+\|z\|_{H^1}^2+\|\tilde{z}\|_{L^2}\big).$$
\end{Lem}

\begin{proof}
This lemma follows from the preliminary computations
\begin{align*}
\Delta \tilde{z}=&\Delta z+\sum_{j=1}^K\left\{ia_j\Delta R_k+b_j\cdot\nabla(\Delta R_j)+c_j\left(\left(\frac{d}{2}+2\right)\Delta R_j+y_j\cdot\nabla(\Delta R_j)\right)\right\}\\
&-\sum_{j=1}^K\left\{\frac{i}{2}v_j\cdot y_j\Delta R_j+iv_j\cdot\nabla R_j\right\}\\
\partial_t\tilde{z}=&\partial_t{z}+\sum_{j=1}^K\left\{ia_j'R_j+b_j'\cdot\nabla R_j+c_j'\left(\frac{d}{2}R_j+y_j\cdot\nabla R_j-\frac{i}{2}v_j\cdot y_jR_j\right)\right\}\\
&+\sum_{j=1}^K\left\{ia_j\partial_tR_j+b_j\cdot\nabla(\partial_tR_j)\right\}\\
&+\sum_{j=1}^Kc_j\left(\frac{d}{2}\partial_tR_j+y_j\cdot\nabla(\partial_tR_j)-\frac{i}{2}v_j\cdot y_j\partial_tR_j-v_j\cdot\nabla R_j+\frac{i}{2}|v_j|^2R_j\right),
\end{align*}
from (\ref{eg1}), (\ref{eg2}) (which still take the same form in the present context with the consideration of a third direction), and from the derivation with respect to $t$ of the third family of orthogonal conditions
$$\Rr\int_{\RR^d}\tilde{z}\left(\overline{\frac{d}{2}R_k+y_k\cdot\nabla R_k-\frac{i}{2}v_k\cdot y_kR_k}\right)\;dx=0,$$ which yields 
\begin{equation}
0=c_k'\left(\int_{\RR^d}\big(x\cdot\nabla Q_{\omega_k}\big)^2\;dx-\frac{d^2}{4}\int_{\RR^d}Q_{\omega_k}^2\;dx\right)+\mathrm{O}\big(\|z\|_{H^1}^2+e^{-\gamma t}\|z\|_{H^1}+\|\tilde{z}\|_{H^1}\big).
\end{equation}
We finish the proof of Lemma \ref{lem_modpar} with \eqref{non_zero} in Appendix.
\end{proof}

Now, to exploit Lemma \ref{estcrit} in order to perform an estimate of $\tilde{H}(t)$, we have to control the scalar products $\Rr\int_{\RR^d}\tilde{z}\overline{R_k}\;dx$, $k=1,\dots,K$. In fact, we state and prove next the analogue of Lemma \ref{derre}.

\begin{Lem}[Control of the $R_k$ directions]\label{derre_crit} We have
$$\left|\frac{d}{dt}\Rr\int_{\RR^d}\tilde{z}\overline{R_k}\;dx\right|\leq C\left(e^{-\gamma t}\|z\|_{H^1}+\|z\|_{H^1}^2\right).$$
\end{Lem}

\begin{proof}
Note that \eqref{prod_scal_tilde_z} is still guaranteed here by observing moreover that
\begin{align*}
\sum_{j=1}^Kc_j\Rr\int_{\RR^d}\left(\frac{d}{2}R_j+y_j\cdot\nabla R_j-\frac{i}{2}v_jR_j\right)\overline{R_k}\;dx&=\int_{\RR^d}\frac{d}{2}|R_k|^2\;dx+\Rr\int_{\RR^d}y_k\cdot\nabla R_k\overline{R_k}\;dx+\epsilon_2(t)\\
&=\epsilon_2(t),
\end{align*}
where $\epsilon_2'(t)=\mathrm{O}\left(e^{-\gamma t}\|z(t)\|_{H^1}\right)$.
Now, the rest of the proof of Lemma \ref{derre} is kept unchanged and thus we obtain Lemma \ref{derre_crit}.
\end{proof}

We deduce from Lemma \ref{lem_modpar} and from Lemma \ref{derre_crit} that 
 \begin{equation}\label{control_term_funct}
\left|c_k'\Rr\int_{\RR^d}\tilde{z}\overline{R_k}\;dx\right|\leq C\left(e^{-\gamma t}\|z\|_{H^1}+\|z\|_{H^1}^2+\|\tilde{z}\|_{H^1}\right)\int_t^{+\infty}\left(e^{-\gamma s}\|z\|_{H^1}+\|z\|_{H^1}^2\right)\;ds.
\end{equation}

\noindent Next, it follows from \eqref{control_term_funct} and Lemma \ref{estcrit} that 
\begin{equation}\label{est_H_crit}
\begin{aligned}
\MoveEqLeft[1]
\tilde{H}(t)\\
\le &\;C\int_t^{+\infty}\left\{\frac{1}{s}\|\tilde{z}(s)\|_{H^1}^2+\left(e^{-\gamma s}+\|z(s)\|_{H^1}\right)\|\tilde{z}(s)\|_{H^1}\|z(s)\|_{H^1}\right\}\;ds\\
&+C\int_t^{+\infty}\left(e^{-\gamma s}\|z(s)\|_{H^1}+\|z(s)\|_{H^1}^2+\|\tilde{z}(s)\|_{H^1}\right)\int_s^{+\infty}\left(e^{-\gamma u}\|z(u)\|_{H^1}+\|z(u)\|_{H^1}^2\right)\;du\;ds.
\end{aligned}
\end{equation}

\subsubsection{Conclusion of the proof of uniqueness in the critical case}

Proposition \ref{coerccritic} allows us to obtain the coercivity estimate in Proposition \ref{prop_coerc}) so that by \eqref{derre2} (which follows also from Lemma \ref{derre_crit}) and \eqref{est_H_crit}, we obtain
\begin{align*}
\MoveEqLeft[1]
\|\tilde{z}(t)\|_{H^1}^2\\
&\le \;C\int_t^{+\infty}\left\{\frac{1}{s}\|\tilde{z}(s)\|_{H^1}^2+\left(e^{-\gamma s}+\|z(s)\|_{H^1}\right)\|\tilde{z}(s)\|_{H^1}\|z(s)\|_{H^1}\right\}\;ds\\
&\:\:+C\int_t^{+\infty}\left(e^{-\gamma s}\|z(s)\|_{H^1}+\|z(s)\|_{H^1}^2+\|\tilde{z}(s)\|_{H^1}\right)\int_s^{+\infty}\left(e^{-\gamma u}\|z(u)\|_{H^1}+\|z(u)\|_{H^1}^2\right)\;du\;ds.
\end{align*}
On the other hand, Proposition \ref{controlz} and in fact \eqref{estz1}, \eqref{estz2}, and \eqref{estz} are still available here in the critical case; this is guaranteed by Lemma \ref{lem_modpar}. \\
Thus, adapting the proof of Lemma \ref{lem_tilde_z}, we deduce the following estimate in which $\tilde{z}$ is the only variable that appears: for $t$ large enough, and for some $\alpha>4$,
\begin{equation}
\|\tilde{z}(t)\|_{H^1}\le C\left[\int_t^{+\infty}\frac{1}{s}\|\tilde{z}(s)\|_{H^1}\;ds+\frac{1}{t^{\alpha-2}}\int_t^{+\infty}\int_s^{+\infty}\|\tilde{z}(u)\|_{H^1}\;du\;ds\right]
\end{equation}
(with $C$ independent of $z$ for the same reasons as those mentioned in Remark \ref{rq:indep_z}).

It results then $\tilde{z}(t)=0$ in the neighborhood of $+\infty$. Note that one requires here $\|z(t)\|_{H^1}=\mathrm{O}\left(\frac{1}{t^\alpha}\right)$, with $\alpha>4$, to hold.
Consequently, uniqueness of a multi-soliton associated with the $R_k$, $k=1,\dots,K$ in the sense of \eqref{classexpo} is proved also in the $L^2$-critical case.

\section*{Appendix}

\subsection*{Linear independence of $\left(\partial_{x_i}Q_\omega\right)_{i\in\{1,\dots,d\}}$}

\begin{Prop}\label{indep}
Let $u\in H^1(\RR^d)$ be such that there exists $n\in \RR^d\setminus\{0\}$ such that $$\forall\;x\in\RR^d,\qquad n\cdot\nabla u(x)=0.$$ Then $u=0$.
\end{Prop}

\begin{proof}
Even if it means completing $\left\{\frac{n}{|n|}\right\}$ in an orthonormal basis of $\RR^d$ and considering the passage matrix between the canonical basis and this new basis, we can always assume that $n$ is the last vector of the canonical basis of $\RR^d$. In that case, our assumption in Proposition \ref{indep} reads:
$$\forall\;(x_1,\dots,x_d)\in\RR^d,\qquad \partial_{x_d}u(x_1,\dots,x_d)=0,$$ or, in other words, for all $x_1,\dots,x_{d-1}\in\RR$, the application $x_d\mapsto u(x_1,\dots,x_{d-1},x_d)$ is constant, equal to $u(x_1,\dots,x_{d-1},0)$. \\
Since $u\in L^2(\RR^d)$, for all $x_1,\dots,x_{d-1}\in\RR$, one must have that $$\int_{\RR}|u(x_1,\dots,x_{d-1},x_d)|^2\;dx_{d}=\int_{\RR}|u(x_1,\dots,x_{d-1},0)|^2\;dx_{d}$$ is a finite quantity by Fubini theorem. This is the case if and only if $u(x_1,\dots,x_{d-1},0)=0$. Thus $u=0$.
\end{proof}

\subsection*{Proof of Lemma \ref{lem:phij}}

Assume that $1\leq k<j\le K$. \\
If $x_1<\xi_{j-1}+(\sigma_{j-1}-A_0)t$, we have $\phi_j(t,x)=0$ . \\
If  $x_1\ge \xi_{j-1}+(\sigma_{j-1}-A_0)t$, then $x_1>v_{k,1}t$ for large values of $t$ and thus, by \eqref{exp_decay_Q}, we have
\begin{align*}
|R_k(t,x)|&\le Ce^{-\frac{\sqrt{\omega_k}}{4}(x_1-v_{k,1}t)}e^{-\frac{\sqrt{\omega_k}}{4}|x-v_kt|} \le Ce^{-\frac{\sqrt{\omega_k}}{4}(\sigma_{j-1}-A_0-v_{k,1})t}e^{-\frac{\sqrt{\omega_k}}{4}|x-v_kt|}\\
&\le Ce^{-\frac{\sqrt{\omega_k}}{4}(\sigma_{j-1}-A_0-v_{j-1,1})t}e^{-\frac{\sqrt{\omega_k}}{4}|x-v_kt|} \le Ce^{-\frac{\sqrt{\omega_k}}{4}\left(\frac{v_{j,1}-v_{j-1,1}}{2}-A_0\right)t}e^{-\frac{\sqrt{\omega_k}}{4}|x-v_kt|}\\
&\le Ce^{-\gamma t}e^{-\frac{\sqrt{\omega_k}}{4}|x-v_kt|}.
\end{align*}
Assume now that $1\leq j<k\le K$. \\
If $x_1>\xi_{j}+(\sigma_{j}+A_0)t$, we have $\phi_j(t,x)=0$ . \\
If  $x_1\le \xi_{j}+(\sigma_{j}+A_0)t$, then $x_1<v_{k,1}t$ for large values of $t$ and thus we have as before
\begin{align*}
|R_k(t,x)|&\le Ce^{-\frac{\sqrt{\omega_k}}{4}(v_{k,1}t-x_1)}e^{-\frac{\sqrt{\omega_k}}{4}|x-v_kt|} \le Ce^{-\frac{\sqrt{\omega_k}}{4}(v_{k,1}-\sigma_j-A_0)t}e^{-\frac{\sqrt{\omega_k}}{4}|x-v_kt|}\\
&\le Ce^{-\frac{\sqrt{\omega_k}}{4}(v_{j+1,1}-\sigma_j-A_0)t}e^{-\frac{\sqrt{\omega_k}}{4}|x-v_kt|} \le Ce^{-\gamma t}e^{-\frac{\sqrt{\omega_k}}{4}|x-v_kt|}.
\end{align*}
Thus for all $k\neq j$,
\[ |R_k(t,x)\phi_j(t,x)|\le Ce^{-\gamma t}e^{-\frac{\sqrt{\omega_k}}{4}|x-v_kt|} \]
 and of course the same estimate is valid for $|\partial_{x_1}R_k(t,x)\phi_j(t,x)|$. This proves \eqref{prel1}.

In a similar way, one proves \eqref{prel2}. Now, let us show how to obtain \eqref{prel3}. First, notice that it is sufficient to prove \eqref{prel3} with $\psi_k$ instead of $\phi_k$. Then,
\[ \partial_{x_1}\psi_k(t,x)=\frac{1}{t}\psi '\left(\frac{x_1-\xi_k-\sigma_kt}{t}\right), \quad\partial_{x_1}^3\psi_k(t,x)=\frac{1}{t^3}\psi^{(3)}\left(\frac{x_1-\xi_k-\sigma_kt}{t}\right), \]
and 
\[ \partial_{t}\psi_k(t,x)=-\frac{x_1-\xi_k-\sigma_k t}{t^2}\psi '\left(\frac{x_1-\xi_k-\sigma_kt}{t}\right)-\frac{\sigma_k}{t}\psi '\left(\frac{x_1-\xi_k-\sigma_kt}{t}\right). \]
Hence,
\[ |\partial_{x_1}\psi_k(t,x)|\le\frac{1}{t}\|\psi '\|_{L^\infty}, \quad|\partial_{x_1}^{(3)}\psi_k(t,x)|\le\frac{1}{t^3}\|\psi^{(3)}\|_{L^\infty}, \]
and 
\[ |\partial_t\psi_k(t,x)|\le\frac{1}{t}\|x\psi '\|_{L^\infty}+\frac{\sigma_k}{t}\|\psi '\|_{L^\infty}, \]
which leads to \eqref{prel3}.

To finish with, let us observe that $\partial_{x_1}\phi_j(t,x)=0$ if $\xi_{j-1}+(\sigma_{j-1}+A_0)t\le x_1\le \xi_{j}+(\sigma_{j}-A_0)t$ (and in fact also if $x_1\le \xi_{j-1}+(\sigma_{j-1}-A_0)t$ or $x_1 \ge \xi_j+(\sigma_j+A_0)t$). Thus, for $k\neq j$, the proof of \eqref{prel4} is just a copy of that of \eqref{prel1}. \\
Moreover, if $x_1\le \xi_{j-1}+(\sigma_{j-1}+A_0)t$, then $x_1<v_{j,1}t$ for $t$ large and thus, as before, we obtain
\begin{align*}
|R_j(t,x)|&\le Ce^{-\frac{\sqrt{\omega_j}}{4}\left(v_jt-x_1\right)}e^{-\frac{\sqrt{\omega_k}}{4}|x-v_jt|} \le Ce^{-\frac{\sqrt{\omega_j}}{4}\left(v_jt-\sigma_{j-1}-A_0\right)t}e^{-\frac{\sqrt{\omega_k}}{4}|x-v_jt|}\\
&\le Ce^{-\gamma t}e^{-\frac{\sqrt{\omega_k}}{4}|x-v_jt|}.
\end{align*}
If $x_1\ge \xi_{j}+(\sigma_{j}-A_0)t$, then $x_1>v_{j,1}t$ and one obtains again 
$$|R_j(t,x)|\le Ce^{-\gamma t}e^{-\frac{\sqrt{\omega_k}}{4}|x-v_jt|}.$$
Hence, using in addition \eqref{prel3}, we deduce from what precedes that 
$$|R_k(t,x)\partial_{x_1}\phi_j(t,x)|\le C\frac{e^{-\gamma t}}{t}e^{-\frac{\sqrt{\omega_k}}{4}|x-v_jt|}.$$ In this manner, we obtain \eqref{prel4}.

\subsection*{Proof of Proposition \ref{coerccritic}: coercivity property in the $L^2$-critical case}

In the $L^2$-critical (pure power) case, we consider $f:r\mapsto r^{\frac{2}{d}}$ so that the linearized operators around $Q_\omega$ rewrite $L_{+,\omega}(v)=-\Delta v+\omega v-\left(1+\frac{4}{d}\right)Q_\omega^{\frac{4}{d}}v$ and $L_{-,\omega}(v)=-\Delta v+\omega v-Q_\omega^{\frac{4}{d}}v$. \\
Let us prove Proposition \ref{coerccritic}, following the results and ideas of Weinstein \cite{weinsteinmod}. \\ 

Due to Weinstein \cite[Proposition 2.7]{weinsteinmod},  $\inf_{v\in H^1(\RR^d), \langle Q_\omega,v\rangle=0}{ \langle L_{+,\omega}v,v\rangle}=0$ so that we can set $\tau:=\inf_{v\in S}{ \langle L_{+,\omega},v\rangle}$, where $S$ is the set of all $v\in H^1(\RR^d)$ such that $\langle Q_\omega,v\rangle=0$, for all $i={1,\dots,d}$, $\langle \partial_{x_i}Q_\omega,v\rangle=0$, $\langle x\cdot \nabla Q_\omega,v\rangle=0$, and $\|v\|_{H^1}=1$. We have obviously $\tau\ge 0$; we aim to show that $\tau>0$.\\

Assume by contradiction that $\tau=0$. Then for all $n\in\NN$, there exists $v_n\in S$ such that $\langle L_{+,\omega}v_n,v_n\rangle\le \frac{1}{n+1}$. \\
This implies
\begin{equation}\label{ineq_v}
0<\min\{\omega,1\}\|v_n\|_{H^1}^2 \le \int_{\RR^d}|\nabla v_n|^2\;dx+\omega\int_{\RR^d}|v_n|^2\;dx\le \left(1+\frac{4}{d}\right)\int_{\RR^d}Q_\omega^{\frac{4}{d}}v_n^2\;dx+\frac{1}{n+1}.
\end{equation}
In addition, $\left(\|v_n\|_{H^1}\right)_n$ is uniformly bounded so that, up to extraction, $(v_n)$ converges in $H^1(\RR^d)$ for the weak topology, say to $v^*\in H^1(\RR^d)$. And so, we have
\begin{equation}\label{ortho_v}
\langle Q_\omega,v^*\rangle=0,\quad \forall\; i\in\{1,\dots,d\},\:\langle \partial_{x_i}Q_\omega,v^*\rangle=0,\quad \langle x\cdot \nabla Q_\omega,v^*\rangle=0.
\end{equation}
By means of Hölder inequality, interpolation, and exponential decay of $Q_\omega$, (\ref{ortho_v}) leads to
\begin{equation}\label{convergence}
\int_{\RR^d}Q_\omega^{\frac{4}{d}}v_n^2\;dx\underset{n\to+\infty}{\longrightarrow} \int_{\RR^d}Q_\omega^{\frac{4}{d}}{v^*}^2\;dx.
\end{equation}
By passing to the limit as $n$ tends to $+\infty$, it results from $(\ref{ineq_v})$ and (\ref{convergence}) that $$0<\min\{\omega,1\}\leq \left(1+\frac{4}{d}\right)\int_{\RR^d}Q_\omega^{\frac{4}{d}}{v^*}^2\;dx.$$
In particular, $v^*\neq 0$. \\
Moreover, by weak convergence, we have 
\begin{equation}\label{liminf}
\|v^*\|_{L^2}\leq\liminf_{n\to+\infty}\|v_n\|_{L^2}\qquad \text{and}\qquad \|\nabla v^*\|_{L^2}\leq\liminf_{n\to+\infty}\|\nabla v_n\|_{L^2}.
\end{equation}
It follows from (\ref{convergence}) and (\ref{liminf}) that $$0\le \langle L_{+,\omega}v^*,v^*\rangle\leq \liminf_{n\to+\infty} \langle L_{+,\omega}v_n,v_n\rangle=0;$$ in other words, $\frac{v^*}{\|v^*\|_{H^1}}$ is an element of $S$ which minimizes $v\mapsto  \langle L_{+,\omega}v,v\rangle$.
Even if it means considering $\frac{v^*}{\|v^*\|_{H^1}}$, we will assume moreover $\|v^*\|_{H^1}=1$. We are thus led to the following Lagrange multiplier condition:
\begin{equation}\label{lagrange}
L_{+,\omega}v^*=\alpha v^*+\beta Q_\omega+\sum_{i=1}^d\gamma_i\partial_{x_i}Q_\omega+\delta x\cdot\nabla Q_\omega,
\end{equation} for some reals $\alpha,\beta,\gamma_i$, and $\delta$.
Since $0= \langle L_{+,\omega}v^*,v^*\rangle$, (\ref{ortho_v}) implies that $\alpha=0$. 
Then, for all $j\in\{1,\dots,d\}$, 
$$0=\langle L_{+,\omega}\partial_{x_j}Q_{\omega},v^*\rangle=\langle L_{+,\omega}v^*,\partial_{x_j}Q_\omega\rangle=\sum_{i=1}^d\gamma_i\langle\partial_{x_i}Q_\omega,\partial_{x_j}Q_\omega\rangle.$$ Given that $\partial_{x_1}Q_\omega,\dots,\partial_{x_d}Q_\omega$ are linearly independent in $L^2(\RR^d)$, the $d\times d$-matrix with entries $\langle \partial_{x_i}Q_\omega,\partial_{x_j}Q_\omega\rangle$ is invertible. Consequently, for all $i\in\{1,\dots,d\}$, $\gamma_i=0$. \\
Now, using $$0=\langle -2Q_\omega,v^*\rangle,\qquad L_{+,\omega}\left(\frac{d}{2}Q_\omega+x\cdot\nabla Q_\omega\right)=-2Q_\omega$$ (which is specific to the critical case), and the symmetry of the bilinear form $\langle L_{+,\omega}\cdot,\cdot\rangle$, we deduce that
\begin{align*}
0&=\left\langle L_{+,\omega}v^*,\frac{d}{2}Q_\omega+x\cdot \nabla Q_\omega\right\rangle =\frac{\beta d}{2}\|Q_\omega\|_{L^2}^2-\frac{\delta d^2}{4}-\frac{\beta d}{2}\|Q_\omega\|_{L^2}^2+\delta \int_{\RR^d}\left(x\cdot \nabla Q_\omega\right)^2\;dx\\
&=\delta\left(\int_{\RR^d}\left(x\cdot\nabla Q_\omega\right)^2\;dx-\frac{d^2}{4}\int_{\RR^d}Q_\omega^2\;dx\right).
\end{align*}
But we have
\begin{equation}\label{non_zero}
\int_{\RR^d}\big(x\cdot\nabla Q_\omega\big)^2\;dx-\frac{d^2}{4}\int_{\RR^d}Q_\omega^2\;dx>0,
\end{equation}
considering that this quantity is nothing but the square of the $L^2$ norm of $\frac{d}{2}Q_\omega+x\cdot \nabla Q_\omega$ (and $\frac{d}{2}Q_\omega+x\cdot \nabla Q_\omega$ is obviously not zero). \\
\noindent Hence $\delta=0$, and finally (\ref{lagrange}) reduces to $L_{+,\omega}v^*=  \beta Q_\omega$. \\
We claim now that $\beta\neq 0$: otherwise (using the well-known non-degeneracy condition (\ref{nondeg}) of $L_{+,\omega}$ in the present case) $v^*$ would be a linear combination of the $\partial_{x_i}Q_\omega$, $i=1,\dots,d$, and then it would result $v^*=0$ (since for all $i$, $\langle v^*,\partial_{x_i}Q_\omega\rangle=0$), which is not the case. \\
Thus $v^*=-\frac{2}{\beta}\left(\frac{d}{2}Q_\omega+x\cdot\nabla Q_\omega\right)$ and 
\begin{align*}
0&=\langle v^*,x\cdot\nabla Q_\omega\rangle =-\frac{2}{\beta}\left\langle \frac{d}{2}Q_\omega+x\cdot\nabla Q_\omega,x\cdot\nabla Q_\omega\right\rangle\\
&=-\frac{2}{\beta}\left(\int_{\RR^d}\left(x\cdot\nabla Q_\omega\right)^2\;dx-\frac{d^2}{4}\int_{\RR^d}Q_\omega^2\;dx\right), 
\end{align*} which contradicts (\ref{non_zero}).
So we come to the conclusion that $\tau$ is positive; hence Proposition \ref{coerccritic} is established.

\small
\bibliographystyle{plain}
\bibliography{references_NLS}

\end{document}